\documentclass[a4paper,11pt]{article}
\usepackage[margin=2.5cm]{geometry}

\usepackage[utf8]{inputenc}
\usepackage{cite}
\usepackage{amsmath}
\usepackage{amssymb}
\usepackage{amsthm}
\usepackage{hyperref} 
\usepackage{bm}
\usepackage{subfig}
\usepackage{enumitem} 
\usepackage{overpic}
\usepackage{algorithm}
\usepackage{algpseudocode}
\usepackage{hhline}
\usepackage{microtype}
\usepackage{tabto}
\usepackage[normalem]{ulem}

\usepackage{pgfplots}
\pgfplotsset{compat=newest}
\usepackage{tikz}

\renewcommand{\=}{\,=\,}

\renewcommand{\vec}{\bm}
\newcommand{\tensor}{\bm}

\newcommand{\dQ}{d_\mathrm{Q}}

\newcommand{\x}{x}

\newcommand{\stress}{\tensor{\sigma}}

\newcommand{\I}{\mathbf{I}}
\newcommand{\g}{\vec{g}}
\newcommand{\f}{\vec{f}}
\renewcommand{\u}{\vec{u}}

\renewcommand{\v}{\vec{v}}

\newcommand{\flux}{\vec{q}}

\newcommand{\n}{\vec{n}}
\newcommand{\q}{q}

\newcommand{\lambdavisco}{\lambda_\mathrm{v}}

\newcommand{\vangenuchtenalpha}{\alpha_\mathrm{vG}}
\newcommand{\vangenuchtenm}{{m_\mathrm{vG}}}
\newcommand{\vangenuchtenn}{{n_\mathrm{vG}}}

\newcommand{\NABLA}{\vec{\nabla}}
\newcommand{\DIV}{\NABLA\cdot}
\newcommand{\GRAD}{\NABLA}
\newcommand{\eps}[1]{\tensor{\varepsilon}\!\left(#1\right)\hspace{-0.025cm}}

\newcommand{\V}{\bm{V}}

\newcommand{\Q}{Q}

\newcommand{\Vh}{\bm{V}_h}
\newcommand{\Qh}{Q_h}

\newcommand{\VQh}{\Vh \times \Qh}

\newcommand{\qh}{q_h}
\newcommand{\uh}{\vec{u}_h}

\newcommand{\vh}{\vec{v}_h}

\newcommand{\chih}{\chi_h}

\newcommand{\uhtau}{\bar{\vec{u}}_{h\tau}}

\newcommand{\chihtau}{\bar{\chi}_{h\tau}}
\newcommand{\uhtauHat}{\hat{\vec{u}}_{h\tau}}
\newcommand{\vhtauHat}{\hat{\vec{u}}_{t,h\tau}}
\newcommand{\vhtauHatTest}{\hat{\vec{v}}_{h\tau}}
\newcommand{\chihtauHat}{\hat{\chi}_{h\tau}}
\newcommand{\porepressurehtauHat}{\hat{p}_{\mathrm{pore},h\tau}}

\newcommand{\uveta}{\vec{u}_{\varepsilon\eta}}
\newcommand{\chiveta}{\chi_{\varepsilon\eta}}
\newcommand{\uv}{\vec{u}_{\varepsilon\eta}}
\newcommand{\chiv}{\chi_{\varepsilon\eta}}
\newcommand{\ueta}{\vec{u}_{\eta}}
\newcommand{\chieta}{\chi_{\eta}}

\newcommand{\vhi}{\vec{v}_{h,i}}
\newcommand{\qhj}{q_{h,j}}

\newcommand{\absolutepermeability}{{\kappa}_\mathrm{abs}}
\newcommand{\absolutepermeabilitymin}{{\kappa}_\mathrm{m,abs}}
\newcommand{\absolutepermeabilitymax}{{\kappa}_\mathrm{M,abs}}
\newcommand{\relativepermeability}{{\kappa}_\mathrm{rel}}

\newcommand{\sw}{s_\mathrm{w}}

\newcommand{\pw}{p_\mathrm{w}}

\newcommand{\kirchhoff}{\chi}

\newcommand{\porepressure}{p_\mathrm{pore}}

\newcommand{\bk}{\hat{b}}
\newcommand{\Bk}{\hat{B}}
\newcommand{\bke}{\hat{b}_\eta}
\newcommand{\Bke}{\hat{B}_\eta}

\newcommand{\pwk}{\hat{p}_\mathrm{w}}
\newcommand{\swk}{\hat{s}_\mathrm{w}}
\newcommand{\pek}{\hat{p}_\mathrm{pore}}

\newcommand{\relativepermeabilityk}{\hat{\kappa}_\mathrm{rel}}
\newcommand{\rhow}{\rho_\mathrm{w}}

\newcommand{\fext}{\f_{\mathrm{ext}}}

\newcommand{\hext}{h_\mathrm{ext}}

\newcommand{\FlowDirichletBoundary}{\Gamma_\mathrm{D}^\mathrm{f}}
\newcommand{\FlowNeumannBoundary}{\Gamma_\mathrm{N}^\mathrm{f}}
\newcommand{\MechanicsDirichletBoundary}{\Gamma_\mathrm{D}^\mathrm{m}}
\newcommand{\MechanicsNeumannBoundary}{\Gamma_\mathrm{N}^\mathrm{m}}

\newcommand{\llangle}{\left\langle}
\newcommand{\rrangle}{\right\rangle}

\def\mathcolor#1#{\@mathcolor{#1}}
\def\@mathcolor#1#2#3{%
  \protect\leavevmode
  \begingroup
    \color#1{#2}#3%
  \endgroup
}

\newtheorem{theorem}{Theorem}

\newtheorem{lemma}[theorem]{Lemma}

\newtheorem{definition}[theorem]{Definition}
\newtheorem{remark}[theorem]{Remark}

\numberwithin{equation}{section}
\numberwithin{theorem}{section}

\newcommand{\htauStabilityUChi}{C^{(1)}}
\newcommand{\htauStabilityDchiDt}{C^{(2)}_{\zeta\eta}}
\newcommand{\htauStabilityLegendre}{C^{(3)}_{\zeta}}
\newcommand{\htauStabilityPpore}{C^{(4)}}
\newcommand{\htauStabilityDbDt}{C^{(5)}_{\zeta}}
\newcommand{\htauStabilityDuDt}{C^{(6)}_{\zeta\eta}}
\newcommand{\htauStabilityDuDtt}{C^{(7)}_{\zeta\eta}}

\newcommand{\epsStabilityDuDt}{C^{(8)}}
\newcommand{\epsStabilityLegendre}{C^{(9)}}
\newcommand{\epsStabilityDbDt}{C^{(10)}}
\newcommand{\epsStabilityDchiDt}{C^{(11)}_\eta}

\newcommand{\etaStabilityDchiDt}{C^{(12)}}

\newcommand{\COmegaPoincare}{C_{\Omega,\mathrm{P}}}
\newcommand{\COmegaDiscretePoincare}{C_{\Omega,\mathrm{DP}}}
\newcommand{\COmegaInfSup}{C_{\Omega,\mathrm{is}}}
\newcommand{\CDiscreteTrace}{C_{\mathrm{tr}}}
\newcommand{\COmegaPoincareKorn}{C_{\Omega,\mathrm{PK}}}

\author{Jakub W.\ Both\thanks{Department of Mathematics, University of Bergen, Bergen, Norway; $\{$\texttt{jakub.both@uib.no}$\}$} \and
       Iuliu Sorin Pop\thanks{Faculty of Science, Hasselt University, Hasselt, Belgium; and \newline \textcolor{white}{.}\hspace{0.75cm} Department of Mathematics, University of Bergen, Bergen, Norway; $\{$\texttt{sorin.pop@uhasselt.be}$\}$}  \and 
       Ivan Yotov\thanks{Department of Mathematics, University of Pittsburgh, Pittsburgh, PA, USA; $\{$\texttt{yotov@math.pitt.edu}$\}$}
}

\title{Global existence of a weak solution to unsaturated poroelasticity}

\begin{document}
  
\maketitle


\begin{abstract}
 In this paper, we consider unsaturated poroelasticity, i.e., coupled hydro-mechanical processes in unsaturated porous media, modelled by a non-linear extension of Biot's quasi-static consolidation model. The coupled, elliptic-parabolic system of partial differential equations is a simplified version of the general model for multi-phase flow in deformable porous media obtained under similar assumptions as usually considered for Richards' equation. In this work, the existence of a weak solution is established using regularization techniques, the Galerkin method, and compactness arguments.  The final result holds under non-degeneracy conditions and natural continuity properties for the non-linearities. The assumptions are demonstrated to be reasonable in view of geotechnical applications.
\end{abstract}


\section{Introduction}

Strongly coupled hydro-mechanical processes in porous media are occurring in various applications of societal relevance within, e.g., geotechnical, structural, and biomechanical engineering. Examples for instance are soil subsidence due to
groundwater withdrawal, geothermal energy storage in fractured rocks, swelling and drying shrinkage of concrete, and deformation of soft, biological tissue components.

In the field of porous media, such microscopically complex processes are typically modelled by a continuum mechanics approach~\cite{deBoer2000}. The multi-phasic solid-fluid mixture is considered a homogenized continuum,  and both geometry, skeleton, and fluid properties are averaged over representative elementary volumes, consisting of a mixture of solid and fluid particles. Ultimately, the microscopic interaction of the different constituents is described by macroscopic, effective equations. The simplest, macroscopic model accounting for the coupling of single-phase flow and elastic deformation in a porous medium is \textit{Biot's linear, quasi-static consolidation model}. Its phenomenological derivation dates back to the seminal works by Terzaghi~\cite{Terzaghi1936b} and Biot~\cite{Biot1941}. In the course of the last century, many more advanced models have been developed, accounting, e.g., for the presence of different interacting fluids, thermal effects, or chemical reactions. We refer to the textbooks~\cite{Lewis1998,Coussy2004} for an introduction and their derivation. 

In this paper, we consider a non-linear, coupled system of partial differential equations, modelling the quasi-static consolidation of variably saturated porous media, also called \textit{unsaturated poroelasticity} -- in particular relevant in soil mechanics. The model can be obtained by simplifying the more general model for two-phase flow in deformable porous media, founded on macroscopic momentum and mass balances combined with constitutive relations~\cite{Lewis1998}. It is assumed that one fluid phase can be simply neglected. This is a common practice for fluids with high viscosity ratios if the negligible fluid phase is continuous and connected to the atmosphere, i.e., the same hypotheses as for Richards' equation~\cite{Szymkiewicz2012,Nordbotten2011}. Finally, the resulting model generalizes Biot's quasi-static, linear consolidation model, combining Richards' equation and linear elasticity equations with non-linear coupling. It is highly non-linear, potentially strongly coupled, and potentially degenerate, which makes its analysis complicated.     
 
Regarding the mathematical theory of poroelasticity, in particular Biot's quasi-static, linear consolidation model has been well-studied. Well-posedness including the existence, uniqueness, and regularity of solutions, has been established~\cite{Auriault1977,Zenisek1984,Showalter2000}; recent advances in the numerical analysis include, e.g., stable finite  discretizations~\cite{Ferronato2010,Haga2012,Wheeler2014,Nordbotten2016,Rodrigo2016,White2016,Castelletto2017,Lee2017}, efficient numerical iterative solvers~\cite{Kim2011,Mikelic2013,Both2017,Gaspar2017,Borregales2019,Storvik2019}, and a posteriori error estimates~\cite{Kumar2018,Ahmed2019,Ahmed2020}.
Lately, linear and non-linear extensions have become of increased interest. Well-posedness and the efficient numerical solution have been analyzed for the dynamic Biot-Allard system~\cite{Mikelic2012}, Biot-Stokes systems~\cite{Showalter2005,Ambartsumyan2018,Ambartsumyan2018c}, the Biot model with deformation dependent permeability~\cite{Tavakoli2013,Bociu2016}, 
poroelasticity in fractured media~\cite{Mikelic2015,Girault2016,Berge2019,Ucar2018}, poroelasticity with non-linear solid and fluid compressibility~\cite{Both2019b,Borregales2018}, general non-linear single-phase poroelasticity~\cite{VanDuijn2019b}, poro-visco-elasticity~\cite{Both2019b,Bociu2016}, thermoporoelasticity~\cite{VanDujin2019,Brun2019,Brun2019b,Kim2018,Both2019b}, poroelasticity from a gradient flow perspective~\cite{Both2019b}, and multiple-permeability poroelasticity systems~\cite{Hong2019,Hong2018,Lee2019}, among others. In all problems, the coupling is linear. 

Despite the large interest, rather few theoretical results have been established for unsaturated or multi-phase poroelasticity. We highlight~\cite{Showalter2001}, in which the first ever mathematical analysis of the consolidation of a variably saturated, porous medium has been presented.
In the aforementioned work, the existence of a weak solution is established under two strict model assumptions: (i) the coupling term in the fluid flow equation is linear; and (ii) after introducing a new pressure variable by applying the Kirchhoff transformation the coupling and the diffusion terms in the mass balance simultaneously become linear. The second assumption implies a specific, artificial form of the so-called pore pressure, a non-linearity arising in the linear momentum balance. Ultimately, the result does not apply to the general model for unsaturated poroelasticity. On the other hand, the analysis accounts for non-linearly variable densities and porosities, and allows for degenerate situations. In addition, we mention efforts on studying the efficient numerical solution for unsaturated poroelasticity~\cite{Both2019} and multi-phase poroelasticity~\cite{Kim2013,Jha2014,Bui2019}. 
 
In this paper, the existence of weak solutions for the general model of unsaturated poroelasticity is established. In order to deal with the non-linear character, the problem is first transformed utilizing the Kirchhoff transformation, a technique commonly used for the analysis of non-linear diffusion problems~\cite{Alt1983}. By this, the diffusion component of the mass balance becomes linear -- a fully non-linear coupling and a non-linear storage coefficient are still present. The analysis then employs regularization techniques and compactness arguments in six steps and goes as follows. First, a physically motivated double regularization is introduced, adding a non-degenerate parabolic character to both balance equations. Regularization is required in order to allow the discussion of the non-linear coupling terms. Ultimately, the regularized model accounts for primary and secondary consolidation of variably saturated, porous media with compressible grains. Second, the problem is discretized combining an implicit time stepping, the finite element method (FEM) for the mechanics equation, and the finite volume method involving a two-point flux approximation (TPFA) for the flow equation. The motivation for the chosen discretization is two-fold: (i) it is a common discretization in the field of poroelasticity~\cite{Wheeler2014,Castelletto2019}, also closely related to mixed finite element discretizations~\cite{Ferronato2010}; moreover, finite volume methods~\cite{Eymard1999b,Klausen2008,Cances2016,Ait2018} and mixed finite element methods~\cite{Arbogast1996,Radu2008} are widely used for discretizing Richards' equation. Even more importantly, (ii) the specific choice of the discretization becomes crucial for the subsequent step of the proof, allowing for straightforward cancelling of the coupling terms. In the third step of the proof, stability of the discrete solution is showed, and compactness arguments are utilized for deriving a weak solution of the doubly regularized problem. For this, on the one hand the Legendre transformation is exploited as in~\cite{Alt1983} and specific finite volume techniques are employed for discussing the limit of the spatial discretization parameters, inspired by~\cite{Saad2013,Eymard1999}. Fourth, improved regularity is showed for the weak solution of the doubly regularized problem. Fifth and finally sixth, the limit of vanishing regularization in the momentum and mass balances are discussed, respectively. 

Difficulties arise in the last steps of the proof due to a possible degenerate character of the problem for vanishing saturation. Our analysis requires an overall parabolic character of the coupled problem and natural continuity properties for the non-linearities. Those are ensured under specific material assumptions and a non-vanishing, minimal amount of fluid saturation. In the appendix, the assumptions are demonstrated to be satisfied for constitutive relations typically utilized in real-life applications. Furthermore, for simplicity, the porous material is assumed to be isotropic, gravity has been neglected and homogeneous, essential boundary conditions have been considered. The focus of this work is on the involved, non-linear, coupled character of the governing equations.

The rest of the paper is organized as follows. In Section~\ref{section:model}, the model is introduced as derived in the engineering literature, and the model is transformed using the Kirchhoff transformation. In Section~\ref{section:main-result}, the notion of a weak solution to the transformed problem is introduced, and the main result is stated: existence of a weak solution to the transformed problem under certain model assumptions and non-degeneracy conditions. The idea of the proof, consisting of six steps, is presented. The details of those six steps are the subject of the remaining Sections~\ref{section:regularization}--\ref{section:eta-regularization}. In the appendix, the feasibility of the required assumptions for the main result are discussed for widely used constitutive models from the literature. In addition, technical results from the literature used in the proof of the main result are recalled for a comprehensive presentation.

\section{Mathematical model for unsaturated poroelasticity}\label{section:model}

We consider a continuum mechanics model for unsaturated poroelasticity, a particular simplification of general multi-phase poroelasticity~\cite{Coussy2004,Lewis1998}. It is based on the fundamental principles of momentum and mass balance combined with constitutive relations. The model is valid under the assumptions of infinitesimal strains and the presence of two fluid phases, an active and a passive phase; the displacement of the passive phase does not impede the advance of the active phase and can be therefore neglected. Finally, the model couples non-linearly the Richards equation and the linear elasticity equations utilizing an effective stress approach.

In the following, we recall the mathematical model employing the mechanical displacement and fluid pressure as primary variables. Additionally, the problem is transformed by the Kirchhoff transformation, a standard tool for the analysis of non-linear diffusion problems, cf., e.g.,~\cite{Alt1983}. The latter will be subject of the subsequent analysis.

\subsection{The original formulation} 
We consider a poroelastic medium occupying the open, connected, and bounded domain $\Omega\subset\mathbb{R}^d$, $d\in\{1,2,3\}$. Let $T>0$ denote the final time and $(0,T)$ denote the time interval of interest. 
Let $Q_T:=\Omega \times (0,T)$ denote the space-time domain.

The balance equations as derived in~\cite{Lewis1998} (note, we use an arbitrary pore pressure, whereas the specific \textit{average pore pressure} has been used in the aforementioned work) reads on $Q_T$:
\begin{align} 
\label{model:original:u}
 - \DIV \left[ 2\mu \eps{\u} + \lambda \DIV \u \I - \alpha \porepressure(\pw) \I \right] &= \f, \\
\label{model:original:p}
  \phi \partial_t \sw(\pw) + \phi c_\mathrm{w} \partial_t \pw + \frac{1}{N} \sw(\pw) \partial_t \porepressure(\pw)  + \alpha \sw(\pw) \partial_t \DIV \u  + \DIV \flux &= h,
\end{align}
where $\u$ is the mechanical displacement and $\pw$ is the fluid pressure (of the active phase). Furthermore, $\flux$ is the volumetric flux described by the generalized Darcy law
\begin{align}
 \label{model:original:q}
 \flux = -\absolutepermeability \, \relativepermeability(\sw(\pw)) \left(\GRAD \pw - \rhow \g \right).
\end{align}
Constitutive laws are given for the pore pressure $\porepressure$, the fluid saturation $\sw$ and the relative permeability $\relativepermeability$; the latter two are assumed to be homogeneous, i.e., they do not vary explicitly in space. Furthermore, $\f$ and $h$ are external load and source terms; $\mu,\lambda$ are the Lam\'e parameters; $\alpha\in[0,1]$ is the Biot constant; $c_\mathrm{w}\in[0,\infty)$ is the storage coefficient associated to fluid compressibility; $N\in(0,\infty]$ is the Biot modulus associated to the compressibility of solid grains; $\absolutepermeability$ is the absolute permeability; $\rhow$ is a reference fluid density and $\g$ is the gravitational acceleration. Finally, $\phi$ is the porosity. Under the hypothesis of small perturbations of the porosity~\cite{Coussy2004}, often applied along with the assumptions of linear elasticity, we can assume that the porosity $\phi$ acting as weight is constant in time, equal to some reference porosity field $\phi_0$. 

From now on, we consider a compact form of~\eqref{model:original:u}--\eqref{model:original:q}. Specifically, we seek $(\u,\pw)$ such that on $Q_T$
\begin{align}
-\DIV \left[2\mu\eps{\u} + \lambda \DIV \u \I - \alpha \porepressure(\pw) \I \right] 		&= \f, \label{model:governing:u} \\
  \partial_t b(\pw) + \alpha \sw(\pw) \partial_t  \DIV\u - \DIV  \left(\absolutepermeability \relativepermeability (\sw(\pw)) \left( \GRAD \pw - \rhow \g \right)	\right)	&= h,\label{model:governing:p}
\end{align}
where the function $b$ is defined as
\begin{align}
\label{model:example-b}
 b(\pw) = \phi_0 \sw(\pw) + c_\mathrm{w} \phi_0 \int_0^{\pw} \sw(p) \, dp  + \frac{1}{N} \int_0^{\pw} s_\mathrm{w}(p) \porepressure'(p) \, dp.
\end{align}
We note that the subsequent analysis is not dependent on specific choices for $b$, $\sw$, $\porepressure$ and~$\relativepermeability$.

In order to close the system~\eqref{model:governing:u}--\eqref{model:governing:p}, we impose: boundary conditions
\begin{align}
  \u &= \u_\mathrm{D}												&&\text{on }\MechanicsDirichletBoundary\times(0,T), 
  \label{model:bc:mechanics-dirichlet}\\
  \left(2\mu \eps{\u} + \lambda \DIV \u \I  - \alpha \porepressure(\pw) \I \right)\n &= \stress_\mathrm{N}			&&\text{on }\MechanicsNeumannBoundary\times(0,T), 
  \label{model:bc:mechanics-neumann}\\
  \pw&=p_\mathrm{w,D}	&&\text{on }\FlowDirichletBoundary\times(0,T), \label{model:bc:fluid-dirichlet}\\
  -\absolutepermeability\,\relativepermeability(\sw\left(\pw\right))\left( \GRAD \pw -  \rhow \g \right)\cdot \n &= q_N	&&\text{on }\FlowNeumannBoundary\times(0,T), \label{model:bc:fluid-neumann}
\end{align}
for the partitions $\{\MechanicsDirichletBoundary,\MechanicsNeumannBoundary\}$ and $\{\FlowDirichletBoundary,\FlowNeumannBoundary\}$ of the boundary $\partial\Omega$, where $\MechanicsDirichletBoundary$ and $\FlowDirichletBoundary$ have positive measure; as well as initial conditions
\begin{align}
  \u  & =\u_\mathrm{0} 		&&\text{in }\Omega\times\{0\}, \label{model:initial:mechanics}\\
  \pw &= p_\mathrm{w,0}, 	&&\text{in }\Omega\times\{0\}. \label{model:initial:flow}  
\end{align}
Putting the focus on the non-linear and coupled character of the balance equations, in the subsequent, mathematical analysis, we consider a simplified setting. We neglect gravity and non-homogeneous, essential boundary conditions, which in particular simplifies notation.

\subsection{The mathematical model under the Kirchhoff transformation}\label{section:kirchhoff-transformation}
The Kirchhoff transformation defines a new pressure-like variable
\begin{align}
\label{definition:kirchhoff-pressure}
 \kirchhoff(\pw) = \int_0^{\pw} \relativepermeability(\sw(\tilde{p}))\,d\tilde{p}.
\end{align} 
Assuming the constitutive laws satisfy $\relativepermeability(\sw(p))>0$, for all $p\in\mathbb{R}$,~\eqref{definition:kirchhoff-pressure} can be inverted. We redefine all functions in $\pw$ as functions in $\kirchhoff$
\begin{align}\label{definition:transformed-functions}
 \pwk := \chi^{-1},\quad
 \bk := b \circ \chi^{-1} ,\quad \swk:=\sw \circ \chi^{-1},\quad \pek := \porepressure \circ \chi^{-1},\quad\relativepermeabilityk := \relativepermeability \circ \swk.
\end{align}
Then under the assumption of a homogeneous relative permeability and saturation, the non-linear Biot equations~\eqref{model:governing:u}--\eqref{model:governing:p} reduces to finding $(\u,\kirchhoff)$, satisfying
\begin{align}
-\DIV \left( 2\mu\eps{\u} + \lambda \DIV \u \I - \alpha \pek(\kirchhoff) \I \right) 		&= \f, \label{model:governing:kirchhoff:u} \\
\label{model:governing:kirchhoff:p}
  \partial_t \bk(\kirchhoff) + \alpha \swk(\kirchhoff) \partial_t  \DIV\u - \DIV \left( \absolutepermeability \GRAD \chi  \right) 		&= h, 
\end{align}
on $Q_T$, and subject to the adapted boundary conditions 
\begin{align}
  \u &= \bm{0}												&&\text{on }\MechanicsDirichletBoundary\times(0,T), 
  \label{model:bc:kirchhoff:mechanics-dirichlet}\\
  \left(2\mu \eps{\u} + \lambda \DIV \u \I  - \alpha \pek(\kirchhoff) \I \right)\n &= \stress_\mathrm{N}			&&\text{on }\MechanicsNeumannBoundary\times(0,T), 
  \label{model:bc:kirchhoff:mechanics-neumann} \\
  \kirchhoff&=0 										&&\text{on }\FlowDirichletBoundary\times(0,T), \label{model:bc:kirchhoff:fluid-dirichlet}\\
  - \absolutepermeability\GRAD \kirchhoff \cdot \n &= w_N	&&\text{on }\FlowNeumannBoundary\times(0,T), \label{model:bc:kirchhoff:fluid-neumann}
\end{align}
and the initial conditions
\begin{align}
  \u& =\u_\mathrm{0} 		            &&\text{in }\Omega\times\{0\}, \label{model:initial:kirchhoff:mechanics}\\
  \kirchhoff&=\kirchhoff_\mathrm{0}, 	&&\text{in }\Omega\times\{0\}. \label{model:initial:kirchhoff:flow}
\end{align}

\section{Main result -- existence of a weak solution for the unsaturated poroelasticity model}\label{section:main-result}
The main result of this work is the existence result of a weak solution for the unsaturated poroelasticity model under the Kirchhoff transformation, cf.\ Section~\ref{section:kirchhoff-transformation}. In this section, we state the main result. This includes the notion of a weak solution, required assumptions and the idea of the proof. The details of the proof are the subject of the remainder of this paper.

\subsection{Definition of a weak solution}

Let $Q_T:=\Omega \times (0,T)$ denote the space-time domain. We use the standard notation for $L^p$, Sobolev and Bochner spaces, together with their inherent norms and scalar products. Let $\llangle\cdot,\cdot\rrangle$ denote the standard $L^2(\Omega)$ scalar product for scalars, vectors and tensors. For shorter notation, we use $\| \cdot \| := \| \cdot \|_{L^2(\Omega)}$. Let
\begin{align*}
  \V 	&= \left\{ \v \in H^1(\Omega)^d \, \left| \, {\v_|}_{\MechanicsDirichletBoundary} = \bm{0} \right.\right\}, \\
  \Q 	&= \left\{ q \in H^1(\Omega) \, \left| \, {q_|}_{\FlowDirichletBoundary} = 0 \right.\right\},
\end{align*}
denote the function spaces corresponding to mechanical displacement and fluid pressure, respectively, incorporating essential boundary conditions. We abbreviate the bilinear form associated to linear elasticity
\begin{align*}
 a(\u,\v) &= 2\mu \int_\Omega \eps{\u} : \eps{\u} \, dx + \lambda \int_\Omega \DIV \u \, \DIV \v \, dx, \quad \u,\v\in\V,
\end{align*}
and define $\|\cdot \|_{\V}:=a(\cdot,\cdot)^{1/2}$, which induces a norm on $\V$ due to Korn's inequality.
Moreover, we combine the external body and surface sources as elements in $\V^\star$ and $\Q^\star$, the duals of $\V$ and $\Q$, respectively. Let $\fext=(\f,\stress_\mathrm{N})$ and $\hext=(h,w_\mathrm{N})$ be defined by
 \begin{alignat*}{2}
  \llangle \fext, \v \rrangle  &= \int_\Omega \f \cdot \v \, dx + \int_{\MechanicsNeumannBoundary} \stress_\mathrm{N} \cdot \v \, ds,&\quad &\v\in \V,\\
  \llangle \hext, q \rrangle  &= \int_\Omega h \, q \, dx + \int_{\FlowNeumannBoundary} w_\mathrm{N} \, q \, ds, && q\in\Q.
 \end{alignat*}

\begin{definition}[Weak solution of the unsaturated poroelasticity model]\label{definition:weak-solution}
A weak solution to~\eqref{model:governing:kirchhoff:u}--\eqref{model:initial:kirchhoff:flow} is a pair $(\u,\chi)\in L^2(0,T;\V) \times L^2(0,T;Q)$ satisfying the following:
\begin{enumerate}
 
 \item[$\mathrm{(W1)}$] $\pek(\chi)\in L^2(Q_T)$, $\swk(\chi)\in L^\infty(Q_T)$.
 
 \item[$\mathrm{(W2)}$] $\bk(\chi)\in L^\infty(0,T;L^1(\Omega))$ and $\partial_t \bk(\chi)\in L^2(0,T; Q^\star)$ such that
 \begin{align*}
  \int_0^T \llangle \partial_t \bk(\chi), \q \rrangle \, dt
  +
  \int_0^T \llangle \bk(\chi) - \bk(\chi_0), \partial_t \q \rrangle \, dt
  =
  0,
 \end{align*}
 for all $\q\in L^2(0,T; Q)$ with $\partial_t \q \in L^1(0,T;L^\infty(\Omega))$ and $\q(T)=0$.
 
 \item[$\mathrm{(W3)}$] $\partial_t \DIV \u \in L^2(Q_T)$ such that 
 \begin{align*}
  \int_0^T \llangle \partial_t \DIV \u, q \rrangle \, dt + \int_0^T \llangle \DIV \u - \DIV \u_0, \partial_t q \rrangle \, dt = 0,
 \end{align*}
 for all $q\in H^1(0,T;L^2(\Omega))$ with $q(T)=0$.
 
 \item[$\mathrm{(W4)}$] $(\u,\chi)$ satisfies the variational equations
 \begin{align}\label{model:weak:kirchhoff:u} 
 \int_0^T  \left[ 
 a(\u, \v)
 - \alpha \llangle \pek(\chi), \DIV \v \rrangle \right]& \, dt = \int_0^T \llangle \fext, \v \rrangle \, dt, \\
 \int_0^T \bigg[
 \llangle \partial_t \bk(\chi), \q \rrangle + \alpha \llangle \swk(\chi) \partial_t \DIV \u, \q \rrangle 
 + \llangle \absolutepermeability \GRAD \chi, \GRAD \q \rrangle
 \bigg]& \, dt = \int_0^T \llangle \hext, \q \rrangle\, dt,   \label{model:weak:kirchhoff:p} 
 \end{align}
 for all $(\v,\q)\in L^2(0,T;\V) \times L^2(0,T;Q)$.
\end{enumerate}
\end{definition}

We note that the weak formulation of the initial conditions (W3) of the mechanical displacement immediately allow for a stronger formulation. See Lemma~\ref{lemma:eta:initial-conditions-u} for more information.

\subsection{Assumptions on model and data}\label{section:assumptions}
For proving the existence of a weak solution, we require several assumptions on the model, including the constitutive laws, model parameters, source terms and initial conditions:
\begin{itemize}
 
 \item[(A0)] $\sw:\mathbb{R}\rightarrow[0,1]$ and $\relativepermeability:[0,1]\rightarrow [0,1]$ such that $\relativepermeability(\sw(p))>0$, for all $p\in\mathbb{R}$ allowing for defining $\pwk$, $\bk$, $\swk$, $\pek$, and $\relativepermeabilityk$ as in~\eqref{definition:transformed-functions}.

 \item[(A1)] $\bk: \mathbb{R} \rightarrow \mathbb{R}$ is continuous and non-decreasing, and it holds that $\bk(0)=0$.
 
 \item[(A2)] $\swk : \mathbb{R} \rightarrow (0,1]$ continuous and differentiable a.e., and $\swk(\chi)=1$ for $\chi\geq 0$.
 
 \item[(A3)] $\pek : \mathbb{R} \rightarrow \mathbb{R}$ is continuously differentiable, non-decreasing, and it holds that $\pek(0)=0$.
 
 \item[(A4)] $\tfrac{\pek}{\swk}:\mathbb{R} \rightarrow \mathbb{R}$ is invertible and uniformly increasing, i.e., there exists a constant
 $c_{{\pek}/{\swk}}>0$ satisfying $\left(\tfrac{\pek}{\swk}\right)'(x)\geq c_{{\pek}/{\swk}}$ for all $x\in\mathbb{R}$. 
 
\end{itemize}
Assumptions~(A0)--(A4) are valid for standard constitutive laws, cf.\ Appendix~\ref{appendix:feasibility-assumptions}. The assumptions on the model parameters read:
\begin{itemize}
 \item[(A5)] $\mu>0$, $\lambda\geq 0$, $\alpha\geq 0$ are constant, and define the bulk modulus $K_\mathrm{dr}:=\tfrac{2\mu}{d} + \lambda$.

 \item[(A6)] $\absolutepermeability$ is uniformly bounded from below and above, such that there exist constants $0<\absolutepermeabilitymin\leq\absolutepermeabilitymax<\infty$ with $\absolutepermeability \in [\absolutepermeabilitymin,\absolutepermeabilitymax]$ on $\Omega$.
\end{itemize}
We note,~(A5) is stated only for simplicity. The assumptions on the external load and source terms read:
\begin{itemize}
 \item[(A7)] $\fext\in H^1(0,T;\V^\star) \cap C(0,T;\V^\star)$ and $\hext\in H^1(0,T;Q^\star)\cap C(0,T;\Q^\star)$, where
 \begin{alignat*}{2}
  \| \fext \|_{L^p(0,T;\V^\star)}^2 &:= \| \f \|_{L^p(0,T;\V^\star)}^2 + \| \stress_\mathrm{N} \|_{L^p(0,T;\V^\star)}^2, &\ & p\in\{2,\infty\},\\
  \| \hext \|_{L^p(0,T;\Q^\star)}^2 &:= \| h \|_{L^p(0,T;L^2(\Omega))}^2 + \| w_\mathrm{N} \|_{L^p(0,T;L^2(\FlowNeumannBoundary))}^2, &\ & p\in\{2,\infty\},
 \end{alignat*}
 and analogously $\|\fext\|_{\V^\star}$, $\| \partial_t \fext \|_{L^2(0,T;\V^\star)}$, $\| \fext \|_{H^1(0,T;\V^\star)}$, and $\|\hext\|_{\Q^\star}$, $\| \partial_t \hext \|_{L^2(0,T;\Q^\star)}$, $\| \hext \|_{H^1(0,T;\Q^\star)}$.
\end{itemize} 
 
The assumptions on the initial data read:
\begin{itemize}
 \item[(A8)] The initial data $(\u_0,\chi_0)\in\V\times \Q$ is sufficient regular such that there exists a constant $C_0$ satisfying 
 \begin{align*}
  &\| \u_{0} \|_{\V}^2 
  + \left\| \GRAD \chi_0 \right\|^2
  +
  \left\| \bk\left(\chi_0\right) \right\|_{L^1(\Omega)}
  + \left\| \Bk \left( \chi_{0}) \right) \right\|_{L^1(\Omega)} \\
  &\qquad 
  + \left\| \bar{B}\left(\frac{\pek(\chi_{0})}{\swk(\chi_{0})}\right) \right\|_{L^1(\Omega)} 
  + \left\| \pek(\chi_0) \right\|^2
  \leq C_0,
 \end{align*}
 where $\Bk$ and $\bar{B}$ are the Legendre transformations of $\bk$ and $\bar{b}:=\bk \circ \left(\frac{\pek}{\swk}\right)^{-1}$, respectively:
 \begin{align}
 \label{legendre-transform-bk-original}
  \Bk(z) &:= \int_0^z (\bk(z) - \bk(s))\, ds \geq 0,\\
  \label{legendre-transform-b-bar-original}
  \bar{B}(z) &:= \int_0^z (\bar{b}(z) - \bar{b}(s))\, ds \geq 0.
 \end{align}
 
 \item[(A9)] The initial data $(\u_0,\chi_0)$ satisfies the compatibility condition: $\pek(\chi_0)\in Q$ and 
 \begin{align*}
  a(\u_0, \v) - \alpha \llangle \pek(\chi_0), \DIV \v \rrangle = \llangle \fext(0), \v \rrangle,\quad \text{for all } \v\in\V,
 \end{align*}
 i.e., the mechanics equation at initial time. 
 \end{itemize}
 
 \noindent 
 Additionally, the following non-degeneracy conditions are required:
 
 \begin{enumerate}
 \item[(ND1)] There exists a constant $C_\mathrm{ND,1}>0$ such that
 \begin{align*}
  \left| \frac{\pek(\chi)}{\swk(\chi)\chi} \right| \leq C_\mathrm{ND,1},\quad \text{for all } \chi\in\mathbb{R}.
 \end{align*}

 \item[(ND2)] There exists a constant $C_\mathrm{ND,2}>0$ such that
 \begin{align*}
  C_\mathrm{ND,2}^{-1}\leq \pek'(\chi) &\leq C_\mathrm{ND,2}, \quad \text{for all } \chi\in\mathbb{R}.
 \end{align*}
 
 \item[(ND3)] There exists a constant $C_\mathrm{ND,3}\in(0,1)$ such that
 \begin{align*}
  K_\mathrm{dr} - \frac{\alpha^2}{4} \left( \frac{\swk(\chi)}{\pek'(\chi)} - 1 \right)^2 \frac{\left(\pek'(\chi)\right)^2}{\bk'(\chi)} \geq C_\mathrm{ND,3} K_\mathrm{dr}, \quad \text{for all } \chi\in\mathbb{R}.
 \end{align*}
\end{enumerate}

In Appendix~\ref{appendix:feasibility-assumptions}, it is demonstrated that for the van Genuchten model for $\sw$ and $\relativepermeability$~\cite{vanGenuchten1980}, and the equivalent pore pressure model for $\porepressure$~\cite{Coussy2004}, (ND1) and (ND2) follow if the saturation takes values above a residual saturation. Thus, (ND1) and (ND2) may be implicitly satisfied assuming (ND3) holds true. Furthermore, the calculations in Appendix~\ref{appendix:feasibility-assumptions} illustrate that for materials typically present in geotechnical application, the condition (ND3) is satisfied in saturation regimes above 1 to 10 percent (depending on the material parameters). Thereby, the practical saturation regime is covered for a wide range of applications.
After all,~(ND3) is the most restrictive assumption of all assumptions. It essentially requires the mechanical system to be sufficiently stiff in relation to the saturation profile. The lower the minimal saturation value, the stiffer the system has to be.

\subsection{Existence of solutions for the unsaturated poroelasticity model}

This section is presenting the main result together with the main steps of the proof.

\begin{theorem}[Existence of a weak solution to the unsaturated poroelasticity model]\label{theorem:existence}
 Under the model assumptions~$\mathrm{(A0)}$--$\mathrm{(A9)}$ and the non-degeneracy conditions~$\mathrm{(ND1)}$--$\mathrm{(ND3)}$, there exists a weak solution of~\eqref{model:governing:kirchhoff:u}--\eqref{model:initial:kirchhoff:flow} in the sense of Definition~\ref{definition:weak-solution}.
\end{theorem}

The main idea of the proof of Theorem~\ref{theorem:existence} is to use the Galerkin method in combination with compactness arguments. The main difficulty here is the control over the non-linear coupling terms. For this a regularization approach is used. After all, the proof consists of six steps. In the following, we present the idea of each step. Details are subject of the remainder of the article and will be presented in the six, subsequent sections. 

\paragraph{Step 1: Double physical regularization.}
 
Applying the Galerkin method along with compactness arguments for the original problem~\eqref{model:weak:kirchhoff:u}--\eqref{model:weak:kirchhoff:p} is challenging due to the coupling terms. A simple way to control the term $\partial_t \DIV \u$ is to add a suitable regularization term in the mechanics equation~\eqref{model:weak:kirchhoff:u}. As the coupling terms also involve non-linearities in the Kirchhoff pressure, strong compactness is required. Therefore, we add a coercive term in the flow equation, which allows for controlling the term $\partial_t \chi$. In this way, one can control the coupling terms, and eventually leading to convergence.

From a physical point of view, the regularized model accounts for secondary consolidation and compressible solid grains. In mathematical terms, it reads as follows. For given regularization parameters  $\zeta,\eta>0$, find $(\uveta,\chiveta)$ to be the solution to the variational equations
 \begin{align}\label{idea:model:regularized:weak:kirchhoff:u} 
 \int_0^T  \Big[ \zeta a(\partial_t \uveta, \v)
 + a(\uveta, \v)
 - \alpha \llangle \pek(\chiveta), \DIV \v \rrangle \Big]& \, dt = \int_0^T \llangle \fext, \v \rrangle \, dt, \\
 \int_0^T \bigg[
 \llangle \partial_t \bke(\chiveta), \q \rrangle + \alpha \llangle \swk(\chiveta) \partial_t \DIV \uveta, \q \rrangle 
 + \llangle \absolutepermeability \GRAD \chiveta, \GRAD \q \rrangle
 \bigg]& \, dt = \int_0^T \llangle \hext, \q \rrangle\, dt,   \label{idea:model:regularized:weak:kirchhoff:p}
 \end{align}
 for all $(\v,\q)\in L^2(0,T;\V) \times L^2(0,T;Q)$, where $\bke$ is a strictly increasing regularization of $\bk$ (see (A1$^\star$) for further properties). The next two steps prove that the regularized problem has a weak solution in an analogous sense to Definition~\ref{definition:weak-solution}.

\paragraph{Step 2: Discretization in space and time.}

We employ the implicit Euler scheme and a Galerkin method based on an inf-sup stable finite element/finite volume method to obtain a fully discrete counterpart of~\eqref{idea:model:regularized:weak:kirchhoff:u}--\eqref{idea:model:regularized:weak:kirchhoff:p}. In particular, the pressure variables are discretized by piecewise constant elements, and for the diffusion term a discrete gradient $\GRAD_h$ is employed corresponding to a two-point flux approximation of the volumetric fluxes~\cite{Eymard1999,Eymard2000}. 

Given an admissible mesh $\mathcal{T}$, cf.\ Definition~\ref{definition:admissible-mesh}, the conforming and non-conforming, discrete spaces $\Vh\subset \V$ and $\Qh\not\subset \Q$, respectively, and a partition $\{t_n\}_n$ of the interval $(0,T)$, the discretization for time steps $n$ reads: given the solution at the previous time step $(\uh^{n-1},\chih^{n-1}) \in \Vh\times\Qh$, find $(\uh^n,\chih^n)\in\Vh\times\Qh$ satisfying for all $(\vh,\qh)\in\Vh\times\Qh$
\begin{align}
   \zeta\tau^{-1}a (\uh^n - \uh^{n-1},\vh) 
   +  a(\uh^n,\vh) - \alpha \langle \pek(\chih^n), \DIV \vh \rangle &\= \langle \fext^n, \vh \rangle, \label{idea:ht:kirchhoff-reduced:u}\\
   \label{idea:ht:kirchhoff-reduced:p}
  \langle \bke(\chih^n) - \bke(\chih^{n-1}), \qh \rangle + \alpha \langle \swk(\chih^n) \DIV(\uh^n - \uh^{n-1}), \qh\rangle 
  + \, \tau \langle \GRAD_h \chih^n, \GRAD_h \qh \rangle_{\absolutepermeability} &\= \tau \langle \hext^n, \qh \rangle.
\end{align}

The reason for this particular choice of a discretization is two-fold: (i) the piecewise constant approximation of the pressure allows for the simple handling of non-linearities; (ii) the discrete gradients $\GRAD_h$ retain the local character of the differential operator. This together allows for simultaneously cancelling the coupling terms and utilizing the coercivity of the diffusion term. This is required, e.g., for proving the existence of a discrete solution employing a corollary of Brouwer's fixed point theorem, or in Step~3.
 
\paragraph{Step 3: Existence of a weak solution to the regularized model.}

Based on the discrete values $\left\{\left(\uh^n,\chih^n\right)\right\}_n$, we define suitable interpolations in time, $(\u_{h\tau},\chi_{h\tau})$, yielding approximations of $(\uveta,\chiveta)$. We remark that various interpolations are in fact introduced in the course of step~3 and~4. To avoid an excess in notations and for the ease of the presentation, we use the same notation, $(\u_{h\tau},\chi_{h\tau})$, for all interpolations.

The goal is to show convergence (in a certain sense) of $\left\{(\u_{h\tau},\chi_{h\tau})\right\}_{h,\tau}$ along a monotonically decreasing sequence of pairs $(h,\tau)\rightarrow (0,0)$ (from now on denoted $h,\tau\rightarrow 0$) towards a solution of ~\eqref{idea:model:regularized:weak:kirchhoff:u}--\eqref{idea:model:regularized:weak:kirchhoff:p}. This is achieved using compactness arguments; however, given the coupled and non-linear nature of~\eqref{idea:model:regularized:weak:kirchhoff:u}--\eqref{idea:model:regularized:weak:kirchhoff:p}, several terms require careful discussion: 
\begin{itemize}
 \item Non-linearities as $\pek(\chiveta)$ or products of independent variables as $\swk(\chiveta)\partial_t \DIV \uveta$ require partially strong convergence.
 
 \item Since $\bke$ is not necessarily Lipschitz continuous, it is not sufficient to show uniform stability for $\{\partial_t \chi_{h\tau} \}$ to conclude weak convergence of $\{ \partial_t \bke(\chi_{h\tau}) \}_{h,\tau}$ towards $\partial_t \bke(\chiveta)$. Instead, we apply techniques by~\cite{Alt1983} utilizing the Legendre transformation, $\Bke$, of $\bke$, analogously defined as in~\eqref{legendre-transform-bk-original}.
 
 \item The diffusion term is discretized using discrete gradients. Thus, weak convergence $\GRAD_h \chi_{h\tau} \rightarrow \GRAD \chiveta$ is not an obvious consequence of uniform stability for $\{\GRAD_h \chi_{h,\tau}\}_{h,\tau}$. For this, we apply techniques from finite volume literature~\cite{Eymard1999,Saad2013}.
 
\end{itemize}
Motivated by that, we first derive stability estimates that are uniform wrt.\ the discretization parameters
\begin{align*}
 &\left\| \u_{h\tau} \right\|_{H^1(0,T;\V)} 
 + 
 \underset{t\in(0,T)}{\mathrm{ess\,sup}}\, \left\| \chi_{h\tau}(t) \right\|_{1,\mathcal{T}}
 +
 \left\| \pek(\chi_{h\tau}) \right\|_{L^2(Q_T)} \\
 &\quad + \left\| \Bke(\chi_{h\tau}) \right\|_{L^\infty(0,T;L^1(\Omega))}
 +
 \left\| \partial_t \bke(\chi_{h\tau}) \right\|_{L^2(0,T;H^{-1}(\Omega))}
 +
 \left\| \partial_t \chi_{h\tau} \right\|_{L^2(Q_T)}
 \leq C_{\zeta\eta}
\end{align*} 
for some constant $C_{\zeta\eta}>0$ independent of $h,\tau$. Therefore, one obtains weak convergence for subsequences (denoted the same as before) for $h,\tau\rightarrow0$
\begin{alignat*}{2}
  \u_{h\tau}            &\rightharpoonup \uveta                && \text{ in }L^2(0,T;\V), \\
  \partial_t \u_{h\tau} &\rightharpoonup \partial_t \uveta     && \text{ in }L^2(0,T;\V),\\
  \pek(\chi_{h\tau})    &\rightharpoonup \pek(\chiveta)        && \text{ in }L^2(Q_T), \\
  \partial_t \bke(\chi_{h\tau}) &\rightharpoonup \partial_t \bke(\chiveta) &&\text{ in }L^2(0,T;Q^\star),\\
  \swk(\chi_{h\tau})\partial_t \DIV \u_{h\tau} &\rightharpoonup \swk(\chiveta)\partial_t \DIV \uveta &\quad&\text{ in }L^2(Q_T),\\
  \GRAD_h \chi_{h\tau} &\rightharpoonup \GRAD \chiveta      && \text{ in }L^2(Q_T).
 \end{alignat*}
 Moreover, by employing finite volume techniques the following convergence of the discrete diffusion term can be showed
 \begin{align*}
  \int_0^T \langle \GRAD_h \chi_{h\tau}, \GRAD_h \qh \rangle_{\absolutepermeability} \, dt
  \rightarrow 
  \int_0^T \langle \GRAD \chiveta, \GRAD q \rangle_{\absolutepermeability} \, dt,
 \end{align*}
 for arbitrary discrete test functions $\qh$, which strongly converge towards continuous functions $q$.
 Finally, the limit, $(\uveta,\chiveta)$, can be identified as weak solution of the regularized problem~\eqref{idea:model:regularized:weak:kirchhoff:u}--\eqref{idea:model:regularized:weak:kirchhoff:p}. 
 
\paragraph{Step 4: Increased regularity for the weak solution of the regularized model.}
When discussing the limit $\zeta \rightarrow 0$ in step~5, it will be beneficial to have access to the derivative in time of the mechanics equation~\eqref{idea:model:regularized:weak:kirchhoff:u} . Under the additional non-degeneracy condition~(ND2), i.e., that $\pek$ is Lipschitz continuous, an increased regularity can be showed for the weak solution of the regularized model, $(\uveta,\chiveta)$. For instance, for all $\v\in L^2(0,T;\V)$ it holds that
 \begin{align}\label{idea:model:regularized:weak:kirchhoff:dtu}
  \int_0^T  \left[ \zeta a(\partial_{tt} \uveta, \v ) 
 + a(\partial_t \uveta, \v)
 - \alpha \llangle \partial_t \pek(\chiveta), \DIV \v \rrangle \right]& \, dt = \int_0^T \llangle \partial_t \fext, \v \rrangle \, dt.
 \end{align}
 The proof follows the same line of argumentation as step~3. First a fully discrete counterpart of~\eqref{idea:model:regularized:weak:kirchhoff:dtu} is constructed by considering differences of~\eqref{idea:ht:kirchhoff-reduced:u} between subsequent time steps
 \begin{align*}
 &\zeta \tau^{-1} a(\uh^n - 2\uh^{n-1} + \uh^{n-2} , \vh)
 + a(\uh^n - \uh^{n-1}, \vh)\\
 &\qquad- \alpha \llangle \pek(\chih^n) - \pek(\chih^{n-1}), \DIV \vh \rrangle = \llangle \fext^n - \fext^{n-1}, \vh \rrangle \quad\text{for all } \vh\in\Vh.
 \end{align*}
 In addition, suitable interpolations $\vhtauHat$ and $\porepressurehtauHat$ of the discrete values $\{\tau^{-1} (\uh^n - \uh^{n-1}) \}_n$ and $\{\pek(\chih^n)\}_n$, respectively, define approximations of $\partial_t \uveta$ and $\pek(\chiveta)$. The uniform stability estimate
 \begin{align*}
  \| \partial_{t} \vhtauHat \|_{L^2(0,T;\V)}^2  
  +
  \| \partial_{t} \u_{h\tau} \|_{L^2(0,T;\V)}^2  
  +
  \| \partial_t \porepressurehtauHat \|_{L^2(Q_T)}^2  
  \leq 
  C_{\zeta\eta}
 \end{align*}
 guarantee the weak convergences
 \begin{alignat*}{2}
  \partial_{t} \vhtauHat        &\rightharpoonup \partial_{tt} \uveta,     &\quad& \text{in }L^2(0,T;\V),\\
  \partial_{t} \u_{h\tau}       &\rightharpoonup \partial_{t} \uveta,      &\quad& \text{in }L^2(0,T;\V),\\
  \partial_t \pek(\chi)_{h\tau} &\rightharpoonup \partial_t \pek(\chiveta),&\quad& \text{in }L^2(Q_T)
 \end{alignat*}
 up to subsequences, for $h,\tau\rightarrow 0$. Finally, one can identify~\eqref{idea:model:regularized:weak:kirchhoff:dtu} in the limit.

\paragraph{Step 5: Vanishing regularization in the mechanics equation.}

For each $\zeta,\eta>0$, there exists a solution $(\uveta,\chiveta)$ to~\eqref{idea:model:regularized:weak:kirchhoff:u}--\eqref{idea:model:regularized:weak:kirchhoff:p}. For the limit $\zeta\rightarrow 0$, we employ compactness arguments similar to step 3. However, now the stability estimates ought to be independent of $\zeta$. We show
\begin{align}
\label{idea:eps:stability-1}
  &\left\| \uv \right\|_{H^1(0,T;\V)}^2 
  +
  \left\| \chiv \right\|_{L^\infty(0,T;\Q)}^2
  +
  \left\| \porepressure(\chiv) \right\|_{L^2(Q_T)}^2\\
\nonumber
  &\quad+
  \left\| \Bke(\chiv) \right\|_{L^\infty(0,T;L^1(\Omega))} 
  +
  \left\| \partial_t \bke(\chiv) \right\|_{L^2(0,T;H^{-1}(\Omega))}
 \leq C,
\end{align}
and 
\begin{align}
\label{idea:eps:stability-2}
\| \partial_t \chiv \|_{L^2(Q_T)}^2 \leq C_\eta.
\end{align}
For~\eqref{idea:eps:stability-1}, one can use $\v=\partial_t \uveta$ and $q=\partial_t \chiveta$ as test functions in~\eqref{idea:model:regularized:weak:kirchhoff:p} and~\eqref{idea:model:regularized:weak:kirchhoff:dtu}. The coupling terms obviously do not match; but by using a binomial identity and the non-degeneracy condition (ND3), one can show that
\begin{align}
\label{idea:coupling-expression}
  &\left\| \partial_t \uv \right\|_{L^2(0,T;\V)}^2
  +
  \int_0^T \llangle \partial_t \bke(\chiveta), \partial_t \chiv \rrangle
  +
  \alpha \int_0^T \llangle \swk \partial_t \chiv - \partial_t \pek, \partial_t \DIV \uv \rrangle \geq 0,
\end{align}
which effectively allows for dropping the coupling terms. With this, letting $\zeta \rightarrow 0$, one obtains
for subsequences (denoted the same as before)
\begin{alignat*}{2}
  \uv     &\rightharpoonup \ueta                       && \text{ in }L^2(0,T;\V), \\
  \partial_t \uv     &\rightharpoonup \partial_t \ueta                       && \text{ in }L^2(0,T;\V), \\
  \zeta \partial_t \uv     &\rightarrow     \bm{0}                       && \text{ in }L^2(0,T;\V), \\
  \chiv &\rightharpoonup \chieta           && \text{ in }L^\infty(0,T;Q), \\
  \pek(\chiv) &\rightharpoonup \pek(\chieta)       && \text{ in }L^2(Q_T), \\
  \swk(\chiv)\partial_t \DIV \uv &\rightharpoonup \swk(\chieta)\partial_t \DIV \ueta &\quad &\text{ in }L^2(Q_T),\\
  \partial_t \bke(\chiv) &\rightharpoonup \partial_t \bke(\chieta) &&\text{ in }L^2(0,T;Q^\star).
 \end{alignat*}
 Finally, it is straightforward to see that the limit $(\ueta,\chieta)$ is weak solution of~\eqref{idea:model:regularized:weak:kirchhoff:u}--\eqref{idea:model:regularized:weak:kirchhoff:p} for $\zeta=0$.

 We underline, that for showing~\eqref{idea:coupling-expression}, the time-continuous character of the variational problem is required. It is not obvious how to use a similar strategy on time-discrete level. Therefore, step 5 has been performed separately from step~3 and~4.

\paragraph{Step 6: Vanishing regularization in the flow equation.}

In the presence of fluid or solid grain compressibility in the original formulation, i.e., $c_\mathrm{w}>0$ or $\frac{1}{N}>0$, respectively, this final step is obsolete. Otherwise, we consider the limit process $\eta \rightarrow 0$ for the sequence of solutions $\{(\ueta,\chieta)\}_\eta$, derived in step~5. The overall idea is the same as in step~5, namely to obtain estimates that are uniform wrt.\ $\eta$ and to use compactness arguments. Referring to~\eqref{idea:eps:stability-1}, the following estimate is uniform in $\eta$
\begin{align}
\label{idea:eta-regularization:uniform-stability-bound}
 &\left\| \ueta \right\|_{H^1(0,T;\V)} 
 + 
 \left\| \chieta \right\|_{L^\infty(0,T;H^1_0(\Omega))}
 +
 \left\| \pek(\chieta) \right\|_{L^2(Q_T)} \\
 \nonumber
 &\quad + \left\| \Bke(\chieta) \right\|_{L^\infty(0,T;L^1(\Omega))}
 +
 \left\| \partial_t \bke(\chieta) \right\|_{L^2(0,T;H^{-1}(\Omega))}
 \leq C.
\end{align}
For estimating $\partial_t\chieta$, we first show that the time derivative of the mechanics equation~\eqref{idea:ht:kirchhoff-reduced:u} is well-defined for $\zeta=0$, i.e., it holds for all $\v\in L^2(0,T;\V)$ that
\begin{align}
 \label{idea:result:regularization-eta-dt-mechanics}
  \int_0^T a(\partial_t \ueta, \v) \, dt -  \int_0^T \alpha\llangle \partial_t \pek(\chieta), \DIV \v \rrangle \, dt = \int_0^T \llangle \partial_t \fext, \v \rrangle \, dt.
\end{align}
Since $\| \partial_t \chieta \| \lesssim \| \partial_t \pek(\chieta) \|$, the uniform stability for $\partial_t \chieta$ follows by an inf-sup argument,~\eqref{idea:result:regularization-eta-dt-mechanics}, and the stability bound~\eqref{idea:eta-regularization:uniform-stability-bound}. Due to the lack of a suitable bound on $\partial_{tt} \uv$ in step~5, this approach only works for $\zeta=0$. Standard compactness arguments allow for extracting subsequences (again denotes as before) such that for $\eta \rightarrow 0$ it holds that
\begin{alignat*}{2}
  \ueta     &\rightharpoonup \u                   && \text{ in }L^2(0,T;\V), \\
  \chieta &\rightharpoonup \chi                   && \text{ in }L^\infty(0,T;Q), \\
  \pek(\chieta) &\rightharpoonup \pek(\chi)       && \text{ in }L^2(Q_T), \\
  \swk(\chieta)\partial_t \DIV \ueta &\rightharpoonup \swk(\chi)\partial_t \DIV \u &\quad &\text{ in }L^2(Q_T),\\
  \partial_t \bke(\chieta) &\rightharpoonup \partial_t \bk(\chi) &&\text{ in }L^2(0,T;Q^\star).
\end{alignat*}
Ultimately, $(\u,\chi)$ can be identified as a weak solution to the unsaturated poroelasticity model in the sense of Definition~\ref{definition:weak-solution}. This finishes the proof of Theorem~\ref{theorem:existence}.

\section{Step 1: Physical regularization -- secondary consolidation and enhanced grain compressibility}\label{section:regularization}
  
We introduce a physical regularization of the weak formulation~\eqref{model:weak:kirchhoff:u}--\eqref{model:weak:kirchhoff:p} by enhancing both the mechanics and the flow equations. We allow for secondary consolidation, which effectively incorporates a linear visco-elasticity contribution in the mechanics equations of the form $a(\partial_t \u,\v)$.  Additionally, we assume non-vanishing grain compressibility by regularizing $\bk$. Specifically, we let $\zeta>0$ and $\eta>0$ be two regularization parameters and analyze the behavior of the solution when passing them to zero.

Motivated by the physical example~\eqref{model:example-b}, for $\eta>0$, define the regularization $\bke$ of $\bk$ by
\begin{align*}
 \bke(\chi) := \bk(\chi) + \eta \int_0^{\pwk(\chi)} s_\mathrm{w}(p) \porepressure'(p) \, dp,
\end{align*}
i.e., $\bke$ has the same structure as $\bk$, but with $\tfrac{1}{N}+\eta$ replacing $\tfrac{1}{N}$. Refering to Section~\ref{section:assumptions}, the function $\bke$ still satisfies~(A1). Additionally, a uniform growth condition holds
\begin{itemize}
 \item[(A1$^\star$)] There exists a $\bk_{\chi,\mathrm{m}}>0$ s.t.\ $\bk_{\chi,\mathrm{m}} \| \chi_1 - \chi_2 \|^2 \leq \llangle \bke(\chi_1) - \bke(\chi_2), \chi_1 - \chi_2 \rrangle \ \text{for all } \chi_1,\chi_2\in L^2(Q_T)$,
\end{itemize}
 cf.\ also Section~\ref{appendix:feasibility-assumptions}. In the subsequent discussion, a growth condition for $\bke$ (or $\bk$) of type~(A1$^\star$) will be required in order to to utilize strong compactness arguments. If $\mathrm{min}\left\{c_\mathrm{w},\tfrac{1}{N} \right\}>0$ in~\eqref{model:example-b} holds, the growth condition~(A1$^\star$) is fulfilled even for $\eta=0$, and the regularization of the flow equation actually is not necessary, cf.\ Step 6 in Section~\ref{section:eta-regularization}. In this context, we emphasize that~(ND3) also holds for $\bke$ as $\bke' \geq \bk'$.
 
 Also~(A8) can be adapted for the regularization $\bke$. With $\bar{b}_\eta:=\bke \circ \left(\frac{\pek}{\swk}\right)^{-1}$, we let $\Bke$ and $\bar{B}_\eta$ be the Legendre transformations of $\bke$ and $\bar{b}_\eta$, respectively, defined by
 \begin{align}
 \label{legendre-transform-bk}
  \Bke(z) &:= \int_0^z (\bke(z) - \bke(s))\, ds \geq 0,\\
  \label{legendre-transform-b-bar}
  \bar{B}_\eta(z) &:= \int_0^z (\bar{b}_\eta(z) - \bar{b}_\eta(s))\, ds \geq 0.
 \end{align}
 \begin{itemize}
  \item[(A8$^\star$)] There exists a $\eta_0>0$ and $C_0>0$, not depending on $\eta_0$, such that
   \begin{align*}
  \| \u_{0} \|_{\V}^2 
  + \left\| \GRAD \chi_0 \right\|^2 
  + \left\| \Bke \left( \chi_{0}) \right) \right\|_{L^1(\Omega)} 
  + \left\| \bar{B}_\eta \left(\frac{\pek(\chi_{0})}{\swk(\chi_{0})}\right) \right\|_{L^1(\Omega)} 
  \leq C_0
 \end{align*}
 for all $\eta\in(0,\eta_0)$. Without loss of generality, we assume $C_0$ in (A8) and (A8$^\star$) to be the same.
 \end{itemize}
 \noindent
 For a non-degenerate initial condition $\chi_0$, the additional terms in $\Bke$ and $\bar{B}_\eta$ can be essentially bounded by $\eta\| \chi_0 \|^2$, which itself is bounded by~(A8). 
 
 
We introduce the notion of a weak solution of the doubly regularized unsaturated poroelasticity model.

\begin{definition}[Weak solution of the doubly regularized model]\label{definition:regularization:weak-solution}

For $\zeta>0$ and $\eta>0$, we call $(\uveta,\chiveta)\in L^2(0,T;\V) \times L^2(0,T;Q)$ a weak solution of the doubly regularized unsaturated poroelasticity model if it satisfies:

\begin{enumerate}
 
 \item[$\mathrm{(W1)_{\zeta\eta}}$] $\pek(\chiveta)\in L^2(Q_T)$, $\swk(\chiveta)\in L^\infty(Q_T)$.
 
 \item[$\mathrm{(W2)_{\zeta\eta}}$] $\bke(\chiveta)\in L^\infty(0,T;L^1(\Omega))$ and $\partial_t \bke(\chiveta)\in L^2(0,T; Q^\star)$ such that
 \begin{align*}
  \int_0^T \llangle \partial_t \bke(\chiveta), \q \rrangle \, dt
  +
  \int_0^T \llangle \bke(\chiveta) - \bke(\chi_0), \partial_t \q \rrangle \, dt
  =
  0,
 \end{align*}
 for all $\q\in L^2(0,T; Q)$ with $\partial_t \q \in L^1(0,T;L^\infty(\Omega))$ and $\q(T)=0$.
 
 \item[$\mathrm{(W3)_{\zeta\eta}}$] $\partial_t \uveta\in L^2(0,T;\V)$ such that 
 \begin{align*}
  \int_0^T a(\partial_t \uveta, \v) \, dt + \int_0^T a(\uveta - \u_0, \partial_t \v) \, dt = 0,
 \end{align*}
 for all $\v\in H^1(0,T;\V)$ with $\v(T)=\bm{0}$.
 
 \item[$\mathrm{(W4)_{\zeta\eta}}$] $(\uveta,\chiveta)$ satisfies the variational equations
 \begin{align}\label{model:regularized:weak:kirchhoff:u} 
 \int_0^T  \Big[ \zeta a(\partial_t \uveta, \v)
 + a(\uveta, \v)
 - \alpha \llangle \pek(\chiveta), \DIV \v \rrangle \Big]& \, dt = \int_0^T \llangle \fext, \v \rrangle \, dt, \\
 \int_0^T \bigg[
 \llangle \partial_t \bke(\chiveta), \q \rrangle + \alpha \llangle \swk(\chiveta) \partial_t \DIV \uveta, \q \rrangle 
 + \llangle \absolutepermeability \GRAD \chiveta, \GRAD \q \rrangle
 \bigg]& \, dt = \int_0^T \llangle \hext, \q \rrangle\, dt,   \label{model:regularized:weak:kirchhoff:p}
 \end{align}
 for all $(\v,\q)\in L^2(0,T;\V) \times L^2(0,T;Q)$.
\end{enumerate}

\noindent
Furthermore, we call $(\uveta,\chiveta)$ a \textit{weak solution with increased regularity for the doubly regularized unsaturated poroelasticity model} if it satisfies $\mathrm{(W1)_{\zeta\eta}}$--$\mathrm{(W4)_{\zeta\eta}}$ and:
\begin{enumerate}

 \item[$\mathrm{(W5)_{\zeta\eta}}$] $\uveta \in H^2(0,T;\V)$ and $\partial_t \pek(\chiveta) \in L^2(Q_T)$. 
  
 \item[$\mathrm{(W6)_{\zeta\eta}}$] It holds 
 \begin{align}\label{model:regularized:weak:kirchhoff:dtu}
  \int_0^T  \Big[ \zeta a(\partial_{tt} \uveta, \v ) 
 + a(\partial_t \uveta, \v)
 - \alpha \llangle \partial_t \pek(\chiveta), \DIV \v \rrangle \Big]& \, dt = \int_0^T \llangle \partial_t \fext, \v \rrangle \, dt, 
 \end{align}
 for all $\v\in L^2(0,T;\V)$, given that $\fext \in H^1(0,T;\V^\star)$.
\end{enumerate}
\end{definition}

\noindent
We will later separately consider $\zeta\rightarrow 0$ and $\eta\rightarrow0$. Therefore, we give the definition of a weak solution for the simply regularized unsaturated poroelasticity model, obtained for $\eta>0$ and $\zeta=0$.

\begin{definition}[Weak solution of the simply regularized model]\label{definition:simply-regularized-weak-solution}
 For $\eta>0$, we call $(\ueta,\chieta)$ a \textit{weak solution of the simply regularized unsaturated poroelasticity model} if it satisfies~$\mathrm{(W1)_{\zeta\eta}}$--$\mathrm{(W4)}_{\zeta\eta}$ for $\zeta=0$. 
\end{definition}

To distinguish between the equations satisfied by the weak solution of a doubly regularized model and the one of the simply regularized one, where $\epsilon_v = 0$, we use the notations $\mathrm{(W1)_{\eta}}$--$\mathrm{(W4)}_{\eta}$.

\begin{lemma}[Existence of a weak solution to the doubly regularized model]\label{lemma:existence-doubly-regularized}
 Let $\zeta > 0$ and $\eta > 0$ be given. Under the assumptions $\mathrm{(A0)}$--$\mathrm{(A9)}$ and~$\mathrm{(ND1)}$ there exists a weak solution to the doubly regularized unsaturated poroelasticity model, in the sense of Definition~\ref{definition:regularization:weak-solution}. 
\end{lemma}

\begin{proof}
 The assertion follows from steps~2--3.
\end{proof}

\begin{lemma}[Existence of a weak solution with increased regularity for the doubly regularized model]\label{lemma:existence-doubly-regularized-increased-regularity}
 Let $\zeta>0$ and $\eta>0$ be given. Under the assumptions $\mathrm{(A0)}$--$\mathrm{(A9)}$ and the non-degeneracy conditions~$\mathrm{(ND1)}$--$\mathrm{(ND2)}$, the doubly regularized unsaturated poroelasticity model has a weak solution with increased regularity, in the sense of Definition~\ref{definition:regularization:weak-solution}. 
\end{lemma}

\begin{proof}
 The assertion follows from steps~2--4.
\end{proof}
\begin{lemma}[Existence of a weak solution for the simply regularized model]\label{lemma:existence-simply-regularized}
Let $\eta > 0$ be given. Under the assumptions $\mathrm{(A0)}$--$\mathrm{(A9)}$ and the non-degeneracy conditions~$\mathrm{(ND1)}$--$\mathrm{(ND3)}$, the doubly regularized unsaturated poroelasticity model has a weak solution with increased regularity, in the sense of Definition~\ref{definition:simply-regularized-weak-solution}. 
\end{lemma}

\begin{proof}
 The assertion follows from step~5.
\end{proof}

\section{Step 2: Implicit Euler non-linear FEM-TPFA discretization}\label{section:fully-discrete}

The next two sections, identified with steps 2 and 3, are providing the proof of Lemma~\ref{lemma:existence-doubly-regularized}. To this aim, we employ the implicit Euler time stepping method, whereas for the spatial discretization of the mechanics equation~\eqref{model:regularized:weak:kirchhoff:u} a conforming Galerkin finite element method is used. For the flow equation~\eqref{model:regularized:weak:kirchhoff:p}, the spatial discretization can be interpreted in various ways. It can be viewed as cell-centered finite volume method utilizing a two point flux approximation (TPFA), the simplest approximation one can consider, but it can also be interpreted as lowest order mixed finite element method with inexact quadrature allowing for lumping~\cite{Baranger1996}. In this section, we show the existence of a fully discrete solution. We start with introducing the notations used in the discretization.

\subsection{Finite volume and finite element notation}

We use standard notations in the finite volume literature, see e.g.~\cite{Eymard1999,Saad2013}. In particular, we introduce notation for elements, faces, their measures, transmissibilities etc. We assume that the domain $\Omega$ is polygonal such that it can be discretized by an admissible mesh, as introduced by~\cite{Eymard2000}.

\begin{definition}[Admissible mesh $\mathcal{T}$]\label{definition:admissible-mesh}
Let $\mathcal{T}$ be a regular mesh of $\Omega$ with mesh size $h$, consisting of simplices in 2D or 3D, or convex quadrilaterals in 2D and convex hexahedrals in 3D. Furthermore, we introduce the following terminology:
\begin{itemize}
 \item $K\in \mathcal{T}$ denotes a single element.
 
 \item $\mathcal{N}(K):= \left\{L\in \mathcal{T} \,| \, L\neq K,\ \bar{L}\cap\bar{K}\neq \emptyset \right\}$ denotes the set of neighboring elements of $K\in \mathcal{T}$.
 
 \item $\mathcal{E}$ denotes the set of all faces, i.e., boundaries of all elements; let $\mathcal{E}_K$ denote the faces of a single element $K\in\mathcal{T}$; let $\mathcal{E}_\mathrm{ext}$ denote the faces lying on the boundary $\partial\Omega$.
 
 \item $K|L\in \mathcal{E}$ denotes the face between two neighboring elements $K,L\in\mathcal{T}$.
  
 \item $\{x_K\}_{K\in\mathcal{T}}$ is such that for all $K\in \mathcal{T},L\in\mathcal{N}(K)$ the connecting line between $\x_K$ and $\x_L$ is perpendicular to $K|L$ .
 
 \item $d_{K,\sigma}$ denotes the distance between center of $K$ and $\sigma\in\mathcal{E}_K$;
 \begin{align*}
  d_\sigma = \left\{ \begin{array}{ll} d_{K,\sigma} + d_{L,\sigma}, & K\in\mathcal{T},\ L\in\mathcal{N}(K),\ \sigma=K|L, \\
  d_{K,\sigma}, & \sigma \in \mathcal{E}_\mathrm{ext}\cap\mathcal{E}_K.\end{array} \right.
 \end{align*}
 
 \item $\tau_\sigma = |\sigma| / d_\sigma$ denotes the transmissibility through $\sigma\in\mathcal{E}$.
\end{itemize}
 Assume there holds the regularity property: there exists a constant $C>0$ such that 
 \begin{align*}
  \sum_{\substack{L \in \mathcal{N}(K) \\ \sigma=K|L}} |\sigma| d_\sigma \leq C |K|\quad \text{for all }K\in \mathcal{T}.
 \end{align*}
\end{definition}

We introduce a dual grid $\mathcal{T}^\star$ with diamonds as elements. It will be used for the approximation of heterogeneous permeability fields. Additionally, it will be utilized within the proof.

\begin{definition}[Dual grid to $\mathcal{T}$]\label{definition:dual-mesh}
 Let $\mathcal{T}$ be an admissible mesh, cf.\ Definition~\ref{definition:admissible-mesh}. For each face $K|L\in\mathcal{E}$, $K\in\mathcal{T}$, $L\in\mathcal{N}(K)$, define a prism $P_{K|L}\subset\Omega$ with $\x_K$, $\x_L$ and the vertices of $K|L$ as vertices. For all $\sigma \in\mathcal{E}_\mathrm{ext}\cap\mathcal{E}_K$, $K\in\mathcal{T}$ define $P_\sigma\subset\Omega$ to be the prism with $\x_K$ and the vertices of $\sigma$ as vertices. By construction, $\mathcal{T}^\star:=\{P_\sigma\}_{\sigma\in\mathcal{E}}$ defines a partition of $\Omega$. 
\end{definition}
Figure~\ref{figure:mesh} displays a two-dimensional, admissible mesh and its auxiliary, dual grid.

\begin{figure}[h!]
 \centering
 \begin{overpic}[width=0.65\textwidth]{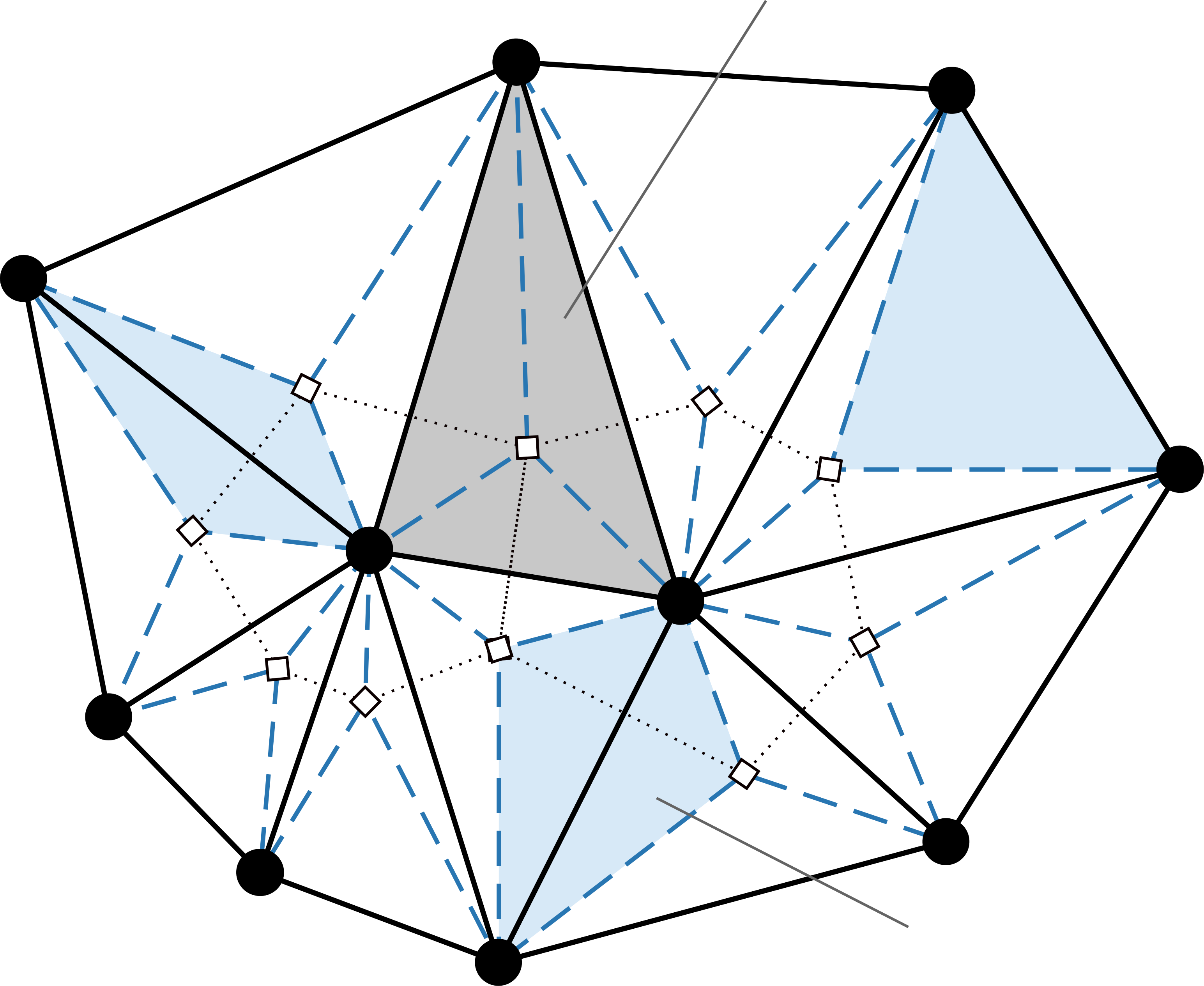}
  \put(78,3.5){diamond $\in\mathcal{T}^\star$}
  \put(60,84){element $\in\mathcal{T}$}
  \put(9,36){$x_K$}
  \put(14,28){$K$}
  \put(27.5,49.5){$x_L$}
  \put(20,59){$L$}
  \put(16,48.5){\small $P_{K|L}$}
  \put(18.15,38.5){\small $K|L$}
  \put(90,58){$\sigma\in\mathcal{E}_\mathrm{ext}$}
  \put(72,38){$M$}
  \put(71,44.5){$x_M$} 
  \put(80,50){$P_\sigma$}
 \end{overpic}
 \caption{\label{figure:mesh} Admissible mesh $\mathcal{T}$ (consisting of elements) in two dimensions, together with the corresponding dual grid $\mathcal{T}^\star$ (consisting of diamonds).}
\end{figure}
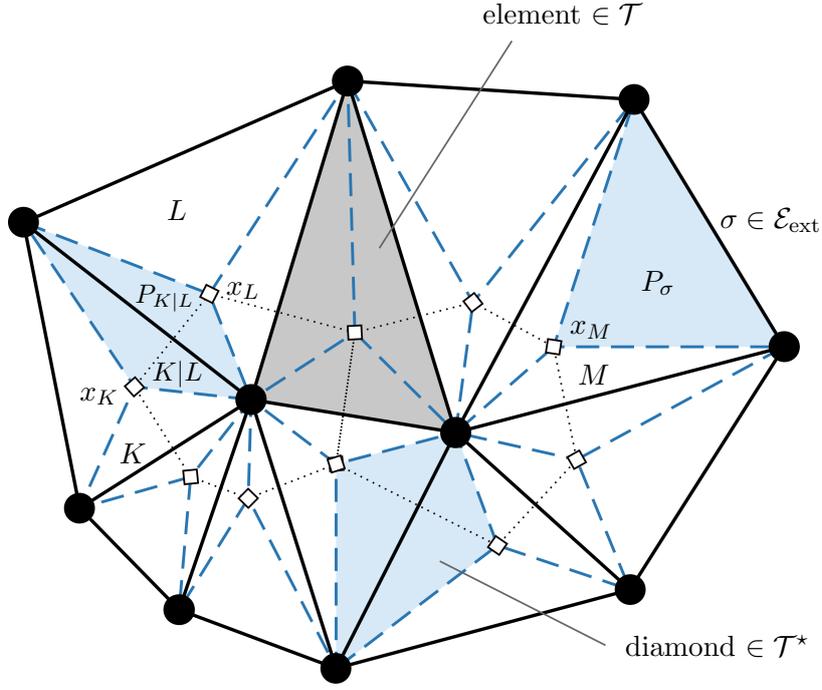

The final discrete scheme is written in variational form. Given an admissible mesh $\mathcal{T}$, we introduce the discrete function spaces and implicitly their bases
\begin{alignat*}{2}
 \Vh &= \text{span}\,\{\vhi\}_{i\in\{1,...,d_\mathrm{V}\}}, \\
 \Qh &= \text{span}\,\{\qhj\}_{j\in\{1,...,d_\mathrm{Q}\}},
\end{alignat*}
providing spaces for the discrete displacement and pressure, respectively. For the analysis below, we assume that the discrete function spaces to satisfy the following conditions:
\begin{itemize}
 \item[(D1)] $\Qh$ is the space of all piecewise constant functions ($\mathbb{P}_0$) on $\mathcal{T}$ and the basis $\{\qhj\}_j$ is equal to the indicator functions of all single elements. Note $\Qh \not\subset \Q$.

 \item[(D2)] $\Vh\subset \V$ such that $\Vh\times\Qh$ is inf-sup stable regarding the bilinear form
 \begin{alignat*}{3}
  \Vh\times\Qh &\rightarrow \mathbb{R},\quad&& (\vh,\qh) \mapsto \langle \qh, \DIV \vh \rangle.
 \end{alignat*}
 In more detail, there exists a constant $\gamma_\mathrm{is}=\COmegaInfSup^{-1}>0$ (independent of $h$), such that
 \begin{alignat}{2}
 \label{inf-sup-div}
  \underset{0\neq \qh\in\Qh }{\mathrm{inf}}\, \underset{\vh\in\Vh}{\mathrm{sup}}\, &\frac{\llangle \qh, \DIV \vh \rrangle}{\| \qh \|\, \|\vh\|_{\V}} &&\geq \gamma_\mathrm{is}.
 \end{alignat} 
\end{itemize}
In the analysis, (D1) will allow for intuitively handling non-linearities in the pressure variable easily. Assumption~(D2) will allow for using standard inf-sup arguments. In two dimensions, one can use piecewise quadratic elements for $\Vh$. In three dimensions, a practical choice is less trivial, cf.~\cite{Boffi2013} for a thorough discussion.

In the analysis, we require the notion of a discrete $H^1(\Omega)$ norm for piecewise constant functions in $Q_h$, see also~\cite{Eymard1999}.
\begin{definition}[Discrete $H^1(\Omega)$ norms on $Q_h$]\label{definition:discrete-h1-norm}
Let $\qh \in \Qh$. We define
 \begin{align*}
  \| q_h \|_{1,\mathcal{T}} := \left( \sum_{\sigma \in\mathcal{E}} \tau_\sigma \, \delta_\sigma( q_h )^2\right)^{\tfrac{1}{2}},
 \end{align*}
 where
 \begin{align*}
  \delta_\sigma q_h := \left\{ \begin{array}{ll} \left| {{q_h}_|}_K - {{q_h}_|}_L \right|, & K\in\mathcal{T},\ L\in\mathcal{N}(K),\ \sigma=K|L,\\[10pt]
                           \left| {{q_h}_|}_K \right| , & \sigma \in \mathcal{E}_\mathrm{ext} \cap \mathcal{E}_K.
                          \end{array} \right.
 \end{align*}
 In the same sense, given a uniformly positive field $\omega \in C(\Omega)$, a scaled inner product of discrete gradients is defined by 
\begin{align*} 
 \langle \GRAD_h \chih, \GRAD_h \qh \rangle_{\omega} 
 &:= 
 \sum_{K\in \mathcal{T}} \sum_{L\in \mathcal{N}(K)} \tau_{K|L} \, \{ \omega \}_{K|L} \, \delta_{K|L}(\chih) \, \delta_{K|L}(\qh) + \sum_{\sigma\in\mathcal{E}_\mathrm{ext} \cap \mathcal{E}_K} \tau_{K,\sigma} \, \{ \omega \}_{\sigma} \,  {{\chih}_|}_K \, {{\qh}_|}_K.
\end{align*}
where the the weight $\omega$ evaluated at faces is approximated as weighted average incorporating the neighboring elements, i.e., utilizing the dual mesh $\mathcal{T}^\star$ to $\mathcal{T}$ it is
\begin{align*}
 \{ \omega \}_{\sigma} := \frac{1}{|P_\sigma|} \int_{P_\sigma} \omega(x) \, dx,\quad \sigma\in\mathcal{E}.
\end{align*}
A norm $\|\cdot\|_{1,\mathcal{T},\omega} := \llangle \GRAD_h \cdot, \GRAD_h \cdot \rrangle_{\omega}^{1/2}$ is naturally induced.
\end{definition}
A discrete Poincar\'e inequality can be showed for $\|\cdot\|_{1,\mathcal{T}}$, introducing a discrete Poincar\'e constant $\COmegaDiscretePoincare>0$ such that
\begin{align*}
 \| \qh \| \leq \COmegaDiscretePoincare \| \qh \|_{1,\mathcal{T}} \qquad \text{for all } \qh\in \Qh,
\end{align*}
cf.\ Lemma~\ref{appendix:lemma:discrete-poincare}; similarly also for $\|\cdot\|_{1,\mathcal{T},\omega}$.

\subsection{Approximation of source terms and initial conditions}
Let $0=t_0<t_1<...<t_N=T$ define a partition of the time interval $(0,T)$ with constant time step size $\tau=t_{n}-t_{n-1}=T/N$, $n,N\in\mathbb{N}$. We interpolate the source terms at discrete time steps. Let
 \begin{align*}
  \fext^n  &:= \frac{1}{\tau} \int_{t_{n-1}}^{t_n} \fext(t) \, dt, \\
  \hext^n  &:= \frac{1}{\tau} \int_{t_{n-1}}^{t_n} \hext(t) \, dt.
 \end{align*}
 Discrete initial conditions are chosen to imitate the compatibility assumption~(A9). Let $\chih^0\in\Qh$ be defined by the piecewise constant projection of $\chi^0$, i.e., on $K\in\mathcal{T}$, we define 
 \begin{align*}
  {{\chih^0}_|}_K := \frac{1}{|K|} \int_K \chi^0 \, dx. 
 \end{align*}
 As $\chi_0 \in L^2(\Omega)$, cf.\ (A8$^\star$), it follows by classical approximation theory for $h\rightarrow 0$
 \begin{alignat*}{2}
  \chih^0 &\rightarrow \chi_0&\ &\text{in }L^2(\Omega),
 \end{alignat*}
 and it holds that $\| \chih^0 \|_{1,\mathcal{T},\absolutepermeability} \leq C \| \chi_0 \|_1$ for some constant $C>0$, cf., e.g.,~\cite{Eymard2000}. 
 Furthermore, since $\pek \in C(\mathbb{R})$, cf.\ (A3), and $\pek(\chi_0)\in L^2(\Omega)$, cf.\ (A8$^\star$), it follows for $h\rightarrow 0$
 \begin{alignat*}{2}
  \pek(\chih^0) &\rightarrow \pek(\chi_0)&\ &\text{in }L^2(\Omega),
 \end{alignat*}
 similarly for $\{\Bke \left(\chih^0 \right)\}_h$ and $\left\{\bar{B}_\eta\left(\frac{\pek(\chih^{0})}{\swk(\chih^{0})}\right)\right\}_h$. Then in order to satisfy~(A9) in a discrete sense, we define $\uh^0\in\Vh$ to be the unique element in $\Vh$, satisfying
 \begin{align}\label{compatibility:initial-conditions}
  a(\uh^0, \vh) - \alpha \llangle \pek(\chih^0), \DIV \vh \rrangle = \llangle \fext(0), \vh \rrangle,\quad \text{for all } \vh\in\Vh.
 \end{align}
 Using standard finite element techniques and (A9), it holds that
 \begin{align*}
  \| \u_0 - \uh^0 \|_{\V} \leq 2 \underset{\vh\in\Vh}{\mathrm{inf}}\, \|\u_0 - \vh\|_{\V} + \frac{\alpha}{K_\mathrm{dr}} \| \pek(\chi_0) - \pek(\chih^0) \|.
 \end{align*}
 Hence, by classical approximation theory and the imposed regularity~(A8$^\star$) it follows for $h\rightarrow 0$
 \begin{alignat*}{2}
  \uh^0   &\rightarrow \u_0 &\ &\text{in }\V. 
 \end{alignat*}
 All in all, due to the convergence,~(A8$^\star$) also applies on discrete level.
 \begin{itemize}
  \item[(A8$^\star$)$_h$] For bounded $\eta>0$, there exists a constant $C_0>0$ (wlog.\ the same as in (A8)) such that 
   \begin{align*}
  \| \uh^{0} \|_{\V}^2
  + \left\| \chih^0 \right\|_{1,\mathcal{T},\absolutepermeability}^2
  + \left\| \Bke \left(\chih^0 \right) \right\|_{L^1(\Omega)} 
  + \left\| \bar{B}_\eta\left(\frac{\pek(\chih^{0})}{\swk(\chih^{0})}\right) \right\|_{L^1(\Omega)}.
  \leq C_0
 \end{align*}
 \end{itemize}
 
\subsection{Approximation of the evolutionary problem}

The discretization of \eqref{model:regularized:weak:kirchhoff:u}--\eqref{model:regularized:weak:kirchhoff:p} is defined by the Galerkin method combined with the standard implicit Euler time discretization:  for $n\geq1$, given $(\uh^{n-1},\chih^{n-1}) \in \Vh\times\Qh$, find $(\uh^n,\chih^n)\in\Vh\times\Qh$ satisfying for all $(\vh,\qh)\in\Vh\times\Qh$
\begin{align}
   \zeta\tau^{-1}a (\uh^n - \uh^{n-1},\vh) 
   +  a(\uh^n,\vh) - \alpha \langle \pek(\chih^n), \DIV \vh \rangle &\= \langle \fext^n, \vh \rangle, \label{ht:kirchhoff-reduced:u}\\
   \label{ht:kirchhoff-reduced:p}
  \langle \bke(\chih^n) - \bke(\chih^{n-1}), \qh \rangle + \alpha \langle \swk(\chih^n) \DIV(\uh^n - \uh^{n-1}), \qh\rangle \\
  \nonumber
  + \, \tau \langle \GRAD_h \chih^n, \GRAD_h \qh \rangle_{\absolutepermeability} &\= \tau \langle \hext^n, \qh \rangle.
\end{align}

\begin{lemma}[Existence of a discrete solution]\label{lemma:existence-discrete-solution}
 Let $n\geq 1$. $\mathrm{(A0)}$--$\mathrm{(A9)}$,~$\mathrm{(ND1)}$, and $\mathrm{(D1)}$--$\mathrm{(D2)}$ hold true. Then there exists a discrete solution $(\uh^n,\chih^n)\in\VQh$ satisfying~\eqref{ht:kirchhoff-reduced:u}--\eqref{ht:kirchhoff-reduced:p}, and
 \begin{align}
 \label{existence-discrete-solution-aux-02}
  \left\| \bar{B}_\eta\left( \frac{\pek(\chih^{n})}{\swk\left(\chih^{n}\right)}\right)\right\|_{L^1(\Omega)}
  + \left\| \uh^{n} \right\|_{\V}^2 < \infty\quad\text{for all }n\geq 1.
 \end{align}
\end{lemma}

\begin{proof}
The proof is by induction. We present only the general step, since the proof for $n=1$ is similar. We employ a corollary of Brouwer's fixed point theorem, cf.~Lemma~\ref{appendix:lemma:brouwer}, to show the existence of a solution of a non-linear algebraic system, which is equivalent to~\eqref{ht:kirchhoff-reduced:u}--\eqref{ht:kirchhoff-reduced:p}.
 
 \paragraph{\textbf{Introduction of a pressure-reduced algebraic problem}.}  We introduce an isomorphism between the discrete function space corresponding to the fluid pressure $\chi$ and a suitable coefficient vector space
 \begin{alignat*}{2}
  \chih:& \ \mathbb{R}^{d_\mathrm{Q}} \rightarrow \Qh, \ \  & \bm{\beta}  &\mapsto \sum_j \left(\tfrac{\pek}{\swk}\right)^{-1}\!\!(\beta_j) \, \qhj.
 \end{alignat*}
 Due to (A4), $\chih$ is well-defined. Similarly, let
 \begin{alignat*}{2}
  \uh:&\  \mathbb{R}^{d_\mathrm{V}} \rightarrow \Vh, \ \ & \bm{\alpha} &\mapsto \sum_i \alpha_i \vhi.
 \end{alignat*}
 For given $\bm{\beta}\in\mathbb{R}^{d_\mathrm{Q}}$, define $\bm{\alpha}=\bm{\alpha}(\bm{\beta})\in\mathbb{R}^{d_\mathrm{V}}$ to be the unique solution to: find $\bm{\alpha}\in\mathbb{R}^{d_\mathrm{V}}$ such that 
\begin{align*}
 \zeta\tau^{-1}a( \uh(\bm{\alpha}) - \uh^{n-1}, \vh)
 +
 a(\uh(\bm{\alpha}), \vh)& \\
 =
 \llangle \bm{f}^n, \vh \rrangle + \alpha \llangle \pek(\chih(\bm{\beta})), \DIV \vh \rrangle 
 ,&\ \text{for all } \vh \in \Vh.
\end{align*}
Finally, we define $\bm{F}:\mathbb{R}^{\dQ} \rightarrow \mathbb{R}^{\dQ}$ by 
 \begin{align*}
  F_j(\bm{\beta}) 
  =& \,
    \llangle \bke(\chih(\bm{\beta})) - \bke(\chih^{n-1}), \qhj \rrangle + \alpha \llangle \swk(\chih(\bm{\beta})) \DIV (\uh(\bm{\alpha}(\bm{\beta})) - \uh^{n-1}), \qhj \rrangle \\
    &\qquad + \tau \llangle \GRAD_h \chih(\bm{\beta}), \GRAD_h \qhj \rrangle_{\absolutepermeability} - \tau \llangle \hext^n, \qhj \rrangle,\quad j\in\{1,...,d_\mathrm{Q}\}.
 \end{align*}
 We note, the existence of a discrete solution of Eq.~\eqref{ht:kirchhoff-reduced:u}--\eqref{ht:kirchhoff-reduced:p} is equivalent to the existence of $\bm{\beta} \in \mathbb{R}^{\dQ}$, satisfying $\bm{F}(\bm{\beta})=\bm{0}$. To prove the existence of a zero of $\bm{F}$, we employ Lemma~\ref{appendix:lemma:brouwer}; we consider the expression
 \begin{align}
 \label{existence-discrete-solution-aux-1}
  \llangle \bm{F}(\bm{\beta}), \bm{\beta} \rrangle
  &=
  \llangle \bke(\chih(\bm{\beta})) - \bke(\chih^{n-1}), \frac{\pek(\chih(\bm{\beta}))}{\swk(\chih(\bm{\beta}))} \rrangle \\
  \nonumber
  &\quad +
  \alpha \llangle \DIV (\uh(\bm{\alpha}) - \uh^{n-1}), \pek(\chih(\bm{\beta})) \rrangle \\
  \nonumber
  &\quad +
  \tau \llangle \GRAD_h \chi(\bm{\beta}), \GRAD_h \tfrac{\pek(\chih(\bm{\beta}))}{\swk(\chih(\bm{\beta}))} \rrangle \\
  \nonumber
  &\quad
  - \tau \llangle \hext^n, \tfrac{\pek(\chih(\bm{\beta}))}{\swk(\chih(\bm{\beta}))} \rrangle\\[0.25em]
  \nonumber
  &=: T_1 + T_2 + T_3 + T_4.
 \end{align}
 where we used
  \begin{align*}
  \sum_{j=1}^{d_\mathrm{Q}} \beta_j \qhj &=
  \frac{\pek(\chih(\bm{\beta}))}{\swk(\chih(\bm{\beta}))}.
 \end{align*}
 and dropped the explicit dependence of $\bm{\alpha}$ on $\bm{\beta}$. We discuss the terms $T_1,...,T_4$ separately. 
 
 \paragraph{Discussion of $T_1$.}
 Using~(A4), we define $\bar{b}_\eta:=\bke \circ \left(\tfrac{\pek}{\swk}\right)^{-1}:\mathbb{R} \rightarrow \mathbb{R}$ and its Legendre transformation, cf.~\eqref{legendre-transform-b-bar}. Finally, using standard properties of the Legendre transformation of non-decreasing functions, cf.\ Lemma~\ref{appendix:lemma:legendre}, we obtain for term $T_1$
 \begin{align*}
  T_1
  \geq 
  \left\| \bar{B}_\eta\left( \frac{\pek(\chih(\bm{\beta}))}{\swk(\chih(\bm{\beta}))}\right) \right\|_{L^1(\Omega)} 
  - 
  \left\| \bar{B}_\eta\left( \frac{\pek(\chih(\bm{\beta}^{n-1}))}{\swk(\chih(\bm{\beta}^{n-1}))}\right) \right\|_{L^1(\Omega)},
 \end{align*}
 where $\bm{\beta}^{n-1}\in\mathbb{R}^{d_\mathrm{Q}}$ such that $\chih^{n-1} = \chih(\bm{\beta}^{n-1})$.
 
 \paragraph{Discussion of $T_2$.}
 From the definition of $\bm{\alpha}$, under the use of a binomial identity, the Cauchy-Schwarz inequality and Young's inequality, the coupling term $T_2$ becomes
 \begin{align*}
  T_2 &=\alpha \llangle \DIV (\uh(\bm{\alpha}) - \uh^{n-1}), \pek(\chih(\bm{\beta})) \rrangle \\
  &=
  \zeta \tau^{-1} \left\| \uh(\bm{\alpha}) - \uh^{n-1} \right\|_{\V}^2
  +
  \tfrac{1}{2} \left\| \uh(\bm{\alpha}) \right\|_{\V}^2
  +\tfrac{1}{2} \left\| \uh(\bm{\alpha}) - \uh^{n-1} \right\|_{\V}^2 \\
  &\qquad  -\tfrac{1}{2} \left\| \uh^{n-1} \right\|_{\V}^2
  -
  \llangle \bm{f}^n, \uh(\bm{\alpha}) - \uh^{n-1} \rrangle \\
  &\geq 
  \zeta \tau^{-1} \left\| \uh(\bm{\alpha}) - \uh^{n-1} \right\|_{\V}^2
  +
  \tfrac{1}{2} \left\| \uh(\bm{\alpha}) \right\|_{\V}^2
  +\tfrac{1}{4} \left\| \uh(\bm{\alpha}) - \uh^{n-1} \right\|_{\V}^2 \\
  &\qquad  -\tfrac{1}{2} \left\| \uh^{n-1} \right\|_{\V}^2
  -
  \| \bm{f}^n \|_{\V^\star}^2.
 \end{align*}
 
 \paragraph{Discussion of $T_3$.}
 By the mean value theorem and~(A4), the diffusion term  $T_3$ can be estimated from below
 \begin{align*}
   T_3 
   \geq
   c_{\pek/\swk} \tau
   \| \chih(\bm{\beta}) \|_{1,\mathcal{T},\absolutepermeability}^2.
 \end{align*}
 
 \paragraph{Discussion of $T_4$.}
 Employing the definition of $\hext=(h,w_\mathrm{N})$, the non-degeneracy condition (ND1), a discrete trace inequality, cf.\ Lemma~\ref{appendix:lemma:discrete-trace}, together with a discrete Poincar\'e inequality (introducing $\COmegaDiscretePoincare$), cf.\ Lemma~\ref{appendix:lemma:discrete-poincare}, we obtain
 \begin{align*}
  & \llangle \hext^n, \tfrac{\pek(\chih(\bm{\beta}))}{\swk(\chih(\bm{\beta}))} \rrangle \\
  &\quad \leq 
  \left \| \tfrac{\pek(\chih(\bm{\beta}))}{\swk(\chih(\bm{\beta})) \chih(\bm{\beta}) } \right\|_{\infty}
  \,
  \left( \left\| h^n \right\| \, \left\| \chih(\bm{\beta}) \right\| + \left\| w_\mathrm{N}^n \right\|_{L^2(\FlowNeumannBoundary)} \, \left\| \chih(\bm{\beta}) \right\|_{L^2(\FlowNeumannBoundary)} \right) \\
  &\quad \leq 
  C\left( C_\mathrm{ND,1}, \CDiscreteTrace,\COmegaDiscretePoincare\right)
  \,
  \left\| \hext^n \right\|_{L^2(\Omega) \times L^2(\FlowNeumannBoundary)}
  \, 
  \left\| \chih(\bm{\beta}) \right\|_{1,\mathcal{T}}
 \end{align*}
 for a constant $C\left( C_\mathrm{ND,1}, \CDiscreteTrace,\COmegaDiscretePoincare\right)>0$
 Hence, by~(A6) and Young's inequality, for the term $T_4$ it holds that
 \begin{align*}
  T_4 
  &\leq 
  \frac{C\left( C_\mathrm{ND,1}, \CDiscreteTrace,\COmegaDiscretePoincare\right)^2}{2c_{\pek/\swk} \absolutepermeabilitymin} \tau \left\| \hext^n \right\|^2_{L^2(\Omega) \times L^2(\FlowNeumannBoundary)}
  +
  \frac{c_{\pek/\swk}}{2} \tau\| \chih(\bm{\beta})  \|_{1,\mathcal{T},\absolutepermeability}^2.
 \end{align*}

\paragraph{Combination of all results.}
 By inserting the estimates for $T_1$, $T_2$, $T_3$, and $T_4$,~\eqref{existence-discrete-solution-aux-1} becomes
 \begin{align}
 \label{existence-discrete-solution-aux-3}
  \llangle \bm{F}(\bm{\beta}), \bm{\beta} \rrangle
  &\geq
    \Bigg(\left\| \bar{B}_\eta\left( \frac{\pek(\chih(\bm{\beta}))}{\swk(\chih(\bm{\beta}))}\right)\right\|_{L^1(\Omega)} + \tfrac{c_{\pek/\swk}}{2} \tau
   \| \chih(\bm{\beta}) \|_{1,\mathcal{T},\absolutepermeability}^2  \\
   \nonumber
   &\qquad +
  \frac{1}{4} \left\| \uh(\bm{\alpha}) \right\|_{\V}^2 
  + \frac{1}{2} \zeta \tau^{-1} \left\| \uh(\bm{\alpha}) - \uh^{n-1} \right\|_{\V}^2
  + \frac{1}{4} \left\| \uh(\bm{\alpha}) - \uh^{n-1} \right\|_{\V}^2 \Bigg)\\
  \nonumber
  &  -
  \Bigg(
  \left\| \bar{B}_\eta\left( \frac{\pek(\chih^{n-1})}{\swk(\chih^{n-1})}\right) \right\|_{L^1(\Omega)}
  +
  \frac{1}{2}
  \left\| \uh^{n-1} \right\|_{\V}^2
  +
  \frac{5}{4}\| \bm{f}^n \|_{\V^\star}^2 \\
  \nonumber
  &\qquad +
  \frac{C\left( C_\mathrm{ND,1}, \CDiscreteTrace,\COmegaDiscretePoincare\right)^2}{2c_{\pek/\swk} \absolutepermeabilitymin} \tau \left\| \hext^n \right\|_{L^2(\Omega) \times L^2(\FlowNeumannBoundary)} \Bigg).
 \end{align}
 Finally, since $\|\cdot\|_{1,\mathcal{T},\absolutepermeability}$ defines a norm on $\Qh$ and~\eqref{existence-discrete-solution-aux-02} holds by induction for $n-1$ if $n\geq 2$ or from (A8$^\star$) for $n=1$, by a corollary of Brouwer's fixed point theorem, cf.\ Lemma~\ref{appendix:lemma:brouwer}, there exists a  $\bm{\beta}\in\mathbb{R}^{d_\mathrm{Q}}$ such that $\bm{F}(\bm{\beta})=\bm{0}$, which implies existence of a solution. The bound~\eqref{existence-discrete-solution-aux-02} for $n$ follows immediately from~\eqref{existence-discrete-solution-aux-3}.
 \end{proof}

\section{Step 3: Limit $h,\tau\rightarrow 0$}\label{section:well-posedness:regularized-continuous-problem}

In the following, we show that the fully-discrete FEM-TPFA discretization, introduced in the previous section, converges to a weak solution of the doubly regularized unsaturated poroelasticity model, i.e., we prove Lemma~\ref{lemma:existence-doubly-regularized}. The proof follows the steps: 1) derive stability results for the fully discrete approximation; 2) define suitable approximations a.e.\ in time using interpolation; 3) deduce stability for those as well; 4) relative compactness arguments are performed yielding a well-defined limit for $h,\tau\rightarrow 0$; 5) the limit is showed to be a weak solution of the doubly regularized model. Throughout the entire section, we assume~(A0)--(A9) and~(ND1) hold true. 

\subsection{Stability estimates for the fully-discrete approximation}

\begin{lemma}[Stability estimate for the primary variables]\label{lemma:discrete-solution-apriori-estimate-1} Let $\tau<\frac{1}{8}$. There exists a constant $\htauStabilityUChi>0$ (independent of $h,\tau,\zeta,\eta$), such that
\begin{align*}
 &\zeta \, \sum_{n} \tau^{-1} \left\| \uh^n - \uh^{n-1} \right\|_{\V}^2
 +
 \underset{n}{\mathrm{sup}} \, \|\uh^n\|_{\V}^2  
 + \sum_{n} \| \uh^n - \uh^{n-1} \|_{\V}^2
 +
 \sum_{n=1}^N \tau \|  \chih^n \|_{1,\mathcal{T}}^2 \\
 &\quad\leq
 \htauStabilityUChi \bigg( C_0, 
 C_\mathrm{ND,1} \| \hext \|_{L^2(0,T;\Q^\star)},
 \| \fext\|_{H^1(0,T;\V^\star)} \bigg),
\end{align*}
where $C_0$ and $C_\mathrm{ND,1}$ are defined in~$\mathrm{(A8}^\star\mathrm{)}_h$ and $\mathrm{(ND1)}$, respectively.
\end{lemma}

\begin{proof}
The proof follows essentially the same steps as in the proof of Lemma~\ref{lemma:existence-discrete-solution}. Therefore, we are quick on similar steps. 
We consider the reduced displacement-pressure formulation~\eqref{ht:kirchhoff-reduced:u}--\eqref{ht:kirchhoff-reduced:p}. We choose $\vh=\uh^n - \uh^{n-1}$ and $\qh=\frac{\pek(\chih^n)}{\swk(\chih^n)}$ as test functions and sum the two equations; note that the second is well-defined as $\swk(\chi)>0$ for all $\chi\in\mathbb{R}$, by (A2). We obtain
\begin{align*}
 &\zeta \tau^{-1} \left\| \uh^n - \uh^{n-1} \right\|_{\V}^2
 +
 a(\uh^n, \uh^n - \uh^{n-1}) \\
 &\quad+
 \llangle \bke(\chih^n) - \bke(\chih^{n-1}), \frac{\pek(\chih^n)}{\swk(\chih^n)} \rrangle
 +
 \tau \llangle \GRAD_h \chih^n, \GRAD_h \frac{\pek(\chih^n)}{\swk(\chih^n)} \rrangle_{\absolutepermeability} \\
 &\quad=
 \llangle \fext^n, \uh^n - \uh^{n-1} \rrangle
 +
 \tau \llangle \hext^n, \frac{\pek(\chih^n)}{\swk(\chih^n)} \rrangle.
\end{align*}
On the left hand side, we employ the binomial identity~\eqref{binomial-identity}, the Legendre transformation, $\bar{B}_\eta$, of $\bar{b}_\eta=\bke\circ \left(\tfrac{\pek}{\swk}\right)^{-1}$, cf.~\eqref{legendre-transform-b-bar} and Lemma~\ref{appendix:lemma:legendre}, and the uniform increase of $\tfrac{\pek}{\swk}$, cf.,~(A4). It holds that
\begin{align*}
 &\zeta \tau^{-1} \left\| \uh^n - \uh^{n-1} \right\|_{\V}^2
 +
 \frac{1}{2} \left( \|\uh^n\|_{\V}^2 - \| \uh^{n-1} \|_{\V}^2 + \| \uh^n - \uh^{n-1} \|_{\V}^2 \right) \\
 &\quad+
 \left\| \bar{B}_\eta\left(\frac{\pek(\chih^n)}{\swk(\chih^n)}\right) \right\|_{L^1(\Omega)} - \left\| \bar{B}_\eta\left(\frac{\pek(\chih^{n-1})}{\swk(\chih^{n-1})}\right) \right\|_{L^1(\Omega)}
 +
 c_{\pek/\swk} \tau \| \chih^n \|_{1,\mathcal{T},\absolutepermeability}^2 \\
 &\quad\leq
 \llangle \fext^n, \uh^n - \uh^{n-1} \rrangle
 +
 \llangle \hext^n, \frac{\pek(\chih^n)}{\swk(\chih^n)} \rrangle.
\end{align*}
Summing over the time steps $1$ to $N$ and rearranging terms, yields
\begin{align*}
 &\zeta \,\sum_{n} \tau^{-1} \left\| \uh^n - \uh^{n-1} \right\|_{\V}^2
 +
 \frac{1}{2} \|\uh^N\|_{\V}^2  + \frac{1}{2} \sum_{n=1}^N \| \uh^n - \uh^{n-1} \|_{\V}^2  \\
 &\quad+
 \left\| \bar{B}_\eta\left(\frac{\pek(\chih^N)}{\swk(\chih^N)}\right)  \right\|_{L^1(\Omega)}
 +
 c_{\pek/\swk} \sum_{n=1}^N \tau \| \chih^n \|_{1,\mathcal{T},\absolutepermeability}^2 \\ 
 &\quad\leq
 \frac{1}{2} \| \uh^{0} \|_{\V}^2
 + \left\| \bar{B}_\eta\left(\frac{\pek(\chih^{0})}{\swk(\chih^{0})}\right)  \right\|_{L^1(\Omega)}\\
 &\quad+
 \sum_{n=1}^N\llangle \fext^n, \uh^n - \uh^{n-1} \rrangle
 + 
 \sum_{n=1}^N \tau \llangle \hext^n, \frac{\pek(\chih^n)}{\swk(\chih^n)} \rrangle.
\end{align*}
It remains to discuss the last two terms on the right hand side. For the first of them, we employ summation by parts, cf.\ Lemma~\ref{appendix:lemma:summation-by-parts}, as well as the Cauchy-Schwarz inequality and Young's inequality:
\begin{align*}
 &\sum_{n=1}^N\llangle \fext^n, \uh^n - \uh^{n-1} \rrangle \\
 &\quad=
 \llangle \fext^N, \uh^N \rrangle - \llangle \fext^1, \uh^0 \rrangle - \sum_{n=1}^{N-1} \llangle \fext^{n+1} - \fext^{n}, \uh^{n} \rrangle \\
 &\quad\leq 
 \| \fext^N\|_{\V^\star}^2 
 + \frac{1}{4} \| \uh^N \|_{\V}^2 
 + \frac{1}{2} \| \fext^1 \|_{\V^\star}^2 
 + \frac{1}{2} \| \uh^0 \|_{\V}^2 
 + \sum_{n=1}^N \tau^{-1} \left\| \fext^n - \fext^{n-1} \right\|^2_{\V^\star}
 + \sum_{n=1}^N \tau \| \uh^{n} \|_{\V}^2.
\end{align*}
The second term is estimated as in the discussion of $T_4$ within the proof of Lemma~\ref{lemma:existence-discrete-solution}. We obtain
\begin{align*}
 &\sum_{n=1}^N \tau \llangle \hext^n, \frac{\pek(\chih^n)}{\swk(\chih^n)} \rrangle  \\
 & \quad \leq
  \frac{C\left( C_\mathrm{ND,1}, \CDiscreteTrace,\COmegaDiscretePoincare\right)^2}{2c_{\pek/\swk} \absolutepermeabilitymin} \sum_{n=1}^N \tau \left\| \hext^n \right\|_{L^2(\Omega)^2 \times L^2(\FlowNeumannBoundary)}
  +
  \frac{c_{\pek/\swk}}{2} \sum_{n=1}^N \tau \| \chih^n \|_{1,\mathcal{T},\absolutepermeability}^2.
\end{align*}
Altogether, after rearranging terms, we obtain
\begin{align*}
 &\frac{\zeta}{2} \sum_{n} \tau^{-1} \left\| \uh^n - \uh^{n-1} \right\|_{\V}^2
 +
 \frac{1}{4} \|\uh^N\|_{\V}^2  + \frac{1}{4} \sum_{n=1}^N \| \uh^n - \uh^{n-1} \|_{\V}^2  \\
 &\quad+
 \left\| \bar{B}_\eta\left(\frac{\pek(\chih^N)}{\swk(\chih^N)}\right)   \right\|_{L^1(\Omega)}
 +
 \frac{c_{\pek/\swk}}{2} \sum_{n=1}^N \tau \| \chih^n \|_{1,\mathcal{T},\absolutepermeability}^2 \\
 &\quad\leq
 \| \uh^{0} \|_{\V}^2
 + \left\| \bar{B}_\eta\left(\frac{\pek(\chih^{0})}{\swk(\chih^{0})}\right) \right\|_{L^1(\Omega)}
 +
 \frac{C\left( C_\mathrm{ND,1}, \CDiscreteTrace,\COmegaDiscretePoincare\right)^2}{2c_{\pek/\swk} \absolutepermeabilitymin} \sum_{n=1}^N \tau \left\| \hext^n \right\|_{L^2(\Omega) \times L^2(\FlowNeumannBoundary)}^2\\ 
 &\quad+
 \| \fext^N\|_{\V^\star}^2 + \| \fext^1 \|_{\V^\star}^2 
 +
 \sum_{n=1}^N \tau^{-1} \left\| \fext^n - \fext^{n-1} \right\|^2_{\V^\star}
 +
 \sum_{n=1}^N \tau \| \fext^n \|_{\V^\star}^2
 +
 2 \sum_{n=1}^N \tau \| \uh^{n} \|_{\V}^2.
\end{align*}
Finally, the last term on the right hand side can be controlled after applying a discrete Gr\"onwall inequality, cf.\ Lemma~\ref{appendix:discrete-gronwall}, using that $2\tau<\frac{1}{4}$. The thesis follows from the assumptions on the regularity of the source terms (A7) (together with a Sobolev embedding) and initial data~(A8$^\star$).
\end{proof}


\begin{lemma}[Stability for the Kirchhoff pressure]\label{lemma:discrete-solution-apriori-estimate-2} 
There exists a constant $\htauStabilityDchiDt>0$ (independent of $h,\tau$) such that
\begin{align*}
    &b_{\chi,m} \sum_{n=1}^N \tau^{-1} \left\| \chih^n - \chih^{n-1} \right\|^2 
  + \left\| \chih^N \right\|_{1,\mathcal{T}}^2  + \sum_{n=1}^N \left\|\chih^n - \chih^{n-1} \right\|_{1,\mathcal{T}}^2  \leq \htauStabilityDchiDt \left(C_0,  \frac{1+\zeta^{-1}}{b_{\chi,\mathrm{m}}} \htauStabilityUChi \right),
\end{align*}
where $\htauStabilityUChi$ is the stability constant from Lemma~\ref{lemma:discrete-solution-apriori-estimate-1}, $b_\mathrm{\chi,m}$ is from the growing condition~$\mathrm{(A1}^\star\mathrm{)}$, and $C_0$ is the bound in $\mathrm{(A8}^\star\mathrm{)}_h$. 
\end{lemma}

\begin{proof}
 We choose $\qh=\chih^n - \chih^{n-1}$ in~\eqref{ht:kirchhoff-reduced:p}. By using the binomial identity~\eqref{binomial-identity} for the diffusion term, we obtain
 \begin{align*}
 &\llangle \bke(\chih^n) - \bke(\chih^{n-1}), \chih^n - \chih^{n-1} \rrangle + \alpha \llangle \swk(\chih^n) \DIV(\uh^n - \uh^{n-1}), \chih^n - \chih^{n-1} \rrangle \\
  &\quad+ \, \frac{\tau}{2} \left( \left\| \chih^n \right\|_{1,\mathcal{T},\absolutepermeability}^2 -  \left\|\chih^{n-1} \right\|_{1,\mathcal{T},\absolutepermeability}^2 + \left\|\chih^n - \chih^{n-1} \right\|_{1,\mathcal{T},\absolutepermeability}^2 \right) \\
  &\quad= \tau \llangle \hext^n, \chih^n - \chih^{n-1} \rrangle.
  \end{align*}
  Dividing by $\tau$ and summing over time steps 1 to $N$, yields
  \begin{align}
  \label{proof:stability-htau-kirchhoff-pressure:aux-1}
  &\sum_{n=1}^N \tau^{-1} \llangle \bke(\chih^n) - \bke(\chih^{n-1}), \chih^n - \chih^{n-1} \rrangle
  + \, \frac{1}{2} \left\| \chih^N \right\|_{1,\mathcal{T},\absolutepermeability}^2  + \frac{1}{2} \sum_{n=1}^N \left\| \chih^n - \chih^{n-1} \right\|_{1,\mathcal{T},\absolutepermeability}^2  \\
  \nonumber
  &\quad= \frac{1}{2} \left\|\chih^{0} \right\|_{1,\mathcal{T},\absolutepermeability}^2 + \sum_{n=1}^N \llangle \hext^n, \chih^n - \chih^{n-1} \rrangle
  -
  \alpha \sum_{n=1}^N \tau^{-1} \llangle \swk(\chih^n) \DIV(\uh^n - \uh^{n-1}), \chih^n - \chih^{n-1} \rrangle.
  \end{align}
  We discuss some of the terms above separately. Employing the growth condition~(A1$^\star$), yields for the first term on the left hand side of~\eqref{proof:stability-htau-kirchhoff-pressure:aux-1}
  \begin{align*}
   \sum_{n=1}^N \tau^{-1} \llangle \bke(\chih^n) - \bke(\chih^{n-1}), \chih^n - \chih^{n-1} \rrangle
  \geq
   b_{\chi,\mathrm{m}}\sum_{n=1}^N \tau^{-1} \left\| \chih^n - \chih^{n-1} \right\|^2.
  \end{align*}
  By employing the Cauchy-Schwarz inequality and Young's inequality, we get for the second term on the right hand side of~\eqref{proof:stability-htau-kirchhoff-pressure:aux-1}
  \begin{align*}
   \sum_{n=1}^N \llangle \hext^n, \chih^n - \chih^{n-1} \rrangle
   &\leq 
   \frac{b_{\chi,\mathrm{m}}}{2} \sum_{n=1}^N \tau^{-1} \left\| \chih^n - \chih^{n-1} \right\|^2 + \frac{1}{2b_{\chi,\mathrm{m}}} \sum_{n=1}^N \tau \left\| \hext^n \right\|_{Q^\star}^2.
  \end{align*}
  Similarly, for the last term on the right hand side of~\eqref{proof:stability-htau-kirchhoff-pressure:aux-1}, we get
  \begin{align*}
   &\alpha \sum_{n=1}^N \tau^{-1} \langle \swk(\chih^n) \DIV(\uh^n - \uh^{n-1}), \chih^n - \chih^{n-1} \rangle\\
   &\quad\leq 
   \frac{b_{\chi,\mathrm{m}}}{4} \sum_{n=1}^N \tau^{-1} \| \chih^n - \chih^{n-1} \|^2
   +
   \frac{\alpha^2}{b_{\chi,\mathrm{m}}} \sum_{n=1}^N \tau^{-1} \| \DIV ( \uh^n - \uh^{n-1}) \|^2.
  \end{align*}
  All in all,~\eqref{proof:stability-htau-kirchhoff-pressure:aux-1} becomes
  \begin{align*}
    &\frac{b_{\chi,m}}{4} \sum_{n=1}^N \tau^{-1} \| \chih^n - \chih^{n-1} \|^2 
  + \, \frac{1}{2} \| \chih^N \|_{1,\mathcal{T},\absolutepermeability}^2  + \frac{1}{2} \sum_{n=1}^N \| \chih^n - \chih^{n-1} \|_{1,\mathcal{T},\absolutepermeability}^2  \\
  &\quad\leq
  \frac{1}{2} \| \chih^{0} \|_{1,\mathcal{T},\absolutepermeability}^2 
   +
   \frac{1}{2b_{\chi,\mathrm{m}}} \sum_{n=1}^N \tau \| \hext^n \|_{Q^\star}^2 
   +
   \frac{\alpha^2}{b_{\chi,\mathrm{m}}} \sum_{n=1}^N \tau^{-1} \| \DIV ( \uh^n - \uh^{n-1}) \|^2.
  \end{align*}
  Finally, the first term on the right hand side is bounded by~$\mathrm{(A8}^\star\mathrm{)}_h$, whereas the last term can be bounded by employing Lemma~\ref{lemma:discrete-solution-apriori-estimate-1}. On the left hand side, we employ~(A6).
\end{proof}

\begin{lemma}[Stability for the Legendre transformation $\Bke$ of $\bke$]\label{lemma:discrete-solution-stabilty:legendre}
 Let $\Bke(z)$ denote the Legendre transformation of $\bk$, cf.~\eqref{legendre-transform-bk}. There exists a constant $\htauStabilityLegendre>0$ (independent of $h,\tau,\eta$), such that
 \begin{align*}
  \underset{n}{\mathrm{sup}}\, \left\| \Bke(\chih^n) \right\|_{L^1(\Omega)}
  \leq 
  \htauStabilityLegendre\left(C_0, \htauStabilityUChi\left( 1 + \zeta^{-1} \right)\right),
 \end{align*}
 where $\htauStabilityUChi$ is the stability constant from Lemma~\ref{lemma:discrete-solution-apriori-estimate-1}, and $C_0$ is the bound in $\mathrm{(A8}^\star\mathrm{)}_h$.
\end{lemma}

\begin{proof}
 Testing~\eqref{ht:kirchhoff-reduced:p} with $\qh=\chih^n$ and employing the properties of the Legendre transformation $\Bke$, cf.\ Lemma~\ref{appendix:lemma:legendre}, yields for all $n$
 \begin{align*}
  &\left\|\Bke(\chih^n)\right\|_{L^1(\Omega)} - \left\|\Bke(\chih^{n-1})\right\|_{L^1(\Omega)} + \tau \left\|\chih^n\right\|_{1,\mathcal{T},\absolutepermeability}^2  \\
  &\quad \leq 
  \tau \llangle \hext^n, \chih^n \rrangle - \alpha \llangle \swk(\chih^n) \DIV (\uh^n - \uh^{n-1}), \chih^n \rrangle.
 \end{align*}
 For the first term on the right hand side, we employ a similar bound as in the discussion of $T_4$ within the proof of Lemma~\ref{lemma:existence-discrete-solution}; for the second term, we employ the Cauchy-Schwarz inequality, a discrete Poincar\'e inequality (introducing $\COmegaDiscretePoincare$), and~(A6). We obtain
 \begin{align*}
  &\left\|\Bke(\chih^n)\right\|_{L^1(\Omega)} - \left\|\Bke(\chih^{n-1})\right\|_{L^1(\Omega)} + \frac{\tau}{2} \|  \chih^n\|_{1,\mathcal{T},\absolutepermeability}^2  \\
  &\qquad \leq 
   \frac{C\left( C_\mathrm{ND,1}, \CDiscreteTrace,\COmegaDiscretePoincare\right)^2}{\absolutepermeabilitymin}
  \tau \| \hext^n \|_{Q^\star}^2 + \frac{\COmegaDiscretePoincare}{\absolutepermeabilitymin}\, \frac{\alpha^2}{ K_\mathrm{dr}} \tau^{-1} \left\| \uh^n - \uh^{n-1} \right\|_{\V}^2.
 \end{align*}
 Finally, summing over time steps $1$ to $N$ and employing Lemma~\ref{lemma:discrete-solution-apriori-estimate-1} and~(A7) proves the assertion.
\end{proof}

\begin{lemma}[Stability for the pore pressure]\label{lemma:discrete-solution-stability:pore-pressure}
 There exists a constant $\htauStabilityPpore>0$ (independent of $h,\tau,\zeta,\eta$), such that
 \begin{align*}
  \sum_{n=1}^N \tau \| \pek(\chih^n) \|^2
  &\leq \htauStabilityPpore \left(\htauStabilityUChi \right),
 \end{align*}
 where $\htauStabilityUChi$ is the stability constant from Lemma~\ref{lemma:discrete-solution-apriori-estimate-1}.
\end{lemma}

\begin{proof} 
 We utilize a standard inf-sup argument  (introducing $\COmegaInfSup$), cf.\ Lemma~\ref{appendix:lemma:inf-sup}. Due to~(D2), there exists a $\vh\in\Vh$ such that
 \begin{align*}
  \| \pek(\chih(\bm{\beta})) \|^2 = \alpha \llangle \pek(\chih(\bm{\beta})), \DIV \vh \rrangle,
  \qquad 
  \| \vh \|_{\V} \leq \COmegaInfSup \| \pek(\chih(\bm{\beta})) \|,
 \end{align*}
 Employing the mechanics equation~\eqref{model:regularized:weak:kirchhoff:u}, we obtain
 \begin{align*}
  \| \pek(\chih(\bm{\beta})) \| \leq \COmegaInfSup \left( \zeta \tau^{-1}  \left\| \uh^n - \uh^{n-1} \right\|_{\V} 
  + \| \uh^n \|_{\V} + \| \fext^n \|_{\V^\star} \right),
 \end{align*}
 and hence,
 \begin{align*}
  \sum_{n=1}^N \tau \| \pek(\chih^n) \|^2
  \leq
  3 \COmegaInfSup^2 \left(
  \zeta^2 \sum_{n=1}^N \tau^{-1} \left\| \uh^n - \uh^{n-1} \right\|_{\V}^2
  +
  \sum_{n=1}^N \tau \| \uh^n \|_{\V}^2
  +
  \sum_{n=1}^N \tau \| \fext^n \|_{\V^\star}^2 \right).
 \end{align*}
 Finally, the assertion follows from Lemma~\ref{lemma:discrete-solution-apriori-estimate-1}, assuming wlog.\ $\zeta$ is bounded from above.
\end{proof}

\begin{lemma}[Stability for the temporal change of $\bk$]\label{lemma:discrete-solution-stability:dtbk}
 There exists a constant $\htauStabilityDbDt>0$ (independent of $h,\tau,\eta$), such that
 \begin{align*}
  &\underset{\{\qh^n\}_n\subset \Qh\setminus\{0\}}{\mathrm{sup}}\, 
  \frac{\sum_{n=1}^N \tau \llangle \frac{\bke(\chih^n) - \bke(\chih^{n-1})}{\tau}, \qh^n \rrangle}{\left(\sum_{n=1}^N \tau \| \qh^n \|_{1,\mathcal{T}}^2 \right)^{1/2}}
  \leq
  \htauStabilityDbDt \left( \htauStabilityUChi \left(1 + \zeta^{-1}\right) \right),
 \end{align*}
 where $\htauStabilityUChi$ is the stability constant from Lemma~\ref{lemma:discrete-solution-apriori-estimate-1}.
\end{lemma}

\begin{proof}
 Let $\{\qh^n\}_n\subset\Qh\setminus\{0\}$ be an arbitrary sequence of test functions. Employ $\qh^n$ as test function for~\eqref{ht:kirchhoff-reduced:p}. Summing over time steps $1$ to $N$ and applying the Cauchy-Schwarz inequality, yields
 \begin{align*}
  &\sum_{n=1}^N \tau \llangle \frac{\bke(\chih^n) - \bke(\chih^{n-1})}{\tau}, \qh^n \rrangle \\
  &\quad\leq
  \left( \frac{\alpha^2}{K_\mathrm{dr}} \sum_{n=1}^N \tau^{-1} \left\| \uh^n - \uh^{n-1} \right\|_{\V}^2 \right)^{1/2}
  \left(\sum_{n=1}^N \tau \| \qh^n \|^2 \right)^{1/2} \\
  &\qquad+
  \left(\sum_{n=1}^N \tau \| \chih^n \|_{1,\mathcal{T},\absolutepermeability}^2 \right)^{1/2}
  \left(\sum_{n=1}^N \tau \|  \qh^n \|_{1,\mathcal{T},\absolutepermeability}^2 \right)^{1/2} \\
  &\qquad+
  (1+\CDiscreteTrace)\COmegaDiscretePoincare \left( \sum_{n=1}^N \tau \| \hext^n \|_{\Q^\star}^2 \right)^{1/2}
  \left(\sum_{n=1}^N \tau \| \qh^n \|_{1,\mathcal{T}}^2 \right)^{1/2}.
 \end{align*}
 For the last term, we employed a discrete trace inequality, cf.\ Lemma~\ref{appendix:lemma:discrete-trace}, and a discrete Poincar\'e inequality, cf.\ Lemma~\ref{appendix:lemma:discrete-poincare}. Finally, utilizing a discrete Poincar\'e inequality for the first term on the right hand side,~(A6), and employing Lemma~\ref{lemma:discrete-solution-apriori-estimate-1}, we prove the assertion with $\htauStabilityDbDt := 3 \sqrt{\htauStabilityUChi}\left( \COmegaDiscretePoincare \frac{\alpha}{\zeta^{1/2} K_\mathrm{dr}^{1/2}} + \absolutepermeabilitymax^{1/2}  + (1+\CDiscreteTrace)\COmegaDiscretePoincare \right)$.
\end{proof}

\subsection{Stability estimates for interpolants in time}\label{section:discrete:interpolation:stability}

Utilizing the discrete-in-time approximations $(\uh^n,\chih^n)_n$, defined by~\eqref{ht:kirchhoff-reduced:u}--\eqref{ht:kirchhoff-reduced:p}, we define continuous-in-time approximations on $(0,T]$ by piecewise constant interpolation
\begin{alignat*}{2}
 \uhtau(t)  &:=\uh^n,   &\ &t\in  (t_{n-1},t_n], \\
 \chihtau(t)&:=\chih^n, &&t\in  (t_{n-1},t_n],
\end{alignat*}
and by piecewise linear interpolation
\begin{alignat}{2}
\label{interpolation:linear-displacement}
 \uhtauHat(t)&:=\uh^{n-1} + \frac{t - t_{n-1}}{\tau} (\uh^{n} - \uh^{n-1}), &\ &t \in  (t_{n-1},t_n], \\
\label{interpolation:linear-pressure}
 \chihtauHat(t)&:=\chih^{n-1} + \frac{t - t_{n-1}}{\tau} (\chih^{n} - \chih^{n-1}), &&t\in (t_{n-1},t_n].
\end{alignat}
We deduce stability for the interpolants from the stability of the fully discrete approximation.

\begin{lemma}[Stability estimate for time interpolants of the mechanical displacement]\label{lemma:interpolants-apriori-estimate-1}
 For all $h,\tau>0$ and $\hat{\tau}\in[0,\tau)$ it holds that
 \begin{align}
 \label{result:stability:interpolation:u:1}
 \zeta \int_0^T \left\| \partial_t \uhtauHat \right\|_{\V}^2\, dt + \| \uhtau \|_{L^\infty(0,T;\V)}^2 
 &\leq
 \htauStabilityUChi, \\
 \label{result:stability:interpolation:u:2}
 \int_0^{T-\hat{\tau}} \| \uhtau(t+\hat{\tau}) - \uhtau(t) \|_{\V}^2 \, dt 
 &\leq
 \htauStabilityUChi \hat{\tau}, \\
 \label{result:stability:interpolation:u:3}
 \| \uhtau - \uhtauHat \|_{L^2(Q_T)}^2 &\leq \htauStabilityUChi\tau,
\end{align}
where $\htauStabilityUChi$ is the stability constant from Lemma~\ref{lemma:discrete-solution-apriori-estimate-1}.
\end{lemma}

\begin{proof}
 The assertion~\eqref{result:stability:interpolation:u:1} follows directly from Lemma~\ref{lemma:discrete-solution-apriori-estimate-1} by definition of the interpolants. Similarly, by definition of the piecewise constant in time interpolation, it holds that
 \begin{align*}
  &\int_0^{T-\hat{\tau}} \| \uhtau(t+\hat{\tau}) - \uhtau(t) \|_{\V}^2\, dt\\
  &\quad=
  \sum_{n=1}^{N-1} \int_{t_{n-1}}^{t_n} \| \uhtau(t+\hat{\tau}) - \uhtau(t) \|_{\V}^2 \, dt
  +
  \int_{t_{N-1}}^{t_N-\hat{\tau}} \| \uhtau(t+\hat{\tau}) - \uhtau(t) \|_{\V}^2 \, dt\\
  &\quad=
  \sum_{n=1}^{N-1} \int_{t_{n}-\hat{\tau}}^{t_n} \| \uh^{n+1} - \uh^n \|_{\V}^2 \, dt\\
  &\quad=
  \hat{\tau} \sum_{n=1}^N \| \uh^{n+1} - \uh^n \|_{\V}^2.
 \end{align*}
 We obtain~\eqref{result:stability:interpolation:u:2} from Lemma~\ref{lemma:discrete-solution-apriori-estimate-1}. By definition of the piecewise constant and piecewise linear interpolation, it holds that 
 \begin{align*}
  \| \uhtau - \uhtauHat \|_{L^2(Q_T)}^2
  &=
  \sum_{n=1}^N \int_{t_{n-1}}^{t^n} \left\| \uh^n - \uh^{n-1} - \tfrac{t - t_{n-1}}{\tau} \left( \uh^{n} - \uh^{n-1} \right) \right\|^2\\
  &=
  \frac{1}{3}\tau   \sum_{n=1}^N  \left\| \uh^n - \uh^{n-1}  \right\|^2.
 \end{align*}
 We conclude~\eqref{result:stability:interpolation:u:3}.
\end{proof}

Analogously, we conclude stability for the interpolants of the Kirchhoff pressure.

\begin{lemma}[Stability estimate for time interpolants of the Kirchhoff pressure]\label{lemma:interpolants-apriori-estimate-2}
For all $h,\tau>0$ and $\hat{\tau}\in[0,\tau)$ it holds that
 \begin{align*}
 \int_0^T \| \chihtau(t) \|_{1,\mathcal{T}}^2 \, dt 
 &\leq
 \htauStabilityUChi, \\
 b_\mathrm{\chi,m} \| \partial_t \chihtauHat \|_{L^2(Q_T)}^2 + \| \chihtau \|_{L^\infty(0,T;L^2(\Omega))}^2 
 &\leq 
 \htauStabilityDchiDt, \\
 \int_0^{T-\hat{\tau}} \| \chihtau(t+\hat{\tau}) - \chihtau(t) \|^2 \, dt  
 &\leq 
 \COmegaDiscretePoincare^2\htauStabilityDchiDt \hat{\tau},\\
 \| \chihtau - \chihtauHat \|_{L^2(Q_T)}^2 &\leq \COmegaDiscretePoincare^2\htauStabilityDchiDt \tau,
\end{align*}
where $\htauStabilityUChi$ and $\htauStabilityDchiDt$ are the stability constants from Lemma~\ref{lemma:discrete-solution-apriori-estimate-1} and Lemma~\ref{lemma:discrete-solution-apriori-estimate-2}, respectively, and $\COmegaDiscretePoincare$ is the discrete Poincar\'e constant, cf.\ Lemma~\ref{appendix:lemma:discrete-poincare}.
\end{lemma}
\begin{proof}
The proof is analogous to the proof Lemma~\ref{lemma:interpolants-apriori-estimate-1}. For the last two estimates in the assertion, a discrete Poincar\'e inequality, cf.\ Lemma~\ref{appendix:lemma:discrete-poincare}, has to be applied before utilizing the stability bound on $\sum_{n=1}^N \| \chih^n - \chih^{n-1} \|_{1,\mathcal{T}}^2$ from Lemma~\ref{lemma:discrete-solution-apriori-estimate-2}.
\end{proof}

Similarly, by definition of the piecewise constant interpolation, we deduce stability for some of the non-linearities used in the model.
\begin{lemma}[Stability estimates for non-linearities evaluated in interpolants]\label{lemma:interpolants-stabilty:non-linearities}
 It holds that
 \begin{align*}
  \left\| \Bke(\chihtau) \right\|_{L^\infty(0,T;L^1(\Omega))} 
  &\leq 
  \htauStabilityLegendre,\\
  \| \pek(\chihtau) \|_{L^2(Q_T)}^2 &\leq \htauStabilityPpore,
 \end{align*}
 where $\htauStabilityLegendre$ and $\htauStabilityPpore$ are the stability constants from Lemma~\ref{lemma:discrete-solution-stabilty:legendre} and Lemma~\ref{lemma:discrete-solution-stability:pore-pressure}, respectively.
\end{lemma}

\begin{lemma}[Stability estimate for the temporal change of $\bk$]\label{lemma:interpolants-stability:dtbk}
 For
 \begin{align*}
  \bar{\lambda}_{h\tau}(t):= \frac{\bke(\chih^n) - \bke(\chih^{n-1})}{\tau}\ t\in(t_{n-1},t_n]
 \end{align*}
 it holds that
 \begin{align*}
  \| \bar{\lambda}_{h\tau} \|_{L^2(0,T;H^{-1}(\Omega)} \leq \COmegaPoincare^{1/2} \htauStabilityDbDt,
 \end{align*}
 where $\htauStabilityDbDt$ is the stability constant from Lemma~\ref{lemma:discrete-solution-stability:dtbk}, and $\COmegaPoincare$ is a Poincar\'e constant.
\end{lemma}

\begin{proof}
 Let $q\in L^2(0,T;Q)$. We define a piecewise constant interpolation in both space and time, and only time by
 \begin{alignat*}{2}
  \qh^n(x,t) &:= \tfrac{1}{\tau} \int_{t_{n-1}}^{t_n} \frac{1}{|K|}\int_K \q \,dx\,dt, 
  &\quad & (x,t)\in K\times(t_{n-1},t_n],\ K\in\mathcal{T}, \\
  q^n(x,t) &:= \tfrac{1}{\tau} \int_{t_{n-1}}^{t_n} \q \,dt,
  && (x,t) \in \Omega \times (t_{n-1},t_n].
 \end{alignat*}
 Then by Lemma~\ref{lemma:discrete-solution-stability:dtbk} it holds that
 \begin{align*}
  \int_0^T \llangle \bar{\lambda}_{h\tau}, \q \rrangle 
  =
  \sum_{n=1}^N \tau \llangle \frac{\bke(\chih^n) - \bke(\chih^{n-1})}{\tau}, \qh^n \rrangle
  \leq 
  \htauStabilityDbDt \left(\sum_{n=1}^N \tau \| \qh^n \|_{1,\mathcal{T}}^2 \right)^{1/2}.
 \end{align*}
 By Lemma~\ref{appendix:lemma:discrete-vs-continuous-gradients}, a (continuous) Poincar\'e inequality  (introducing $\COmegaPoincare$), analogous to Lemma~\ref{appendix:lemma:discrete-poincare}, the triangle inequality and the Cauchy-Schwarz inequality, it holds that
 \begin{align*}
 \sum_{n=1}^N \tau  \| \qh^n \|_{1,\mathcal{T}}^2 
 &\leq 
 \COmegaPoincare \sum_{n=1}^N \tau \| \GRAD \q^n \|^2 \\
 &=
 \COmegaPoincare \sum_{n=1}^N \tau \left\| \tau^{-1} \int_{t_{n-1}}^{t_n} \GRAD q \, dt \right\|^2 \\
 &\leq
 \COmegaPoincare \sum_{n=1}^N \tau^{-1} \left( \int_{t_{n-1}}^{t_n} \left\| \GRAD q  \right\| \, dt \right)^2 \\
 &\leq 
  \COmegaPoincare \sum_{n=1}^N \int_{t_{n-1}}^{t_n} \left\| \GRAD q  \right\|^2  \, dt \\
 &=
 \COmegaPoincare \| q \|_{L^2(0,T;H_0^1(\Omega))}^2,
 \end{align*}
 which concludes the proof.
\end{proof}

\subsection{Relative (weak) compactness for the limit $h,\tau\rightarrow 0$}\label{section:limit-regularized-degenerate-case}

We utilize the stability results from the previous section to conclude relative compactness. We deduce limits for the interpolants which eventually converge towards a weak solution of the doubly regularized unsaturated poroelasticity model, i.e., it fulfils~(W1)$_{\zeta\eta}$--(W4)$_{\zeta\eta}$.

\begin{lemma}[Convergence of the mechanical displacement]\label{lemma:convergence-u} 
We can extract subsequences of $\{\uhtau\}_{h,\tau}$ and $\{\uhtauHat\}_{h,\tau}$ (still denoted like the original sequences), and there exists $\uveta\in L^\infty(0,T;\V)$ with $\partial_t \uveta \in L^2(0,T;\V)$ such that for $h,\tau\rightarrow0$
 \begin{alignat}{2}
 \label{result:convergence-u:1}
  \uhtau     &\rightharpoonup \uveta                       && \text{ in }L^\infty(0,T;\V), \\
 \label{result:convergence-u:2}
  \uhtau     &\rightarrow     \uveta                       && \text{ in }L^2(Q_T), \\
 \label{result:convergence-u:3}
  \uhtauHat     &\rightharpoonup     \uveta                && \text{ in }L^2(0,T;\V), \\
 \label{result:convergence-u:4}
  \partial_t \uhtauHat &\rightharpoonup \partial_t \uveta  && \text{ in }L^2(0,T;\V).
 \end{alignat}
\end{lemma}

\begin{proof} 
 By the Eberlein-$\check{\text{S}}$mulian theorem, cf.\ Lemma~\ref{appendix:lemma:eberlein-smulian}, and Lemma~\ref{lemma:interpolants-apriori-estimate-1}, we obtain directly~\eqref{result:convergence-u:1}. For~\eqref{result:convergence-u:2}, we employ a relaxed Aubin-Lions-Simon type compactness result for Bochner spaces, cf.\ Lemma~\ref{appendix:lemma:aubin-lions}, together with Lemma~\ref{lemma:interpolants-apriori-estimate-1}. Furthermore, by the Eberlein-$\check{\text{S}}$mulian theorem, cf.\ Lemma~\ref{appendix:lemma:eberlein-smulian}, and Lemma~\ref{lemma:interpolants-apriori-estimate-1}, there exists a $\hat{\u}\in L^2(0,T,\V)$ such that up to a subsequence
 \begin{alignat*}{2}
   \uhtauHat &\rightharpoonup \hat{\u} &&\text{ in }L^2(0,T;\V), \\
   \partial_t \uhtauHat &\rightharpoonup \partial_t \hat{\u} &&\text{ in }L^2(0,T;\V).
 \end{alignat*}
 We can identify $\hat{\u}=\uveta$ as follows. Employing the triangle inequality and Lemma~\ref{lemma:interpolants-apriori-estimate-1}, yields
 \begin{align*}
  \| \uhtauHat - \uveta \|_{L^2(Q_T)} 
  &\leq  \| \uhtauHat - \uhtau \|_{L^2(Q_T)} + \| \uhtau - \uveta \|_{L^2(Q_T)} \\
  &\leq  \htauStabilityUChi \tau + \| \uhtau - \uveta \|_{L^2(Q_T)},
 \end{align*}
 which converges to zero for $h,\tau\rightarrow0$. This concludes the proof.
\end{proof}

In order to discuss the limit of the pressure, we utilize techniques employed in the finite volume literature~\cite{Saad2013}. We define a piecewise constant discrete gradient of $\chihtau$ utilizing the dual grid $\mathcal{T}^\star$, cf.\ Definition~\ref{definition:dual-mesh},
 \begin{align*}
  \left(\overline{\GRAD \chi}\right)_{h\tau} := \left\{ \begin{array}{lll} 
                                    d \frac{{{\chih^n}_|}_L  - {{\chih^n}_|}_K}{d_{K|L}}\, \n_{K|L},
                                    &(x,t) \in P_\sigma \times(t_{n-1},t_n],
                                    & K\in\mathcal{T},\ L\in\mathcal{N}(L),\ \sigma=K|L,\\
                                    d \frac{ {{\chih^n}_|}_K }{d_{\sigma,K}}\, \n_{\sigma,K}, 
                                    &(x,t) \in P_{\sigma} \times(t_{n-1},t_n], 
                                    & \sigma\in\mathcal{E}_\mathrm{ext}\cap\mathcal{E}_K,
                                   \end{array} \right.
 \end{align*}
 where $\n_{K|L}$ denotes the outward normal on $K|L\in\mathcal{E}$, pointing towards $L$; and $\n_{\sigma,K}$  denotes the outward normal on $\sigma\in\mathcal{E}_\mathrm{ext}\cap\mathcal{E}_K$, pointing towards $K$.
 
\begin{lemma}[Convergence of the Kirchhoff pressure]\label{lemma:convergence-p}
We can extract a subsequence of $\{\chihtau\}_{h,\tau}$ (still denoted like the original sequences), and there exists $\chiveta\in H^1(0,T;Q)$ such that
  \begin{alignat}{2}
  \label{result:lemma:convergence-chi:1}
  \chihtau &\rightarrow \chiveta                && \text{ in }L^2(Q_T), \\
  \label{result:lemma:convergence-chi:2}
  \left(\overline{\GRAD \chi}\right)_{h\tau} &\rightharpoonup \GRAD \chiveta      && \text{ in }L^2(Q_T),\\
  \label{result:lemma:convergence-chi:3}
  \partial_t \chihtauHat &\rightharpoonup \partial_t \chiveta && \text{ in }L^2(Q_T).
 \end{alignat}
\end{lemma}

\begin{proof}
 Let $\hat{h}\in\mathbb{R}^d$ and $\Omega_{\hat{h}}:=\{ \x\in\Omega \, | \, \x+\hat{h} \in \Omega \}$. Using Lemma~4 from~\cite{Eymard1999}, for all $\qh\in \Qh$ it holds that
 \begin{align*}
  \int_{\Omega_{\hat{h}}} \left\| \qh(\x+\hat{h})  - \qh(\x) \right\|^2 \, d\x 
  \leq 
  C \left\| \qh \right\|_{1,\mathcal{T}}^2 |\hat{h}| \left( |\hat{h}| + |\Omega| \right)
 \end{align*}
 for some $C>0$. Hence, we obtain
 \begin{align*}
  \int_0^T \int_{\Omega_{\hat{h}}} \left\| \chihtau(\x+\hat{h})  - \chihtau(\x) \right\|^2 \, d\x \, dt
  &=
  \sum_{n=1}^N \tau \int_{\Omega_{\hat{h}}} \left\| \chih^n(\x+\hat{h})  - \chih^n(\x) \right\|^2 \, d\x \\
  &\leq 
  |\hat{h}| \left( |\hat{h}| + |\Omega| \right) \sum_{n=1}^N \tau \left\|\chih^n\right\|_{1,\mathcal{T}}^2.
 \end{align*}
 Consequently, by Lemma~\ref{lemma:interpolants-apriori-estimate-2}, $\chihtau$ satisfies a translation property in space and time wrt.\ $L^2(Q_T)$. We conclude by the Riesz-Frechet-Kolmogorov compactness criterion, cf.\ Lemma~\ref{appendix:lemma:kolmogorov}, that there exists a $\chiveta\in L^2(Q_T)$ satisfying~\eqref{result:lemma:convergence-chi:1}. 
 
  By definition of $\left(\overline{\GRAD \chi}\right)_{h\tau}$ and the geometrical identity $|P_\sigma|=d^{-1} |\sigma| d_\sigma$, it holds that
 \begin{align*}
  &\left\| \left(\overline{\GRAD \chi}\right)_{h\tau} \right\|_{L^2(Q_T)}^2 \\
  &\qquad= 
  \sum_{n=1}^N \tau \sum_{\sigma\in\mathcal{E}} \int_{P_\sigma} \left| \left(\overline{\GRAD \chi} \right)_{h\tau} \right|^2 \, dx \\
  &\qquad=
  \sum_{n=1}^N \tau \sum_{K\in\mathcal{T}} \sum_{L\in\mathcal{N}(K)} |P_{K|L}| d^2 \frac{\left|{{\chih^n}_|}_K - {{\chih^n}_|}_L \right|^2}{d_{K|L}^2} +
  \sum_{n=1}^N \tau \sum_{\sigma \in \mathcal{E}_\mathrm{ext}\cap\mathcal{E}_K} |P_\sigma| d^2 \frac{\left|{{\chih^n}_|}_K \right|^2}{d_{\sigma,K}^2} \\
  &\qquad=
  d \sum_{n=1}^N \tau \sum_{\sigma \in \mathcal{E}} \tau_\sigma  \left|\delta_\sigma(\chih^n)\right|^2 \\
  &\qquad=
  d \int_0^T \| \chihtau \|_{1,\mathcal{T}}^2\, dt,
 \end{align*}
 which is uniformly bounded by Lemma~\ref{lemma:interpolants-apriori-estimate-2}. Hence, by the Eberlein-$\check{\text{S}}$mulian theorem, cf.\ Lemma~\ref{appendix:lemma:eberlein-smulian}, there exist a $\g_{\chi}\in L^2(Q_T)$ such that (up to a subsequence)
 \begin{align*}
  \left(\overline{\GRAD \chi}\right)_{h\tau} \rightharpoonup \g_{\chi}\text{ in }L^2(Q_T).
 \end{align*}
 It remains to show that $\g_{\chi}=\GRAD \chiveta$ in the sense of distributions, i.e.,
 \begin{align*}
  \int_0^T \langle \g_\chi, \bm{\varphi} \rangle \, dt + \int_0^T \langle \chiveta, \DIV \bm{\varphi} \rangle \, dt = 0\quad \text{for all } \bm{\varphi}\in C^\infty(Q_T)^d.
 \end{align*}
 For that, we follow an argument in~\cite{Saad2013}. Let $\bm{\varphi} \in C^\infty(Q_T)^d$. As
 \begin{align*}
  \int_0^T \langle \left(\overline{\GRAD \chi}\right)_{h\tau}, \bm{\varphi} \rangle \, dt &\rightarrow \int_0^T \langle \g_\chi, \bm{\varphi} \rangle \, dt, \quad\text{and}\\
  \int_0^T \langle \chihtau, \DIV \bm{\varphi} \rangle \, dt &\rightarrow \int_0^T \langle \chiveta, \DIV  \bm{\varphi} \rangle \, dt
 \end{align*}
 for $h,\tau\rightarrow 0$, it suffices to show that
 \begin{align*}
  \int_0^T \langle \left(\overline{\GRAD \chi}\right)_{h\tau}, \bm{\varphi} \rangle \, dt + \int_0^T \langle \chihtau, \DIV \bm{\varphi} \rangle \, dt \rightarrow 0.
 \end{align*}
 By definition of $\left(\overline{\GRAD \chi}\right)_{h\tau}$ and the construction of $\mathcal{T}^\star$ with $\tfrac{d}{d_\sigma} = \tfrac{|\sigma|}{|P_\sigma|}$ for all $\sigma\in\mathcal{E}$, it holds that
 \begin{align*}
  \int_0^T \langle \left(\overline{\GRAD \chi}\right)_{h\tau}, \bm{\varphi} \rangle \, dt
  &=
  \sum_{n=1}^N \int_{t_{n-1}}^{t_n} \sum_{\sigma \in \mathcal{E}} \int_{P_\sigma}\left(\overline{\GRAD \chi}\right)_{h\tau} \cdot \bm{\varphi} \, dx\, dt \\
  &=
  \sum_{n=1}^N \int_{t_{n-1}}^{t_n} \sum_{K\in\mathcal{T}} \sum_{L\in\mathcal{N}(K)} 
  d \frac{ {{\chih^n}_|}_L - {{\chih^n}_|}_K}{d_{K|L}} \int_{P_{K|L}} \bm{\varphi} \cdot \n_{K|L} \, dx\,dt \\
  &\qquad +
  \sum_{n=1}^N \int_{t_{n-1}}^{t_n} \sum_{\sigma \in \mathcal{E}_\mathrm{ext} \cap \mathcal{E}_K} 
  d \frac{ {{\chih^n}_|}_K}{d_{\sigma,K}} \int_{P_{\sigma}} \bm{\varphi} \cdot \n_{\sigma,K} \, dx\,dt \\
  &=\sum_{n=1}^N \int_{t_{n-1}}^{t_n} \sum_{K\in\mathcal{T}} \sum_{L\in\mathcal{N}(K)} 
  |\sigma| \left( {{\chih^n}_|}_L - {{\chih^n}_|}_K \right) \, \frac{1}{|P_{K|L}|}\int_{P_{K|L}} \bm{\varphi} \cdot \n_{K|L} \, dx\,dt \\
  &\qquad + 
  \sum_{n=1}^N \int_{t_{n-1}}^{t_n} \sum_{\sigma \in \mathcal{E}_\mathrm{ext} \cap \mathcal{E}_K} 
  |\sigma| {{\chih^n}_|}_K \, \frac{1}{|P_\sigma|}\int_{P_{\sigma}} \bm{\varphi} \cdot \n_{\sigma,K} \, dx\,dt.
 \end{align*}
 On the other hand, since $\chihtau$ is constant and hence continuous within each $K\in\mathcal{T}$, it holds that
 \begin{align*}
  &\int_0^T \llangle \chihtau, \DIV \bm{\varphi} \rrangle \,dt \\
  &\qquad=
  \sum_{n=1}^N \int_{t_{n-1}}^{t_n} \sum_{\sigma\in\mathcal{E}} \int_{P_\sigma}\chihtau \DIV \bm{\varphi}\, dx  \, dt  \\
  &\qquad=
  \sum_{n=1}^N \int_{t_{n-1}}^{t_n} 
  \Bigg[
  \sum_{K\in\mathcal{T}} \sum_{L\in\mathcal{N}(K)} \left( {{\chih^n}_|}_K \int_{P_{K|L}\cap K} \DIV \bm{\varphi}\, dx + {{\chih^n}_|}_L \int_{P_{K|L}\cap L} \DIV \bm{\varphi}\, dx \right) \\
  &\qquad\qquad\qquad\qquad+
  \sum_{\sigma\in\mathcal{E}_\mathrm{ext}\cap \mathcal{E}_K} {{\chih^n}_|}_K \int_{P_\sigma\cap K} \DIV \bm{\varphi}\, dx  \Bigg]\, dt  \\
  &\qquad=
  -\sum_{n=1}^N \int_{t_{n-1}}^{t_n} 
  \Bigg[
  \sum_{K\in\mathcal{T}} \sum_{L\in\mathcal{N}(K)} \left( {{\chih^n}_|}_L - {{\chih^n}_|}_K \right) \int_{K|L}  \bm{\varphi}\cdot \n_{K|L}\, ds \\
  &\qquad\qquad\qquad\qquad+
  \sum_{\sigma\in\mathcal{E}_\mathrm{ext}\cap \mathcal{E}_K} {{\chih^n}_|}_K \int_{\sigma}  \bm{\varphi}\cdot \n_{\sigma,K} \, ds \Bigg]\, dt.
 \end{align*}
 As $\bm{\varphi}\in C^\infty(Q_T)^d$ is smooth, there exists a constant $C>0$ such that
 \begin{align*}
  \left| \frac{1}{\tau} \int_{t_{n-1}}^{t_n} \frac{1}{|P_\sigma|} \int_{P_\sigma} \bm{\varphi} \cdot \n_{\sigma} \, dx\,dt - 
  \frac{1}{\tau} \int_{t_{n-1}}^{t_n} \frac{1}{|\sigma|}\int_{\sigma} \bm{\varphi} \cdot \n_{\sigma} \, ds\,dt \right|
  \leq C h.
 \end{align*}
 By abuse of notation, we used $\n_\sigma$ for both $\n_{K|L}$ and $\n_{\sigma,K}$. After all, together with the Cauchy-Schwarz inequality, it holds that
 \begin{align*}
  &\left|\int_0^T \langle \left(\overline{\GRAD \chi}\right)_{h\tau}, \bm{\varphi} \rangle \, dt
  +
  \int_0^T \int_\Omega \chihtau \DIV \bm{\varphi} \,dx\,dt \right|\\
  &\quad \leq 
  C h \sum_{n=1}^N  \tau \left(\sum_{K\in\mathcal{T}} \sum_{L\in\mathcal{N}(K)} 
  |\sigma| \left|{{\chih^n}_|}_L - 
  {{\chih^n}_|}_K \right| 
  +
  \sum_{\sigma\in\mathcal{E}_\mathrm{ext}\cap\mathcal{E}_K} |\sigma| \left| {{\chih^n}_|}_K \right|
  \right) 
  \\
  &\quad \leq 
  C h \left( \sum_{n=1}^N  \tau \left\| \chih^n \right\|_{1,\mathcal{T}}^2 \right)^{1/2} \left( \sum_{n=1}^N  \tau \sum_{\sigma\in\mathcal{E}} |\sigma| d_\sigma \right)^{1/2}.
 \end{align*}
 By Lemma~\ref{lemma:interpolants-apriori-estimate-2} and the regularity assumption on $\mathcal{T}$, convergence towards $0$ follows for $h,\tau\rightarrow 0$. This concludes the proof of~\eqref{result:lemma:convergence-chi:2}.
 
 The proof of~\eqref{result:lemma:convergence-chi:3} is standard and follows mainly from the stability results in Lemma~\ref{lemma:interpolants-apriori-estimate-2} and the Eberlein-$\check{\text{S}}$mulian theorem, cf.\ Lemma~\ref{appendix:lemma:eberlein-smulian}. This concludes the proof.
\end{proof}

The main purpose of the double regularization has been the aim to get control over the non-linear coupling terms, and eventually establish convergence.

\begin{lemma}[Convergence of the coupling terms]\label{lemma:convergence-coupling}
We can extract a subsequence of $\{\chihtau\}_{h,\tau}$ (still denoted like the original sequences) such that
  \begin{alignat}{2}
  \label{result:lemma:convergence-pek}
  \pek(\chihtau) &\rightharpoonup \pek(\chiveta)       &\quad & \text{ in }L^2(Q_T), \\
  \label{result:lemma:convergence-sw-divu}
  \swk(\chihtau)\partial_t \DIV \uhtauHat &\rightharpoonup \swk(\chiveta)\partial_t \DIV \uveta &&\text{ in }L^2(Q_T).
 \end{alignat}
\end{lemma}

\begin{proof}
 By the Eberlein-$\check{\text{S}}$mulian theorem, cf.\ Lemma~\ref{appendix:lemma:eberlein-smulian}, and Lemma~\ref{lemma:interpolants-stabilty:non-linearities}, we can extract a subsequence of $\{\chihtau\}_{h,\tau}$ (still denoted $\{\chihtau\}_{h,\tau}$), and there exists a $\hat{p}\in L^2(Q_T)$ such that
 \begin{align*}
  \pek(\chihtau) \rightharpoonup \hat{p} \text{ in }L^2(Q_T).
 \end{align*}
 We can identify $\hat{p}=\pek(\chiveta)$ as follows. From Lemma~\ref{lemma:convergence-p}, we have $\chihtau\rightarrow \chiveta$ a.e.\ on $Q_T$ for a subsequence (still denoted $\{\chihtau\}_{h,\tau}$). As $\pek$ is continuous by~(A3), it holds that $\pek(\chihtau)\rightarrow \pek(\chiveta)$ a.e.\ on $Q_T$. This concludes~\eqref{result:lemma:convergence-pek}.
 
 The convergence property~\eqref{result:lemma:convergence-sw-divu} follows from the convergence properties of the single contributions. Let $q\in L^2(Q_T)$; it holds that $\swk(\chihtau)q \rightarrow \swk(\chiveta) q$ in $L^2(Q_T)$ (up to a subsequence). Indeed, by Lemma~\ref{lemma:convergence-p}, we have $\chihtau\rightarrow \chiveta$ a.e.\ on $Q_T$ (up to a subsequence); due to~(A2), it holds that $\swk(\chihtau)q \rightarrow \swk(\chiveta)q$ a.e.\ on $Q_T$ and $|\swk(\chihtau)q | \leq |q|$ a.e.; hence, by the dominated convergence theorem $\swk(\chihtau)q \rightarrow \swk(\chiveta) q$ in $L^2(Q_T)$. In particular, it holds that $\swk(\chiveta)q \in L^2(\Omega)$. Moreover from Lemma~\ref{lemma:convergence-u}, we have $\partial_t \DIV \uhtauHat \rightharpoonup \partial_t \DIV \uveta$ in $L^2(Q_T)$. Altogether, we obtain
 \begin{align*}
  &\left| \llangle \swk(\chihtau) \partial_t \DIV \uhtauHat - \swk(\chiveta) \partial_t \DIV \uveta, q \rrangle \right| \\
  &\quad\leq
  \left| \llangle \left(\swk(\chihtau) - \swk(\chiveta)\right) \partial_t \DIV \uhtauHat, q \rrangle \right| +
  \left| \llangle \swk(\chiveta) \left( \partial_t \DIV \uhtauHat - \partial_t \DIV \uveta\right), q \rrangle \right|\\
  &\quad\leq 
  \| \swk(\chihtau)q - \swk(\chiveta)q\| \, \| \partial_t \DIV \uhtauHat \| +
  \left| \llangle \partial_t \DIV \uhtauHat - \partial_t \DIV \uveta, \swk(\chiveta)q \rrangle \right|,
 \end{align*}
 which converges towards $0$ for $h,\tau\rightarrow 0$, due to strong and weak convergence of the single components. 
\end{proof}

\begin{lemma}[Initial conditions for the fluid flow]\label{lemma:convergence-dbdt}
It holds that
 \begin{alignat}{2}
  \label{result:lemma:convergence-dtbk}
  \bar{\lambda}_{h\tau} &\rightharpoonup \partial_t \bke(\chiveta) &&\text{ in }L^2(0,T;Q^\star)
 \end{alignat}
 (up to a subsequence), where $\partial_t \bke(\chiveta)\in L^2(0,T;Q^\star)$ is understood in the sense of~$\mathrm{(W2)_{\zeta\eta}}$.
\end{lemma}

\begin{proof}
 By definition of the Legendre transformation $\Bk$ and its properties, cf.\ Lemma~\ref{appendix:lemma:legendre}, it holds that
 \begin{align*}
  |\bke(x)| \leq \delta \Bke(x) + \underset{|y|\leq \delta^{-1}}{\mathrm{sup}}\, |\bke(y)|,
 \end{align*}
 for all $\delta>0$. Since $\Bke(\chihtau) \in L^\infty(0,T;L^1(\Omega))$ is uniformly bounded by Lemma~\ref{lemma:interpolants-stabilty:non-linearities}, and $\bke$ is continuous by~(A1)$^\star$, it holds that $\|\bke(\chihtau)\|_{L^\infty(0,T;L^1(\Omega))}$ is uniformly bounded. Hence, by the Eberlein-$\check{\text{S}}$mulian theorem, cf.\ Lemma~\ref{appendix:lemma:eberlein-smulian}, we can extract a subsequence of $\{\chihtau\}_{h,\tau}$ (still denoted $\{\chihtau\}_{h,\tau}$), and there exists a $\bk_\chi \in L^\infty(0,T;L^1(\Omega))$ such that
 \begin{align*}
  \bke(\chihtau) \rightharpoonup \bk_\chi\text{ in }L^\infty(0,T;L^1(\Omega)).
 \end{align*} 
 As $\bk$ is continuous by~(A1), and $\chihtau \rightarrow \chiveta$ in $L^2(Q_T)$ (up to a subsequence) by Lemma~\ref{lemma:convergence-p}, it holds that $\bke(\chihtau)\rightarrow \bke(\chiveta)$ a.e.\ on $Q_T$ (up to a subsequence). We conclude $\bk_\chi=\bke(\chiveta)$, which proves
 \begin{align}
  \label{result:lemma:convergence-bk}
  \bke(\chihtau) \rightharpoonup \bke(\chiveta) \text{ in }L^\infty(0,T;L^1(\Omega)).
 \end{align}
 
 By the Eberlein-$\check{\text{S}}$mulian theorem, cf.\ Lemma~\ref{appendix:lemma:eberlein-smulian}, and Lemma~\ref{lemma:interpolants-stability:dtbk}, we can extract a subsequence of $\{\chihtau\}_{h,\tau}$ (still denoted $\{\chihtau\}_{h,\tau}$), and there exists a $\bk_t \in L^2(0,T;Q^\star)$ such that
 \begin{align*}
  \bar{\lambda}_{h\tau} \rightharpoonup \bk_t \text{ in }L^2(0,T;Q^\star).
 \end{align*}
 It remains to show that $\bk_t=\partial_t \bke(\chiveta)$ in the sense of~(W2)$_{\zeta\eta}$. For this, we follow arguments by~\cite{Alt1983} as follows. Let $\q\in L^2(0,T; Q)$ with $\partial_t \q \in L^1(0,T;L^\infty(\Omega))$ and $\q(T)=0$. Due to~\eqref{result:lemma:convergence-bk} it holds that
 \begin{align*}
  \int_0^T \llangle \bke(\chih^0) - \bke(\chi_0), \partial_t q \rrangle \, dt \rightarrow 0,
 \end{align*}
 for $h,\tau\rightarrow0$. Thus, it suffices to show that
 \begin{align*}
  \int_0^T \llangle \bar{\lambda}_{h\tau}, q \rrangle \, dt + \int_0^T \llangle \bke(\chihtau) - \bke(\chih^0), \partial_t q \rrangle \, dt \rightarrow 0,
 \end{align*}
 for $h,\tau\rightarrow0$. By definition of $\bar{\lambda}_{h\tau}$, after applying summation by parts, cf.\ Lemma~\ref{appendix:lemma:summation-by-parts}, we obtain
 \begin{align*}
  &\int_0^T \llangle \bar{\lambda}_{h\tau}, q \rrangle \, dt\\
  &\quad=
  \sum_{n=1}^N \llangle \bke(\chih^n) - \bke(\chih^{n-1}), \tau^{-1} \int_{t_{n-1}}^{t_n} q \, dt \rrangle \\
  &\quad=
  \llangle \bke(\chih^N),\tau^{-1} \int^{T}_{T-\tau} q \, dt \rrangle
  -
  \llangle \bke(\chih^0),\tau^{-1} \int^{\tau}_{0} q \, dt \rrangle \\
  &\qquad-
  \sum_{n=1}^{N-1} \llangle \bke(\chih^n), \tau^{-1} \int_{t_{n}}^{t_{n+1}} q \, dt - \tau^{-1} \int_{t_{n-1}}^{t_n} q \, dt \rrangle\\
  &\quad=
  \llangle \bke(\chih^N) - \bke(\chih^0) ,\tau^{-1} \int^{T}_{T-\tau} q \, dt \rrangle \\
  &\qquad -
  \sum_{n=1}^{N-1} \int_{t_{n-1}}^{t_n} \llangle \bke(\chih^n) - \bke(\chih^0), \frac{\tau^{-1} \int_{t_{n}}^{t_{n+1}} q \, dt - \tau^{-1} \int_{t_{n-1}}^{t_n} q \, dt}{\tau} \rrangle \, d\tilde{t} \\
  &\quad \rightarrow
  0 - \int_0^T \llangle \bke(\chiveta) - \bke(\chi_0), \partial_t q \rrangle \, dt,
  \end{align*}
  for $h,\tau\rightarrow0$, due to the smoothness of $q$ and the convergence properties of $\bke(\chihtau)$. This concludes the proof. 
\end{proof}

\begin{lemma}[Initial conditions for the mechanical displacement]\label{lemma:htau:initial-conditions-u}
 The limit $\uveta\in H^1(0,T;\V)$ from Lemma~\ref{lemma:convergence-u} satisfies~$\mathrm{(W3)_{\zeta\eta}}$.
\end{lemma}

\begin{proof}
 Let $\v \in H^1(0,T;\V)$ with $\v(T)=\bm{0}$. We obtain, using the same calculations as in the proof of Lemma~\ref{lemma:convergence-dbdt},
 \begin{align*}
  \int_0^T a(\partial_t \uhtauHat, \v) \, dt
  =
  a\left(\uh^N - \uh^0, \tau^{-1} \int_{T-\tau}^T \v \, dt \right)
  - \int_0^{T-\tau} a\left( \uhtau - \uh^0, \partial_t \vhtauHatTest\right),
 \end{align*}
 where
 \begin{align*}
  \vhtauHatTest(t) = \tau^{-1} \int_{t_{n-2}}^{t_{n-1}} \v \, dt + 
  \frac{t - t_{n-1}}{\tau} \left( \tau^{-1} \int_{t_{n-1}}^{t_n} \v \, dt - \tau^{-1} \int_{t_{n-2}}^{t_{n-1}} \v \, dt \right), \quad t\in (t_{n-1},t_n].
 \end{align*}
 By construction of $\uh^0$ it holds that $\uh^0 \rightharpoonup \u_0$ in $L^2(0,T;\V)$. Furthermore, by Lemma~\ref{lemma:convergence-u}, it holds that $\uhtau \rightharpoonup \uveta$ in $L^2(0,T;\V)$ and $\partial_t \uhtauHat \rightharpoonup \partial_t \uveta$ in $L^2(0,T;\V)$ (up to subsequences). Hence, for $h,\tau\rightarrow 0$, we obtain
 \begin{align*}
  \int_0^T a(\partial_t \uveta, \v) \, dt
  =
  - \int_0^{T} a\left( \uveta - \u_0 , \partial_t \v\right),
 \end{align*}
 and thereby~(W3)$_{\zeta\eta}$.
\end{proof}

\subsection{Identifying a weak solution for $h,\tau\rightarrow 0$}

Finally, we show the limit $(\uveta,\chiveta)$, introduced in the previous section, is a weak solution of the doubly regularized unsaturated poroelasticity model, cf. Definition~\ref{definition:regularization:weak-solution}.

\begin{lemma}[Limit satisfies (W1)$_{\zeta\eta}$--(W4)$_{\zeta\eta}$]\label{lemma:limit-satisfies-w1-w3}
 The limit $(\uveta,\chiveta)$ introduced in the previous section is a weak solution to the doubly regularized unsaturated poroelasticity model, cf.\ Definition~\ref{definition:regularization:weak-solution}.
\end{lemma}

\begin{proof}
 The limit $(\uveta,\chiveta)$ satisfies~(W1)$_{\zeta\eta}$--(W3)$_{\zeta\eta}$ by Lemma~\ref{lemma:convergence-u}, Lemma~\ref{lemma:convergence-p} Lemma~\ref{lemma:convergence-coupling}, and Lemma~\ref{lemma:htau:initial-conditions-u}. It remains to show~(W4)$_{\zeta\eta}$, i.e., that $(\uveta,\chiveta)$ satisfies the balance equations~\eqref{model:regularized:weak:kirchhoff:u}--\eqref{model:regularized:weak:kirchhoff:p}. We first consider sufficiently smooth test functions and then use a density argument. Let $(\v,q)\in L^2(0,T;\V \cap C^\infty(\Omega)^d) \times L^2(0,T;Q  \cap C^\infty(\Omega))$. For given mesh $\mathcal{T}$, we define spatial projection and interpolation operators, respectively, by
\begin{alignat}{4}
 \label{projection-operator-vh}
 \Pi_{\Vh} \,&:\, \V \cap C^\infty(\Omega) \rightarrow \Vh, &\ &\mathrm{s.t.\quad } &\llangle \Pi_{\Vh} \v , \vh \rrangle &= \llangle \v, \vh \rrangle&\ &\text{for all } \vh\in\Vh,\\
 \label{interpolation-operator-qh}
 \mathcal{I}_{\Qh} \,&:\, \Q \cap C^\infty(\Omega) \rightarrow \Qh, &\ &\mathrm{s.t.\ } &{{\mathcal{I}_{\Qh}q}_|}_K &= q(\x_K)\ &\ &\text{for all } K\in\mathcal{T}.
\end{alignat}
Using those, we define piecewise-constant-in-time interpolants of $(\v,q)$
\begin{alignat}{6}
\label{smooth-interpolation-vh}
 \bar{\v}_{h\tau}(t) &:= \vh^n,&\ &t\in(t_{n-1},t_n],&\qquad&  \vh^n &:= \Pi_{\Vh}\v^n,\qquad &&\v^n&:=  \tau^{-1}\int_{t_{n-1}}^{t_n} \v \, dt, \\
\label{smooth-interpolation-qh}
 \bar{q}_{h\tau}(t)  &:= \qh^n,&\ &t\in(t_{n-1},t_n],&\qquad& \qh^n &:= \mathcal{I}_{\Qh} q^n,\qquad &&q^n &:=  \tau^{-1}\int_{t_{n-1}}^{t_n} q \, dt.
\end{alignat}
Similarly, let
\begin{alignat*}{2}
 \bar{\f}_\mathrm{ext,\tau}(t) &:= \fext^n, &\ &t\in(t_{n-1},t_n],\\
 \bar{h}_\mathrm{ext,\tau}(t)  &:= \hext^n,  &\ &t\in(t_{n-1},t_n].
\end{alignat*}
Combining classical results, based on the assumed regularity~(A7), for $h,\tau\rightarrow 0$ it holds that
\begin{alignat*}{2}
  \bar{\v}_{h\tau} &\rightarrow \v  &\ &\text{in } L^2(0,T;\V), \\
  \bar{q}_{h\tau}  &\rightarrow q   &\ &\text{in } L^2(0,T;\Q),\\
  \bar{\f}_\mathrm{ext,\tau}    &\rightarrow \fext  &\ &\text{in } L^2(0,T;\V^\star),\\
  \bar{h}_\mathrm{ext,\tau}     &\rightarrow \hext  &\ &\text{in } L^2(0,T;\Q^\star).
\end{alignat*}
We choose $\vh=\vh^n$ and $\qh=\qh^n$ as test functions in~\eqref{ht:kirchhoff-reduced:u}--\eqref{ht:kirchhoff-reduced:p}, multiply both equations with $\tau$ and sum over all time steps $1$ to $N$; we obtain
 \begin{align}
 \label{proof:limit-weak-solution:aux-1:start}
 \int_0^T  \Big[ \lambdavisco \llangle \partial_t\DIV \uhtauHat, \DIV \bar{\v}_{h\tau} \rrangle 
 + a(\uhtau, \bar{\v}_{h\tau})
 - \alpha \llangle \pek(\chihtau), \DIV \bar{\v}_{h\tau} \rrangle \Big]& \, dt = \int_0^T \llangle \bar{\f}_\mathrm{ext,\tau}, \bar{\v}_{h\tau} \rrangle \, dt, \\
 \int_0^T \bigg[ 
  \llangle \bar{\lambda}_{h\tau}, \bar{q}_{h\tau} \rrangle \, +
  \alpha \llangle \swk(\chihtau) \partial_t \DIV \uhtauHat, \bar{q}_{h\tau} \rrangle 
 + \llangle \GRAD_h \chihtau, \GRAD_h \bar{q}_{h\tau} \rrangle_{\absolutepermeability}
 \bigg]& \, dt = \int_0^T \llangle \bar{h}_\mathrm{ext,\tau}, \bar{q}_{h\tau} \rrangle\, dt.
 \label{proof:limit-weak-solution:aux-1:end}
 \end{align}
 For most terms we can apply the fact that the product of weakly and strongly convergent sequences converge to the product of their limits. The only term needing discussion is the diffusion term in the flow equation. For this, we follow an argument by~\cite{Saad2013}. 
 
 By definition of the continuous extension of the discrete gradient $\left(\overline{\GRAD \chi}\right)_{h\tau}$, it holds that
 \begin{align*}
  &\int_0^T \llangle \GRAD_h \chihtau, \GRAD_h \bar{q}_{h\tau} \rrangle_{\absolutepermeability} \, dt \\
  &\quad=
  \sum_{n=1}^N \tau \sum_{K\in\mathcal{T}} \sum_{L\in\mathcal{N}(K)} \tau_{K|L} \{ \absolutepermeability \}_{K|L} \left( {{\chih^n}_|}_K - {{\chih^n}_|}_L \right) \, \left( \q^n(\x_K) - \q^n(\x_L) \right)\\
  &\qquad+
  \sum_{n=1}^N \tau \sum_{\sigma \in \mathcal{E}_\mathrm{ext} \cap \mathcal{E}_K } \tau_\sigma \{ \absolutepermeability \}_{\sigma} {{\chih^n}_|}_K  \, \q^n(\x_K) \\
  &\quad=
  \sum_{n=1}^N \tau \sum_{K\in\mathcal{T}} \sum_{L\in\mathcal{N}(K)} |P_{K|L}| \{ \absolutepermeability \}_{K|L} {{\left(\overline{\GRAD \chi}\right)_{h\tau}}_|}_{P_{K|L} \times (t_{n-1},t_n]} \cdot \n_{L|K} \, \tfrac{1}{d_{K|L}}\left( \q^n(\x_K) - \q^n(\x_L) \right)\\
  &\qquad+
  \sum_{n=1}^N \tau \sum_{\sigma\in\mathcal{E}_\mathrm{ext}\cap\mathcal{E}_K} |P_\sigma| \{ \absolutepermeability \}_{\sigma} {{\left(\overline{\GRAD \chi}\right)_{h\tau}}_|}_{P_\sigma \times (t_{n-1},t_n]} \cdot \left(-\n_{\sigma,K}\right) \, \tfrac{1}{d_{\sigma,K}}\q^n(\x_K).
 \end{align*}
 By the mean value theorem, there exists an $\x_{K|L} \in P_{K|L}$ on the line between $\x_K$ and $\x_L$, and an $\x_{\sigma} \in P_\sigma$ on the line between $\x_K$ and the closest point of $\x_K$ on $\sigma$ such that
 \begin{align*}
  \tfrac{1}{d_{K|L}} \left(q^n(\x_K) - q^n(\x_L) \right) &= \GRAD q^n(\x_{K|L}) \cdot \n_{L|K}, \\
  \tfrac{1}{d_{\sigma,K}} q^n(\x_K) &= \GRAD q^n(\x_\sigma) \cdot \left(-\n_{\sigma,K}\right).
 \end{align*}
 Due to identical alignment of the discrete gradients, it holds that
 \begin{align*}
  &\int_0^T \llangle \GRAD_h \chihtau, \GRAD_h \bar{q}_{h\tau} \rrangle_{\absolutepermeability} \, dt \\
  &\quad=
  \sum_{n=1}^N \tau \sum_{K\in\mathcal{T}} \sum_{L\in\mathcal{N}(K)} |P_{K|L}| \{ \absolutepermeability \}_{K|L} {{\left(\overline{\GRAD \chi}\right)_{h\tau}}_|}_{P_{K|L} \times (t_{n-1},t_n]} \cdot \GRAD \q^n(\x_{K|L})\\
  &\qquad+
  \sum_{n=1}^N \tau \sum_{\sigma\in\mathcal{E}_\mathrm{ext}\cap\mathcal{E}_K} |P_\sigma| \{ \absolutepermeability \}_{\sigma} {{\left(\overline{\GRAD \chi}\right)_{h\tau}}_|}_{P_\sigma \times (t_{n-1},t_n]} \cdot \GRAD q(\x_\sigma).
 \end{align*}
 We define the piecewise constant functions
 \begin{alignat*}{2}
  \left(\overline{\GRAD q}\right)_{h\tau}(x,t) &= \GRAD q^n(\x_{\sigma}), &\quad& (\x,t)\in P_{\sigma}\times (t_{n-1},t_n],\ \sigma\in\mathcal{E},\\[5pt]
  \{ \absolutepermeability \}_\mathcal{T}(\x) &= \{ \absolutepermeability \}_{\sigma}, && \x\in P_{\sigma},\ \sigma\in\mathcal{E}.
 \end{alignat*}
 We obtain for $h,\tau\rightarrow 0$
 \begin{align*}
  &\int_0^T \llangle \GRAD_h \chihtau, \GRAD_h \bar{q}_{h\tau} \rrangle_{\absolutepermeability} \, dt \\
  &\qquad=
  \int_0^T \int_{\Omega} \{ \absolutepermeability\}_\mathcal{T} \left(\overline{\GRAD \chi}\right)_{h\tau}\cdot \left(\overline{\GRAD q}\right)_{h\tau} \,dx\, dt
  \rightarrow \int_0^T \int_{\Omega} \absolutepermeability \GRAD \chiveta \cdot \GRAD q \,dx\, dt.
 \end{align*}
 Indeed, due to sufficient regularity, it holds that $\left(\overline{\GRAD q}\right)_{h\tau}\rightarrow \GRAD q$ a.e., and also in $L^2(Q_T)$ by the dominated convergence theorem. Furthermore, it holds that $\{ \absolutepermeability\}_{\mathcal{T}} \rightarrow \absolutepermeability$ in $L^\infty(Q_T)$, and by Lemma~\ref{lemma:convergence-p}, it holds that $\left(\overline{\GRAD \chi}\right)_{h\tau}\rightharpoonup \GRAD \chiveta$ in $L^2(Q_T)$. That suffices to discuss the product.
 
 All in all, together with the convergence properties of the test functions $\bar{\v}_{h\tau}$, $\bar{q}_{h\tau}$, the source terms $\bar{\f}_\mathrm{ext,\tau}$, $\bar{h}_\mathrm{ext,\tau}$, and the interpolants for the fully discrete approximations (cf.\ Lemma~\ref{lemma:convergence-u}, Lemma~\ref{lemma:convergence-p}, Lemma~\ref{lemma:convergence-coupling} and Lemma~\ref{lemma:convergence-dbdt}), we conclude that~\eqref{proof:limit-weak-solution:aux-1:start}--\eqref{proof:limit-weak-solution:aux-1:end} converges to~\eqref{model:regularized:weak:kirchhoff:u}--\eqref{model:regularized:weak:kirchhoff:p}, evaluated in $(\uveta,\chiveta)$ and tested with $(\v,q)\in L^2(0,T;\V \cap C^\infty(\Omega)^d) \times L^2(0,T;Q  \cap C^\infty(\Omega))$. Finally, a density argument yields the final result.
\end{proof}

\section{Step 4: Increased regularity in a non-degenerate case}\label{section:improved-regularity-nondegenerate-case}

In the following, further stability estimates for the fully-discrete problem are derived, allowing for showing that the limit $(\uveta,\chiveta)$ introduced in the previous section also satisfies (W5)$_{\zeta\eta}$ and (W6)$_{\zeta\eta}$, i.e., we prove Lemma~\ref{lemma:existence-doubly-regularized-increased-regularity}. For this, non-degeneracy assumptions are required. For compact presentation throughout the entire section, we assume~(A0)--(A9) and~(ND1)--(ND2) hold true, and we define $\uh^{-1}:=\uh^0$.

\subsection{Improved stability estimates for fully-discrete approximation}

\begin{lemma}[Improved stability estimate for the structural velocity]\label{lemma:stability-discrete-time-derivative-mechanics}
 There exists a constant $\htauStabilityDuDt>0$ (independent of $h,\tau$), satisfying
 \begin{align*}
  &\zeta \, \underset{n}{\mathrm{sup}}\, \left\|\tau^{-1} (\uh^n - \uh^{n-1} ) \right\|_{\V}^2
  + \sum_{n=1}^N \tau^{-1}  \| \uh^n - \uh^{n-1} \|_{\V}^2 \\
  &\quad+ \zeta \,\sum_{n=1}^N \left\|\tau^{-1} (\uh^n - \uh^{n-1} ) - \tau^{-1} (\uh^{n-1} - \uh^{n-2} ) \right\|_{\V}^2 \\
  &\quad + \sum_{n=1}^N \tau^{-1} \| \pek(\chih^n) - \pek(\chih^{n-1}) \|^2 \\
  &\quad\leq 
  \htauStabilityDuDt\left( \| \partial_t \fext \|_{L^2(Q_T)}^2, \frac{C_\mathrm{ND,2}}{b_\mathrm{\chi,m}} \, \htauStabilityDchiDt \right),
 \end{align*}
 where $\htauStabilityDchiDt$ is the stability constant from Lemma~\ref{lemma:discrete-solution-apriori-estimate-2}, $C_\mathrm{ND,2}$ comes from the non-degeneracy condition $\mathrm{(ND2)}$, and $b_\mathrm{\chi,m}$ comes from the growth condition $\mathrm{(A1}^\star\mathrm{)}$.
\end{lemma}

\begin{proof}
 First we observe, that the compatibility condition for the initial conditions~\eqref{compatibility:initial-conditions} is equivalent to the mechanics equation~\eqref{ht:kirchhoff-reduced:u} for $n=0$, since $\uh^0 - \uh^{-1}=\bm{0}$. This allows for considering the difference of the mechanics equation~\eqref{ht:kirchhoff-reduced:u} at time steps $n$ and $n-1$, $n\geq1$,
 \begin{align*}
   &\zeta a \left(\tau^{-1} (\uh^n - \uh^{n-1}) - (\uh^{n-1} - \uh^{n-2}), \vh\right)
   +
   a(\uh^n - \uh^{n-1},\vh) \\
   &\qquad - 
   \alpha \langle \pek(\chih^n) - \pek(\chih^{n-1}), \DIV \vh \rangle 
   = \langle \fext^n - \fext^{n-1}, \vh \rangle \quad \text{for all } \vh\in\Vh.
 \end{align*}
 By testing with $\vh=\tau^{-1}(\uh^n - \uh^{n-1})$ and using the binomial identity~\eqref{binomial-identity}, we obtain
 \begin{align*}
  &\frac{\zeta}{2} \Big( 
  \left\|\tau^{-1} (\uh^n - \uh^{n-1} ) \right\|_{\V}^2
  -\left\| \tau^{-1} (\uh^{n-1} - \uh^{n-2} ) \right\|_{\V}^2 \\
  &\quad+ \left\|\tau^{-1} (\uh^n - \uh^{n-1} ) - \tau^{-1} (\uh^{n-1} - \uh^{n-2} ) \right\|_{\V}^2 \Big)
  + \tau^{-1}  \| \uh^n - \uh^{n-1} \|_{\V}^2 \\
  &\quad=
  \tau^{-1} \llangle \fext^n - \fext^{n-1}, \uh^n - \uh^{n-1} \rrangle
  +
  \alpha
  \tau^{-1} \llangle \pek(\chih^n) - \pek(\chih^{n-1}), \uh^n - \uh^{n-1} \rrangle.
 \end{align*}
 Summing over $n\in\{1,...,N\}$, yields after applying the Cauchy-Schwarz inequality and Young's inequality for the right hand side terms
 \begin{align} 
\nonumber
  &\frac{\zeta}{2} 
   \left\|\tau^{-1} (\uh^n - \uh^{n-1} ) \right\|_{\V}^2
  + \frac{1}{2} \sum_{n=1}^N \tau^{-1}  \| \uh^n - \uh^{n-1} \|_{\V}^2 \\
  \nonumber
  &\quad+ \frac{1}{2} \sum_{n=1}^N \left\|\tau^{-1} (\uh^n - \uh^{n-1} ) - \tau^{-1} (\uh^{n-1} - \uh^{n-2} ) \right\|_{\V}^2 \\
  &\quad\leq 
  \sum_{n=1}^N \tau^{-1} \left\| \fext^n - \fext^{n-1} \right\|_{\V^\star}^2
  +\frac{\alpha^2}{K_\mathrm{dr}}  \sum_{n=1}^N \tau^{-1} \left\| \pek(\chih^n) - \pek(\chih^{n-1}) \right\|^2.
 \label{proof:stability-discrete-time-derivative-mechanics:aux-1:end}
 \end{align}
 Due to~(ND2), $\pek=\pek(\chi)$ is Lipschitz continuous. Therefore, by Lemma~\ref{lemma:discrete-solution-apriori-estimate-2} it holds that
 \begin{align*}
  \sum_{n=1}^N \tau^{-1} \| \pek(\chih^n) - \pek(\chih^{n-1}) \|^2 \leq C_\mathrm{ND,2}^2 \frac{\htauStabilityDchiDt}{b_\mathrm{\chi,m}},
 \end{align*}
 which together with~\eqref{proof:stability-discrete-time-derivative-mechanics:aux-1:end} concludes the proof.
\end{proof}

\begin{lemma}[Consequence for the structural acceleration]\label{lemma:stability-d2udtt}
 There exists a constant $\htauStabilityDuDtt>0$ (independent of $h,\tau$), such that
  \begin{align*}
  & \sum_{n=1}^N \tau \left\| \frac{\uh^n - 2\uh^{n-1} + \uh^{n-2} }{\tau^2} \right\|_{\V}^2 \leq \htauStabilityDuDtt\left( \zeta^{-2}\htauStabilityDuDt\right),
 \end{align*}
 where $\htauStabilityDuDt$ is the stability constant from Lemma~\ref{lemma:stability-discrete-time-derivative-mechanics}.
\end{lemma}

\begin{proof}
 Let $\{\vh^n \}_n\subset \Vh \setminus \{ \bm{0} \}$ be an arbitrary sequence of test functions. Consider the difference of~\eqref{ht:kirchhoff-reduced:u} at $n$ and $n-1$, $n\geq 1$; it holds that
 \begin{align*}
  &\tau^{-1} \,\zeta a\left(\uh^n - 2\uh^{n-1} + \uh^{n-2}, \vh^n \right) \\
  &\quad=
  \llangle \fext^n - \fext^{n-1}, \vh^n \rrangle
  - a\left( \uh^n - \uh^{n-1}, \vh^n \right)
  + \alpha \llangle \pek(\chih^n) - \pek(\chih^{n-1}), \DIV \vh^n \rrangle.
 \end{align*}
 Summing over $n\in\{1,...,N\}$, applying the Cauchy-Schwarz inequality and Lemma~\ref{lemma:stability-discrete-time-derivative-mechanics}, yields
  \begin{align*}
  &\underset{\{\vh^n\}_n\subset \Vh\setminus\{\bm{0}\}}{\mathrm{sup}} \, \frac{\zeta\,  \sum_{n=1}^N \tau^{-1}  a(\uh^n - 2\uh^{n-1} + \uh^{n-2}, \vh^n ) }{\left( \sum_{n=1}^N \tau \|\vh^n \|_{\V}^2\right)^{1/2}} \leq 3 \sqrt{\htauStabilityDuDt}.
 \end{align*}
 As $\|\cdot\|_{\V}^2 = a(\cdot,\cdot)$, we obtain equivalence of norms, which concludes the proof.
\end{proof}

\subsection{Improved stability estimates for interpolants in time}
We define piecewise linear interpolations of the discrete structural velocities and the pore pressure. For $t\in(t_{n-1},t_n]$, $n\geq 1$, let
\begin{align}
\label{interpolation:linear-structural-velocity}
 \vhtauHat(t) &:= \frac{\uh^{n-1} - \uh^{n-2}}{\tau} + \frac{t - t_{n-1}}{\tau} \, \frac{\uh^{n} - 2\uh^{n-1}+ \uh^{n-2}}{\tau}, \\
\label{interpolation:linear-pore-pressure} 
 \porepressurehtauHat(t) &:= \pek(\chih^{n-1}) + \frac{t - t_{n-1}}{\tau} \left( \pek(\chih^n) - \pek(\chih^{n-1}) \right).
\end{align}
Note that $\partial_t \uhtauHat$ defines the piecewise constant analog of $\vhtauHat$. 
Stability bounds are obtained as direct consequence of Lemma~\ref{lemma:stability-discrete-time-derivative-mechanics} and Lemma~\ref{lemma:stability-d2udtt}.

\begin{lemma}[Stability estimate for interpolations of the structural velocity]\label{lemma:stability-interpolant:du-dt}
Let $\uhtauHat$ and $\vhtauHat$, as defined by~\eqref{interpolation:linear-displacement} and~\eqref{interpolation:linear-structural-velocity}. For all $h,\tau>0$ and $\hat{\tau}\in(0,\tau)$, it holds that
 \begin{align}
 \label{result:stability-interpolant:du-dt:1}
  \| \partial_t \uhtauHat \|_{L^2(0,T;\V)}^2  &\leq \htauStabilityDuDt, \\
 \label{result:stability-interpolant:du-dt:2}
  \int_0^{T-\hat{\tau}} \| \partial_t \uhtauHat(t + \hat{\tau}) - \partial_t \uhtauHat(t) \|^2 \, dt &\leq \COmegaPoincareKorn^2\htauStabilityDuDt \hat{\tau},\\
 \label{result:stability-interpolant:du-dt:3}
  \| \vhtauHat \|_{L^2(0,T;\V)}^2 &\leq 2\htauStabilityDuDt, \\
 \label{result:stability-interpolant:du-dt:4}
  \| \vhtauHat - \partial_t \uhtauHat \|_{L^2(Q_T)}^2 &\leq 
  \frac{\COmegaPoincareKorn^2 \htauStabilityDuDtt}{\zeta} \tau^2, \\
 \label{result:stability-interpolant:du-dt:5}
  \| \partial_t \vhtauHat \|_{L^2(0,T;\V)}^2 &\leq \frac{\htauStabilityDuDtt}{\zeta},
 \end{align}
 where $\htauStabilityDuDt$ and $\htauStabilityDuDtt$ are the stability constants from Lemma~\ref{lemma:stability-discrete-time-derivative-mechanics} and Lemma~\ref{lemma:stability-d2udtt}, respectively, and $\COmegaPoincareKorn$ is the product of the Poincar\'e and the Korn constants.
\end{lemma}

\begin{proof}
By construction, it holds that
\begin{align*}
 \| \partial_t \uhtauHat \|_{L^2(0,T;\V)}^2 &= \sum_{n=1}^N \tau^{-1} \left\| \uh^n - \uh^{n-1} \right\|_{\V}^2.
\end{align*}
Hence,~\eqref{result:stability-interpolant:du-dt:1} follows directly from Lemma~\ref{lemma:stability-discrete-time-derivative-mechanics}. The time-translation property~\eqref{result:stability-interpolant:du-dt:2} follows from the fact that $\partial_t \uhtauHat$ is piecewise constant. Analogous to the proof of Lemma~\ref{lemma:interpolants-apriori-estimate-1}, one can show
\begin{align*}
  &\int_0^{T-\tau} \| \partial_t \uhtauHat(t + \tau) - \partial_t \uhtauHat(t) \|^2 \, dt 
  =
  \hat{\tau} \sum_{n=1}^N \left\| \tau^{-1} \left( \uh^n - \uh^{n-1} \right) - \tau^{-1} \left( \uh^{n-1} - \uh^{n-2} \right) \right\|^2.
\end{align*}
Finally, after using a Poincar\'e inequality and Korn's inequality,~\eqref{result:stability-interpolant:du-dt:2} follows from Lemma~\ref{lemma:stability-discrete-time-derivative-mechanics}.

In order to show~\eqref{result:stability-interpolant:du-dt:3}, we expand the integral over the time interval. By definition of $\vhtauHat$, it holds that
\begin{align*}
 &\| \vhtauHat \|_{L^2(0,T;\V)}^2\\
 &\quad=
 \sum_{n=1}^N \tau^{-2} \int_{t_{n-1}}^{t_n} \left\| \left(\uh^{n-1} - \uh^{n-2}\right) + \frac{t - t_{n-1}}{\tau} \left( \uh^n - 2\uh^{n-1} + \uh^{n-2} \right) \right\|_{\V}^2 \, dt\\
 &\quad\leq
 2\sum_{n=2}^N \tau^{-2} \int_{t_{n-1}}^{t_n} \left(\frac{t-t_n}{\tau}\right)^2 \left\| \uh^{n-1} - \uh^{n-2}\right\|_{\V}^2 \, dt \\
 &\qquad+
 2\sum_{n=2}^N \tau^{-2} \int_{t_{n-1}}^{t_n} \left(\frac{t-t_{n-1}}{\tau}\right)^2 \left\| \uh^{n} - \uh^{n-1}\right\|_{\V}^2 \, dt \\
 &\quad\leq
 \frac{4}{3}
 \sum_{n=1}^N \tau^{-1} \left\| \uh^n - \uh^{n-1} \right\|_{\V}^2.
\end{align*}
Hence,~\eqref{result:stability-interpolant:du-dt:3} follows by Lemma~\ref{lemma:stability-discrete-time-derivative-mechanics}. In order to show~\eqref{result:stability-interpolant:du-dt:4}, we again expand the integral over the time interval. By definition of $\vhtauHat$ and $\uhtauHat$, it holds that
\begin{align*}
 &\| \vhtauHat - \partial_t \uhtauHat \|_{L^2(Q_T)}^2 \\
 &\quad=
 \sum_{1=2}^N \int_{t_{n-1}}^{t_n}  \left(1 - \frac{t-t_{n-1}}{\tau} \right)^2 \tau^{-2} \left\| \uh^{n} - 2\uh^{n-1} + \uh^{n-2} \right\|^2 \, dt \\
 &\quad=
 \frac{1}{3} \sum_{n=1}^N \tau^{-1} \left\| \uh^{n} - 2\uh^{n-1} + \uh^{n-2} \right\|^2.
\end{align*}
Hence, after employing a Poincar\'e inequality and Korn's inequality,~\eqref{result:stability-interpolant:du-dt:4} follows from Lemma~\ref{lemma:stability-d2udtt}. Finally,~\eqref{result:stability-interpolant:du-dt:5} follows directly from Lemma~\ref{lemma:stability-d2udtt}, since
\begin{align*}
 \| \partial_t \vhtauHat \|_{L^2(0,T;\V)}^2 &= \sum_{n=1}^N \tau \left\| \frac{\uh^n - 2\uh^{n-1} + \uh^{n-2}}{\tau^2}\right\|_{\V}^2.
\end{align*}
\end{proof}

\begin{lemma}[Stability result for the interpolation of the pore pressure]\label{lemma:stability-interpolant:time-derivative-porepressure}
For $\porepressurehtauHat$ defined in~\eqref{interpolation:linear-pore-pressure}. It holds that
\begin{align*}
  \left\| \partial_t \porepressurehtauHat \right\|_{L^2(Q_T)}^2 &\leq  \htauStabilityDuDt, \\
  \left\| \porepressurehtauHat - \pek(\chihtau) \right\|_{L^2(Q_T)}^2 &\leq \htauStabilityDuDt \tau^2,
\end{align*}
where $\htauStabilityDuDt$ is the stability constant from Lemma~\ref{lemma:stability-discrete-time-derivative-mechanics}.
\end{lemma}

\begin{proof}
 By construction, it holds that
 \begin{align*}
  \left\| \partial_t \porepressurehtauHat \right\|_{L^2(Q_T)}^2 &= \sum_{n=1}^N \tau^{-1} \| \pek(\chih^n) - \pek(\chih^{n-1}) \|^2,\qquad \text{and} \\
  \left\| \porepressurehtauHat - \pek(\chihtau) \right\|_{L^2(Q_T)}^2 &= \sum_{n=1}^N \frac{\tau}{3} \| \pek(\chih^n) - \pek(\chih^{n-1}) \|^2,
 \end{align*}
 where the second result follows by expanding time integration. Hence, the assertion follows directly from the stability result for the discrete time derivative of the pore pressure, cf.\ Lemma~\ref{lemma:stability-discrete-time-derivative-mechanics}.
\end{proof}

\subsection{More relative (weak) compactness for $h,\tau\rightarrow 0$}
The previous stability results allow for analyzing the limit in relation to $(\uveta,\chiveta)$.

\begin{lemma}[Convergence of the structural velocity and acceleration]\label{lemma:convergence-du-dt}
We can extract subsequences of $\{\uhtauHat\}_{h,\tau}$ and $\{\vhtauHat\}_{h,\tau}$ (still denoted like the original sequences) such that $\partial_t \uveta, \partial_{tt} \uveta \in L^2(0,T;\V)$ and 
 \begin{alignat}{2}
 \label{result:convergence-du-dt:1}
  \partial_t \uhtauHat &\rightharpoonup \partial_t \uveta, &\quad& \text{in }L^2(0,T;\V), \\
 \label{result:convergence-du-dt:2}
  \partial_t \vhtauHat &\rightharpoonup \partial_{tt} \uveta, &&\text{in }L^2(0,T;\V).
 \end{alignat}
\end{lemma}

\begin{proof}
 The convergence result~\eqref{result:convergence-du-dt:1} follows from the stability result~\eqref{result:stability-interpolant:du-dt:1}, the Eberlein-$\check{\text{S}}$mulian theorem, cf.\ Lemma~\ref{appendix:lemma:eberlein-smulian}, and the fact that $\uhtauHat \rightharpoonup \uveta$ in $L^2(0,T;\V)$, cf.\ Lemma~\ref{lemma:convergence-u}. Furthermore, due to the additional translation property~\eqref{result:stability-interpolant:du-dt:2}, by employing a relaxed Aubin-Lions-Simon type compactness result for Bochner spaces, cf.\ Lemma~\ref{appendix:lemma:aubin-lions}, we can extract a further subsequence (still denoted the same)
 \begin{align}
 \label{proof:convergence-du-dt:aux-1}
  \partial_t \uhtauHat \rightarrow \partial_t \uveta,\quad \text{in }L^2(Q_T).
 \end{align}
 
 Using the stability result~\eqref{result:stability-interpolant:du-dt:5}, by the Eberlein-$\check{\text{S}}$mulian theorem, cf.\ Lemma~\ref{appendix:lemma:eberlein-smulian}, we can extract a subsequence (still denoted the same) such that $\partial_t \vhtauHat \rightharpoonup \u_{tt}$ in $L^2(0,T;\V)$ for some $\u_{tt}\in L^2(0,T;\V)$. It holds that $\u_{tt}=\partial_{tt} \uveta$ if also $\vhtauHat \rightharpoonup \partial_t \uveta$ in $L^2(0,T;\V)$. From the stability result~\eqref{result:stability-interpolant:du-dt:3}, and the Eberlein-$\check{\text{S}}$mulian theorem, cf.\ Lemma~\ref{appendix:lemma:eberlein-smulian}, there exists a $\u_t\in L^2(0,T;\V)$ such that $\vhtauHat \rightharpoonup \u_t$ in $L^2(0,T;\V)$ (up to a subsequence). Employing the triangle inequality, yields
 \begin{align*}
  \| \vhtauHat - \partial_t \uveta \|_{L^2(Q_T)} \leq \| \vhtauHat - \partial_t \uhtauHat \|_{L^2(Q_T)} + \| \partial_t \uhtauHat - \partial_t \uveta \|_{L^2(Q_T)}.
 \end{align*}
 Hence, due to~\eqref{result:stability-interpolant:du-dt:4} and~\eqref{proof:convergence-du-dt:aux-1}, it holds that $\u_t = \partial_t \uveta$, and consequently $\u_{tt}=\partial_{tt} \uveta$, concluding the proof.
\end{proof}

\begin{lemma}[Convergence of the time derivative of the pore pressure]\label{lemma:convergence-dporepressure-dt}
 There exists a subsequence of $\{\porepressurehtauHat\}_{h,\tau}$ (still denoted $\{ \porepressurehtauHat \}_{h,\tau}$) satisfying
 \begin{align*}
  \partial_t \porepressurehtauHat \rightharpoonup \partial_t \pek(\chiveta), \ \text{in }L^2(Q_T).
 \end{align*}
\end{lemma}
\begin{proof}
 By Lemma~\ref{lemma:convergence-p}, we have $\chihtau \rightarrow \chiveta$ in $L^2(Q_T)$ (up to a subsequence). Hence, also $\pek(\chihtau)\rightarrow \pek(\chiveta)$ in $L^2(Q_T)$ (up to a subsequence). From Lemma~\ref{lemma:stability-interpolant:time-derivative-porepressure}, it follows $\porepressurehtauHat \rightarrow \pek(\chiveta)$ and $\partial_t \porepressurehtauHat \rightharpoonup p_t$ for some $p_t\in L^2(Q_T)$ (up to a subsequence). Consequently, $p_t = \partial_t \pek(\chiveta)$, which concludes the proof.
\end{proof}

\subsection{Identifying a weak solution with increased regularity for $h,\tau\rightarrow 0$}

Finally, we show the limit $(\uveta,\chi)$, derived in Section~\ref{section:limit-regularized-degenerate-case}, also satisfies (W5)$_{\zeta\eta}$--(W6)$_{\zeta\eta}$, i.e., $(\uveta,\chiveta)$ is a weak solution with increased regularity for the doubly regularized unsaturated poroelasticity model, cf.\ Definition~\ref{definition:regularization:weak-solution}.

\begin{lemma}[Limit satisfies (W1)$_{\zeta\eta}$--(W6)$_{\zeta\eta}$]\label{lemma:limit-satisfies-w1-w5}
 The limit $(\uveta,\chiveta)$, derived in Section~\ref{section:limit-regularized-degenerate-case}, is a weak solution with increased regularity for the doubly regularized unsaturated poroelasticity model, cf.\ Definition~\ref{definition:regularization:weak-solution}.
\end{lemma}

\begin{proof}
 The limit $(\uveta,\chiveta)$ satisfies~(W1)$_{\zeta\eta}$--(W4)$_{\zeta\eta}$ by Lemma~\ref{lemma:limit-satisfies-w1-w3}. Furthermore,~(W5)$_{\zeta\eta}$ follows directly from Lemma~\ref{lemma:convergence-du-dt} and Lemma~\ref{lemma:convergence-dporepressure-dt}. In order to show~(W6)$_{\zeta\eta}$, let $\v \in L^2(0,T;\V\cap C^\infty(\Omega)^d)$. We utilize $\bar{\v}_{h\tau}$ and $\vh^n$, as introduced in~\eqref{projection-operator-vh} and~\eqref{smooth-interpolation-vh}, respectively; again it holds that
 \begin{align}
 \label{proof:limit-satisfies-w1-w5:aux-0}
  \bar{\v}_{h\tau} \rightarrow \v \quad \text{in }L^2(0,T;\V).
 \end{align}
 We consider the difference of the mechanics equation~\eqref{ht:kirchhoff-reduced:u} at time steps $n$ and $n-1$, $n\geq1$, tested with $\vh=\vh^n$; we obtain 
 \begin{align*}
  &\zeta\tau^{-1}a(\uh^n - 2\uh^{n-1} + \uh^{n-2}, \vh^n)
  +
  a(\uh^n - \uh^{n-1}, \vh^n) \\
  &\quad
  -
  \alpha \llangle \pek(\chih^n) - \pek(\chih^{n-1}), \DIV \vh^n \rrangle
  =
  \llangle \fext^n - \fext^{n-1}, \vh^n \rrangle.
 \end{align*}
%
%
%
%
 Summing over $n\in\{1,...,N\}$, and employing the definitions of $\bar{\v}_{h\tau}$, $\vhtauHat$, $\uhtauHat$, and $\porepressurehtauHat$, yields
\begin{align}
 \label{proof:limit-satisfies-w1-w5:aux-1}
 &\int_0^T \Big[ 
 \zeta a(\partial_t \vhtauHat, \bar{\v}_{h\tau})
 +
 a(\partial_t \uhtauHat , \bar{\v}_{h\tau})
  -
  \alpha  \llangle \partial_t \porepressurehtauHat, \DIV \bar{\v}_{h\tau} \rrangle \Big] \, dt
  =
 \int_0^T \llangle \partial_t \hat{\f}_\tau, \bar{\v}_{h\tau} \rrangle \, dt,
 \end{align}
 where $\hat{\f}_\tau$ denotes the piecewise linear interpolation of the discrete values $\{ \fext^n \}_n$
 \begin{align*}
  \hat{\f}_\mathrm{ext,\tau} (t) := \fext^{n-1} + \frac{t - t_{n-1}}{\tau} \left( \fext^n - \fext^{n-1} \right),\quad t\in(t_{n-1},t_n].
 \end{align*}
 It holds $\hat{\f}_\mathrm{ext,\tau} \rightarrow \fext$ in $L^2(0,T;\V^\star)$ and also $\partial_t \hat{\f}_\mathrm{ext,\tau} \rightharpoonup \partial_t \fext$ in $L^2(0,T;\V^\star)$, for $\tau\rightarrow 0$. Hence, together with the weak convergence properties of $\vhtauHat$, $\uhtauHat$ and $\porepressurehtauHat$, cf.\  Lemma~\ref{lemma:convergence-du-dt} and Lemma~\ref{lemma:convergence-dporepressure-dt}, and the strong convergence properties of the test function~$\bar{\v}_{h\tau}$, cf.~\eqref{proof:limit-satisfies-w1-w5:aux-0}, we conclude that
 \begin{align*}
  \int_0^T  \Big[ \zeta a(\partial_{tt} \uveta, \v ) 
 + a(\partial_t \uveta, \v)
 - \alpha \llangle \partial_t \pek(\chiveta), \DIV \v \rrangle \Big]& \, dt = \int_0^T \llangle \partial_t \fext, \v \rrangle \, dt, 
 \end{align*}
 for all $\v\in L^2(0,T;\V\cap C^\infty(\Omega)^d)$. A density argument yields the final result.
\end{proof}

\section{Step 5: Limit $\zeta\rightarrow0$}\label{section:pure-consolidation}

In this section, we prove Lemma~\ref{lemma:existence-simply-regularized}, i.e., the existence of weak solution to the simply regularized unsaturated poroelasticity model, cf.\ Definition~\ref{definition:simply-regularized-weak-solution}. For this we utilize the fact that under the assumptions of Lemma~\ref{lemma:existence-simply-regularized}, there exists weak solution, $(\uv,\chiv)$, with increased regularity for the doubly regularized unsaturated poroelasticity model, cf.\ Definition~\ref{definition:regularization:weak-solution}. We show that $\{(\uv,\chiv)\}_{\zeta}$ has a limit for $\zeta\rightarrow 0$, which is a weak solution to the simply regularized unsaturated poroelasticity model, i.e., it satisfies (W1)$_\eta$--(W4)$_\eta$ for $\zeta=0$. For this, we employ compactness arguments. The central uniform stability bound is derived utilizing~(W6)$_{\zeta\eta}$ and the non-degeneracy condition~(ND3). Throughout the entire section, we assume~(A0)--(A9) and~(ND1)--(ND3) hold true.

\subsection{Stability estimates independent of $\zeta$}

The key ingredients for the subsequent discussion are stability estimates, which are independent of $\zeta$. In Section~\ref{section:discrete:interpolation:stability}, some derived stability results are independent of $\zeta$; they remain true for weak limits. In particular, there exists a constant $C=C\left(\htauStabilityUChi, \htauStabilityPpore \right)>0$ (independent of $\zeta>0$ and $\eta>0$), such that
\begin{align}
\label{eps-regularization:stability-bound}
 \left\| \uv \right\|_{L^\infty(0,T;\V)}^2 + \left\| \porepressure(\chiv) \right\|_{L^2(Q_T)}^2
 \leq C,
\end{align}
where Lemma~\ref{lemma:interpolants-apriori-estimate-1} and Lemma~\ref{lemma:convergence-u} yield stability for the displacement, and Lemma~\ref{lemma:interpolants-stabilty:non-linearities} and Lemma~\ref{lemma:convergence-coupling} yield stability for the pore pressure. Further stability bounds can be obtained by exploiting the continuous nature of the balance equations and the time derivative of the mechanics equation. The following stability estimate is the essential step.

\begin{lemma}[Stability for the primal variables]\label{lemma:alternative-step-2}
 There exists a constant $\epsStabilityDuDt>0$ (independent of $\zeta$ and $\eta$), such that
 \begin{align*}
  &\zeta \left\| \partial_t \uv \right\|_{L^\infty(0,T;\V)}^2 
  +
  \left\| \partial_t \uv \right\|_{L^2(0,T;\V)}^2  
  +
  \left\| \GRAD \chiv \right\|_{L^\infty(0,T;L^2(\Omega))}^2 \\
  &\quad\leq
  \epsStabilityDuDt \left(
  C_0,
  \left\| \partial_t \fext \right\|_{L^2(0,T;\V^\star)}^2,
  \left\| \hext \right\|_{H^1(0,T;\Q^\star)}^2
  \right),
 \end{align*}
 where $C_0$ comes from~$\mathrm{(A8}^\star\mathrm{)}$.
\end{lemma}
 
\begin{proof}
 Consider the flow equation~\eqref{model:regularized:weak:kirchhoff:p} and the mechanics equation differentiated in time~\eqref{model:regularized:weak:kirchhoff:dtu}, tested with $q=\partial_t \chiv$ and $\v=\partial_t \uv$, respectively. Summing both equation yields
 \begin{align}
\nonumber
  &\zeta \int_0^T a(\partial_{tt} \uv, \partial_t \uv) \, dt
  +
  \int_0^T \llangle \absolutepermeability \GRAD \chiv, \GRAD \partial_t \chiv \rrangle \, dt \\
\nonumber
  &\qquad+
  \left\| \partial_t \uv \right\|_{L^2(0,T;\V)}^2
  +
  \int_0^T \llangle \partial_t \bke(\chiv), \partial_t \chiv \rrangle
  +
  \alpha \int_0^T \llangle \swk \partial_t \chiv - \partial_t \pek, \partial_t \DIV \uv \rrangle \\
  \label{proof:alternative-step-2:aux:1} 
  &\qquad=
  \int_0^T \llangle \partial_t \fext, \partial_t \uv \rrangle \, dt
  +
  \int_0^T \llangle \hext, \partial_t \chiv \rrangle \, dt. 
 \end{align}
 We discuss the individual terms separately. For the first two terms on the left hand side of~\eqref{proof:alternative-step-2:aux:1}, we employ the fundamental theorem of calculus
 \begin{align*}
  &\zeta \int_0^T a(\partial_{tt} \uv, \partial_t \uv) \, dt
  +
  \int_0^T \llangle \absolutepermeability \GRAD \chiv, \GRAD \partial_t \chiv \rrangle \, dt \\
  &\quad=
  \frac{\zeta}{2} \left\| \partial_t \uv(T) \right\|_{L^2(0,T;\V)}^2
  +
  \frac{1}{2} \left( \llangle \absolutepermeability \GRAD \chiv(T), \GRAD \chiv(T) \rrangle - \llangle \absolutepermeability \GRAD \chiv(0), \GRAD \chiv(0) \rrangle \right),
 \end{align*}
 where we used that $\partial_t \uv(0) = \bm{0}$, following from the temporal derivative of the mechanics equations~\eqref{model:regularized:weak:kirchhoff:dtu} and the compatibility condition for the initial conditions~(A9).
 
 For the remaining terms on the left hand side of~\eqref{proof:alternative-step-2:aux:1}, we employ the fact that $\bke$ is increasing with $\bke'\geq\bk'$, that $a(\v,\v) \geq K_\mathrm{dr} \| \DIV \v \|^2$ for all $\v\in\V$ with $K_\mathrm{dr}=\tfrac{2\mu}{d} + \lambda$, and~(ND3). Starting with a binomial identity, we obtain
 \begin{align*}
  &\left\| \partial_t \uv \right\|_{L^2(0,T;\V)}^2
  +
  \int_0^T \llangle \partial_t \bke(\chiv), \partial_t \chiv \rrangle
  +
  \alpha \int_0^T \llangle \swk \partial_t \chiv - \partial_t \pek, \partial_t \DIV \uv \rrangle \\
  &\qquad =
  \left\| \partial_t \uv \right\|_{L^2(0,T;\V)}^2 
  - 
  \frac{\alpha^2}{4} \int_0^T \int_\Omega 
  \left( \frac{\sw(\chiv)}{\pek'(\chiv)} - 1 \right)^2 \frac{\left(\pek'(\chiv)\right)^2}{\bke'(\chiv)} \, \left|\partial_t \DIV \uv \right|^2 \,dx\,dt
  \\
  &\qquad\qquad +
  \int_0^T \int_\Omega \left[
  \left(\partial_t \bke \partial_t \chiv\right)^{1/2} + \frac{\alpha}{2} \left(\swk \partial_t \chiv - \partial_t \pek \right) \left( \partial_t \bke \partial_t \chiv\right)^{-1/2} \partial_t \DIV \uv 
  \right]^2 \, dx \, dt \\
  &\qquad \geq 
  \left(1- C_\mathrm{ND,3}\right) \left\| \partial_t \uv \right\|_{L^2(0,T;\V)}^2.
 \end{align*}
 For the first term on the right hand side of~\eqref{proof:alternative-step-2:aux:1}, we apply the Cauchy-Schwarz inequality and Young's inequality
 \begin{align*}
  &\int_0^T \llangle \partial_t \fext, \partial_t \uv \rrangle \, dt  \leq 
  \frac{1}{2(1-C_\mathrm{ND,3})} \left\| \partial_t \fext \right\|_{L^2(0,T;\V^\star)}^2
  + \frac{1-C_\mathrm{ND,3}}{2} \left\| \partial_t \uv \right\|_{L^2(0,T;\V)}^2.
 \end{align*}
 For the second term on the right hand side of~\eqref{proof:alternative-step-2:aux:1}, we apply integration by parts, a Cauchy-Schwarz inequality and Young's inequality, a Poincar\'e inequality (introducing the Poincar\'e constant $\COmegaPoincare$) and a Sobolev embedding (introducing the constant $C_\mathrm{T,Sob}$), as well as~(A6). All in all, we obtain
 \begin{align*}
  &\int_0^T \llangle \hext, \partial_t \chiv \rrangle \, dt \\
  &\quad=
  \llangle \hext(T),\chiv(T) \rrangle - \llangle \hext(0), \chi(0) \rrangle  
  -\int_0^T \llangle \partial_t \hext, \chiv \rrangle \, dt \\
  &\quad\leq
  \frac{\COmegaPoincare^2}{\absolutepermeabilitymin} \left( \left\| \hext(T) \right\|^2 + \left\| \hext(0) \right\|^2 + \left\| \partial_t \hext \right\|_{L^2(Q_T)}^2 \right) \\
  &\qquad\quad
  + \frac{\absolutepermeabilitymin}{4\COmegaPoincare^2} \left( \left\| \chiv(T) \right\|^2 + \left\| \chiv(0) \right\|^2 + \left\| \chiv \right\|_{L^2(Q_T)}^2 \right)\\
  &\quad\leq 
  \frac{3\left(C_\mathrm{T,Sob}\COmegaPoincare\right)^2}{\absolutepermeabilitymin} \left\| \hext \right\|_{H^1(0,T;\Q^\star)}^2 \\
  &\qquad\quad+ \frac{1}{4} \bigg( \llangle \absolutepermeability \GRAD \chiv(T), \GRAD \chiv(T) \rrangle + \llangle \absolutepermeability \GRAD \chiv(0), \GRAD \chiv(0) \rrangle + \int_0^T \llangle \absolutepermeability \GRAD \chiv, \GRAD \chiv \rrangle \, dt \bigg).
 \end{align*}
 Altogether,~\eqref{proof:alternative-step-2:aux:1} becomes
 \begin{align*}
  &\frac{\zeta}{2} \left\| \partial_t \uv(T) \right\|_{L^2(0,T;\V)}^2 
  +
  \frac{1}{4} \llangle \absolutepermeability \GRAD \chiv(T), \GRAD \chiv(T) \rrangle 
  +
  \frac{1- C_\mathrm{ND,3}}{2} \left\| \partial_t \uv \right\|_{L^2(0,T;\V)}^2  \\
  &\quad\leq 
  \frac{3}{4} \llangle \absolutepermeability \GRAD \chiv(0), \GRAD \chiv(0) \rrangle
  +
  \frac{1}{2(1-C_\mathrm{ND,3})} \left\| \partial_t \fext \right\|_{L^2(0,T;\V^\star)}^2 \\
  &\qquad +
  \frac{3\left(C_\mathrm{T,Sob}\COmegaPoincare\right)^2}{\absolutepermeabilitymin} \left\| \hext \right\|_{H^1(0,T;\Q^\star)}^2
  + 
  \frac{1}{4} \int_0^T \llangle \absolutepermeability \GRAD \chiv, \GRAD \chiv \rrangle \, dt. 
 \end{align*}
 Applying a Gr\"onwall inequality proves the assertion under the given assumptions.
\end{proof}

The last stability estimate allows for deriving further stability estimates.

\begin{lemma}[Stability for the Legendre transformation of $\bke$]
There exists a constant $\epsStabilityLegendre>0$ (independent of $\zeta,\eta$), such that
 \begin{align*}
  &\left\| \Bke(\chiv) \right\|_{L^\infty(0,T;L^1(\Omega))} 
  \leq 
  \epsStabilityLegendre \left( C_0, \epsStabilityDuDt\right),
 \end{align*}
 where $\epsStabilityDuDt$ is the stability constant from Lemma~\ref{lemma:alternative-step-2}, and $C_0$ is the stability constant from $\mathrm{(A8}^\star\mathrm{)}$.
\end{lemma}

\begin{proof}
 Testing the flow equation~\eqref{model:regularized:weak:kirchhoff:p} with $q=\chiv$, yields
 \begin{align*}
  \int_0^T \llangle \partial_t \bke(\chiv), \chiv \rrangle \, dt + \int_0^T \left\| \GRAD \chiv \right\|_{\absolutepermeability}^2 \, dt
  =
  \int_0^T \llangle \hext, \chiv \rrangle \, dt - \alpha \int_0^T \llangle \sw \partial_t \DIV \uv, \chiv \rrangle \, dt.
 \end{align*}
 For the first term on the left hand side, we apply an identity for Legendre transformations, cf.~\cite{Alt1983},
 \begin{align*}
  \int_0^T \llangle \partial_t \bke(\chiv), \chiv \rrangle \, dt = \left\| \Bke(\chiv(T)) \right\|_{L^1(\Omega)} - \left\| \Bke(\chi_0) \right\|_{L^1(\Omega)},
 \end{align*}
 where $\Bke$ is the Legendre transformation for $\bke$. On the right hand side, we apply the Cauchy-Schwarz inequality, Young's inequality, a Poincar\'e inequality  (introducing $\COmegaPoincare$) and~(A6), and obtain
 \begin{align*}
  &\left\| \Bke(\chiv(T)) \right\|_{L^1(\Omega)} 
  + \frac{1}{2} \int_0^T \left\| \GRAD \chiv \right\|_{\absolutepermeability}^2 \, dt \\
  &\quad \leq 
  \left\| \Bke(\chi_0) \right\|_{L^1(\Omega)}
  +
  \frac{\COmegaPoincare^2}{\absolutepermeabilitymin} \left( \left\| \hext \right\|_{L^2(0,T;\Q^\star)}^2 + \alpha^2 \left\| \partial_t \DIV \uv \right\|_{L^2(Q_T)}^2 \right).
 \end{align*}
 Finally, the thesis follows from Lemma~\ref{lemma:alternative-step-2}.
\end{proof}

\begin{lemma}[Stability for the temporal change of $\bke$]\label{lemma:regularization:stability:dkbe-dt}
 There exists a constant $\epsStabilityDbDt>0$ (independent of $\zeta,\eta$), such that
 \begin{align*}
  \underset{ 0 \neq q \in L^2(0,T;Q) }{\mathrm{sup}}\, 
  \frac{ \int_0^T \llangle \partial_t \bke(\chiv), q \rrangle \, dt }{ \| \GRAD q \|_{L^2(Q_T)} }
  \leq
  \epsStabilityDbDt\left( \epsStabilityDuDt \right),
 \end{align*}
 where $\epsStabilityDuDt$ is the stability constant from Lemma~\ref{lemma:alternative-step-2}.
\end{lemma}

\begin{proof}
 The proof is analog to the proof of Lemma~\ref{lemma:discrete-solution-stability:dtbk}. However, this time, we exploit
 \begin{align*}
  \left\| \partial_t \DIV \uv \right\|_{L^2(Q_T)} \leq \frac{1}{K_\mathrm{dr}^{1/2}} \left\| \partial_t \uv \right\|_{L^2(0,T;\V)} \leq \frac{\epsStabilityDuDt}{K_\mathrm{dr}^{1/2}}
 \end{align*}
 by Lemma~\ref{lemma:alternative-step-2}. Thus, we drop the dependence on $\zeta$.
\end{proof}

We will require to show strong convergence of the Kirchhoff pressure. Having that in mind, we conclude with a stability estimate for $\partial_t \chiv$. We note, this is the only stability estimate in this section, requiring the regularizing growth condition~(A1$^\star$).

\begin{lemma}[Stability estimate for the temporal change of the Kirchhoff pressure]\label{lemma:regularization:stability:dchidt}
 There exists a constant $\epsStabilityDchiDt>0$ (independent of $\zeta$), such that
 \begin{align*}
  \| \partial_t \chiv \|_{L^2(Q_T)}^2
 \leq \epsStabilityDchiDt\left(b_\mathrm{\chi,m}^{-1} C_0, b_\mathrm{\chi,m}^{-2} \epsStabilityDuDt \right),
 \end{align*}
 where $\epsStabilityDuDt$ is the stability constant from Lemma~\ref{lemma:alternative-step-2}, $b_\mathrm{\chi,m}$ is from~$\mathrm{(A1}^\star\mathrm{)}$, and $C_0$ is from~$\mathrm{(A8}^\star\mathrm{)}$.
\end{lemma}

\begin{proof}
 We repeat parts of the proof of Lemma~\ref{lemma:alternative-step-2}. We test the flow equation~\eqref{model:regularized:weak:kirchhoff:p} with $q=\partial_t \chiv$ and apply~(A1$^\star$) and the Cauchy-Schwarz inequality; we obtain 
 \begin{align*}
 &b_\mathrm{\chi,m} \| \partial_t \chiv \|_{L^2(Q_T)}^2
 + \frac{1}{2} \llangle \absolutepermeability \GRAD \chiv(T), \GRAD \chiv(T) \rrangle \\
 &\quad \leq
 \int_0^T 
 \llangle \partial_t \bke(\chiv), \partial_t \chiv \rrangle \, dt
 + \frac{1}{2} \llangle \absolutepermeability \GRAD \chiv(T), \GRAD \chiv(T) \rrangle \\
 &\quad = 
 \frac{1}{2} \llangle \absolutepermeability \GRAD \chiv(0),\GRAD \chiv(0) \rrangle 
 + \int_0^T \left( \llangle \hext, \partial_t \chiv \rrangle - \alpha \llangle \swk \partial_t \DIV \uv, \partial_t \chiv \rrangle \right)\, dt \\
 &\quad \leq
 \frac{1}{2} \llangle \absolutepermeability \GRAD \chiv(0),\GRAD \chiv(0) \rrangle 
 + \frac{1}{b_\mathrm{\chi,m}} \left( \left\| \hext \right\|_{L^2(0,T;Q^\star)}^2
 + \alpha^2 \| \partial_t \DIV \uv \|_{L^2(Q_T)}^2\right) \\
 &\qquad+ \frac{b_\mathrm{\chi,m}}{2} \| \partial_t \chiv \|_{L^2(Q_T)}^2.
 \end{align*}
 After rearranging terms, applying the regularity of the data, and applying from~Lemma~\ref{lemma:alternative-step-2}, the assertion follows.
\end{proof}

\subsection{Relative (weak) compactness for $\zeta\rightarrow 0$}\label{section:regularization:relative-compactness}
We utilize the stability results from the previous section to conclude relative compactness. 

\begin{lemma}[Convergence of the primary variables]\label{lemma:regularization:convergence-u-chi}
 We can extract subsequences of $\{\uv \}_{\zeta}$ and $\{\chiv\}_{\zeta}$ (still denoted like the original sequences), and there exist $\ueta\in H^1(0,T;\V)$ and $\chieta\in H^1(0,T;L^2(\Omega))\cap L^\infty(0,T;Q)$ such that for $\zeta\rightarrow0$
 \begin{alignat}{2}
 \label{result:regularization:convergence-u:1}
  \uv     &\rightharpoonup \ueta                       && \text{ in }H^1(0,T;\V), \\
 \label{result:regularization:convergence-u:3}
  \zeta \partial_t \uv     &\rightarrow     \bm{0}                       && \text{ in }L^2(0,T;\V), \\
  \label{result:regularization:convergence-chi:1}
  \chiv &\rightarrow \chieta               && \text{ in }L^2(Q_T), \\
  \label{result:regularization:convergence-chi:2}
  \chiv &\rightharpoonup \chieta           && \text{ in }L^\infty(0,T;Q), \\
  \label{result:regularization:convergence-chi:3}
  \partial_t \chiv &\rightharpoonup \partial_t \chieta && \text{ in }L^2(Q_T).
 \end{alignat}
\end{lemma}

\begin{proof}
 The proof follows standard arguments based on the Eberlein-$\check{\text{S}}$mulian theorem, cf.\ Lemma~\ref{appendix:lemma:eberlein-smulian}, the Aubin-Lions lemma, cf.\ Lemma~\ref{appendix:lemma:aubin-lions}, and the stability results for $\uv$, cf.\ Lemma~\ref{lemma:alternative-step-2} and~\eqref{eps-regularization:stability-bound}, as well as the stability results for $\chiv$, cf.\ Lemma~\ref{lemma:alternative-step-2} and Lemma~\ref{lemma:regularization:stability:dchidt}. In particular, for~\eqref{result:regularization:convergence-u:3}, we employ the uniform stability result from Lemma~\ref{lemma:alternative-step-2}; for all fixed $\v\in L^2(0,T;\V)$ it holds that
 \begin{align*}
  \left|\int_0^T \zeta a(\partial_t \uv, \v)\, dt \right| \leq \zeta \epsStabilityDuDt \| \v \|_{L^2(0,T;\V)} \rightarrow 0 \quad\text{for }\zeta \rightarrow 0.
 \end{align*}
\end{proof}

\begin{lemma}[Convergence of the coupling terms]\label{lemma:regularization:convergence-coupling}
Up to subsequences it holds for $\zeta\rightarrow0$
  \begin{alignat}{2}
  \label{result:regularization:convergence-pek}
  \pek(\chiv) &\rightharpoonup \pek(\chieta)       && \text{ in }L^2(Q_T), \\
  \label{result:regularization:convergence-coupling}
  \swk(\chiv)\partial_t \DIV \uv &\rightharpoonup \swk(\chieta)\partial_t \DIV \ueta &\quad&\text{ in }L^2(Q_T).
 \end{alignat}
\end{lemma}

\begin{proof}
 The proof is analogous to the proof of Lemma~\ref{lemma:convergence-coupling}. Essentially, first, one has to utilize stability estimates together with the Eberlein-$\check{\text{S}}$mulian theorem, cf.\ Lemma~\ref{appendix:lemma:eberlein-smulian}; second, continuity properties of the non-linearities have to be employed together with the convergence of $\{\uv\}_{\zeta}$ and $\{\chiv\}_{\zeta}$, cf.\ Lemma~\ref{lemma:regularization:convergence-u-chi}. We note that for~\eqref{result:regularization:convergence-pek} the stability result~\eqref{eps-regularization:stability-bound} has to be utilized.
\end{proof}

\begin{lemma}[Initial conditions for the fluid flow]\label{lemma:regularization:convergence-dbdt} 
Up to subsequences it holds for $\zeta\rightarrow0$
\begin{alignat}{2}
  \label{result:regularization:convergence-dtbk}
  \partial_t \bke(\chiv) &\rightharpoonup \partial_t \bke(\chieta) &&\text{ in }L^2(0,T;Q^\star),
\end{alignat}
where $\partial_t \bke(\chieta)\in L^2(0,T;Q^\star)$ is understood in the sense of~$\mathrm{(W2)_\eta}$.
\end{lemma}

\begin{proof}
 The proof is analogous to the proof of Lemma~\ref{lemma:convergence-dbdt}. By Lemma~\ref{lemma:regularization:stability:dkbe-dt} and the Eberlein-$\check{\text{S}}$mulian theorem, cf.\ Lemma~\ref{appendix:lemma:eberlein-smulian}, there exists a $b_t\in L^2(0,T;Q^\star)$ such that $\partial_t \bke(\chiv) \rightharpoonup b_t$ in $L^2(0,T;Q^\star)$ (up to a subsequence). We can identify $b_t = \partial_t b(\chieta)$ by showing (W2)$_\eta$. For this we utilize (W2)$_{\zeta\eta}$. For $q\in L^2(0,T;Q)$ with $\partial_t q\in L^1(0,T:L^\infty(\Omega))$ and $q(T)=0$, it holds that
 \begin{align*}
  \int_0^T \llangle \partial_t \bke(\chiv), q \rrangle \, dt = \int_0^T \llangle \bke(\chi_0) - \bke(\chiv), \partial_t q \rrangle \, dt.
 \end{align*}
 The assertion follows immediately if
 \begin{align}
 \label{result:regularization:convergence-bk}
  \bke(\chiv) &\rightarrow \bke(\chieta) \quad \text{ in }L^\infty(0,T;L^1(\Omega))
 \end{align}
 (up to a subsequence). And indeed, by the uniform boundedness of the Legendre transformation, $\| \Bke (\chieta) \|_{L^\infty(0,T;L^1(\Omega))}$, there exists $b_\chi\in L^\infty(0,T;L^1(\Omega))$ such that $\bke(\chiv)\rightharpoonup b_\chi$ in $L^\infty(0,T;L^1(\Omega))$. Using the strong convergence of $\{\chieta\}_\eta$ and the dominated convergence theorem, we can identify $b_\chi = \bke(\chieta)$, and thus~\eqref{result:regularization:convergence-bk}.
\end{proof}

\begin{lemma}[Initial conditions of the mechanical displacement]\label{lemma:regularization:initial-conditions-u}
 $\partial_t \ueta\in L^2(0,T;\V)$ satisfies~$\mathrm{(W3)_{\eta}}$.
\end{lemma}

\begin{proof}
 Using the uniform stability bound for $\{\partial_t \uveta\}_\zeta$ by Lemma~\ref{lemma:alternative-step-2} and the weak convergence $\uveta \rightharpoonup \ueta$ in $L^2(0,T;\V)$ (up to a subsequence) by Lemma~\ref{lemma:regularization:convergence-u-chi}, standard compactness arguments yield $\partial_t \uveta \rightharpoonup \partial_t \ueta$ in $L^2(0,T;\V)$ (up to a subsequence). Hence, $\zeta\rightarrow 0$ of (W3)$_{\zeta\eta}$ yields (W3)$_{\eta}$ immediately.
\end{proof}

\subsection{Identifying a weak solution for $\zeta \rightarrow 0$}\label{section:eps-regularization:limit-equations}

Finally, we show the limit $(\ueta,\chieta)$, introduced above, is a weak solution of the simply regularized unsaturated poroelasticity model.

\begin{lemma}[Limit satisfies (W1)$_\eta$--(W4)$_\eta$]\label{lemma:limit-satisfies-w1-w3-eps}
 The limit $(\ueta,\chieta)$, derived in Section~\ref{section:regularization:relative-compactness}, is a weak solution of the simply regularized unsaturated poroelasticity model, cf.\ Definition~\ref{definition:simply-regularized-weak-solution}.
\end{lemma}

\begin{proof}
 The limit $(\ueta,\chieta)$ satisfies~(W1)$_\eta$--(W3)$_\eta$ by Lemma~\ref{lemma:regularization:convergence-u-chi}, Lemma~\ref{lemma:regularization:convergence-coupling}, Lemma~\ref{lemma:regularization:convergence-dbdt}, and Lemma~\ref{lemma:regularization:initial-conditions-u}. It remains to show~(W4)$_\eta$, i.e., that $(\ueta,\chieta)$ satisfies the balance equations~\eqref{model:regularized:weak:kirchhoff:u}--\eqref{model:regularized:weak:kirchhoff:p} for $\zeta=0$. By definition, the sequence $(\uv,\chiv)$ satisfies $\mathrm{(W4)_{\zeta\eta}}$ for $\zeta>0$, i.e., it holds for all $(\v,q)\in L^2(0,T;\V) \cap L^2(0,T;Q)$
 \begin{align*} 
  \int_0^T  \left[ \zeta a(\partial_t \uv, \v)
 + a(\uv, \v)
 - \alpha \llangle \pek(\chiv), \DIV \v \rrangle \right]& \, dt = \int_0^T \llangle \fext, \v \rrangle \, dt, \\
 \int_0^T \bigg[
 \llangle \partial_t \bke(\chiv), \q \rrangle + \alpha \llangle \swk(\chiv) \partial_t \DIV \uv, \q \rrangle 
 + \llangle \absolutepermeability \GRAD \chiv, \GRAD \q \rrangle
 \bigg]& \, dt = \int_0^T \llangle \hext, \q \rrangle\, dt.
 \end{align*}
 Utilizing the weak convergence results, cf.\ Lemma~\ref{lemma:regularization:convergence-u-chi} and Lemma~\ref{lemma:regularization:convergence-coupling}, (W4)$_{\eta}$ follows directly for $\zeta \rightarrow 0$.
\end{proof}

\begin{remark}[Existence of a weak solution for compressible system]\label{remark:consequence-compressible-system}
 If compressibility is present either for the fluid or the solid grains, the regularizing property~$\mathrm{(A1}^\star\mathrm{)}$ is fulfilled for $\eta=0$. For instance, for $b$ as in~\eqref{model:example-b}, the equivalent pore pressure and the van Genuchten-Mualem model, it holds that $b_\mathrm{\chi,m}= \phi_0 c_\mathrm{w} + \tfrac{1}{N}$, cf.\ Appendix~\ref{appendix:feasibility-assumptions}. Consequently, the limit $(\ueta,\chieta)$ in Lemma~\ref{lemma:limit-satisfies-w1-w3-eps} is also well-defined for $\eta=0$. In particular, it is a weak solution of~\eqref{model:governing:kirchhoff:u}--\eqref{model:initial:kirchhoff:flow}, cf.\ Definition~\ref{definition:weak-solution}.
\end{remark}

\section{Step 6: Limit $\eta\rightarrow0$ in the incompressible case}\label{section:eta-regularization}

In this section, we show the main result, Theorem~\ref{theorem:existence}, for the more demanding case of an incompressible fluid and incompressible solid grains. Otherwise, by Remark~\ref{remark:consequence-compressible-system} the main result of this paper follows already. In the incompressible case, $b$ as in~\eqref{model:example-b} is monotone but with $\bk'=0$ on a part of the domain with non-zero measure. Under the use of regularization with $\eta>0$, it holds that $b_\mathrm{\chi,m}=\eta$. In the following, we prove that the limit of $\{(\ueta,\chieta)\}_\eta$ for $\eta\rightarrow 0$ exists, and that it is a weak solution of~\eqref{model:governing:kirchhoff:u}--\eqref{model:initial:kirchhoff:flow} according to Definition~\ref{definition:weak-solution}. Throughout the entire section, we assume~(A0)--(A9) and~(ND1)--(ND3) hold true.

\subsection{Stability estimates independent of $\eta$}

In Section~\ref{section:pure-consolidation}, almost all stability bounds have been independent of $\eta$. To summarize, there exists a constant $C>0$ (independent of $\eta$) such that
\begin{align}
\label{eta-regularization:uniform-stability-bound}
 &\left\| \ueta \right\|_{H^1(0,T;\V)} 
 + 
 \left\| \chieta \right\|_{L^\infty(0,T;H^1_0(\Omega))}
 +
 \left\| \pek(\chieta) \right\|_{L^2(Q_T)} \\
 \nonumber
 &\quad + \left\| \Bke(\chieta) \right\|_{L^\infty(0,T;L^1(\Omega))}
 +
 \left\| \partial_t \bke(\chieta) \right\|_{L^2(0,T;H^{-1}(\Omega))}
 \leq C.
\end{align}
The only bound depending on $\eta$ is the stability of $\partial_t \chieta$, cf.\ Lemma~\ref{lemma:regularization:stability:dchidt}. We recall, there exists a constant $C_\eta>0$, depending on $\eta$, satisfying
\begin{align}
\label{eta-regularization:stability-bound-dchi-dt}
 \left\| \partial_t \chieta \right\|_{L^2(Q_T)} \leq C_\eta.
\end{align}
In order to conclude that $(\ueta,\chieta)$ converges towards a weak solution of the unsaturated poroelasticity model, it will be sufficient to replace the stability result~\eqref{eta-regularization:stability-bound-dchi-dt} by a uniform stability estimate. The remaining discussion for $\eta\rightarrow 0$ can be done along the lines of Section~\ref{section:regularization:relative-compactness}--\ref{section:eps-regularization:limit-equations}.

In the following, we prove a uniform stability bound replacing~\eqref{eta-regularization:stability-bound-dchi-dt} in two steps. We show that the temporal derivative of the mechanics equation, i.e., (W5)$_{\zeta\eta}$ for $\zeta=0$, is well-defined; and then we use an inf-sup argument and the uniform stability estimate~\eqref{eta-regularization:uniform-stability-bound}.

\begin{lemma}[Temporal derivative of the mechanics equation]\label{regularization-eta-dt-mechanics}
 It holds for all $\v\in L^2(0,T;\V)$
 \begin{align}
 \label{result:regularization-eta-dt-mechanics}
  \int_0^T a(\partial_t \ueta, \v) \, dt -  \int_0^T \alpha\llangle \partial_t \pek(\chieta), \DIV \v \rrangle \, dt = \int_0^T \llangle \partial_t \fext, \v \rrangle \, dt.
 \end{align}
\end{lemma}

\begin{proof}
 First, we argue that the mechanics equation~\eqref{model:weak:kirchhoff:u} holds pointwise on $[0,T]$. Let $\v\in L^2(0,T;\V)\cap C^\infty(0,T;\V)$. By Lemma~\ref{lemma:limit-satisfies-w1-w3-eps}, it holds that
 \begin{align*}
  \int_0^T a(\ueta, \v) \, dt -  \int_0^T \alpha\llangle \pek(\chieta), \DIV \v \rrangle \, dt = \int_0^T \llangle \fext, \v \rrangle \, dt.
 \end{align*}
 By the fundamental lemma of calculus of variations it follows a.e.\ on $[0,T]$
 \begin{align}
 \label{proof:regularization-eta-dt-mechanics:aux1}
   a(\ueta, \v) - \alpha \llangle \pek(\chieta), \DIV \v \rrangle = \llangle \fext, \v \rrangle,\quad\text{for all } \v\in\V.
 \end{align}
 Applying a standard embedding for Bochner spaces~\cite{Evans2010}, we can assume wlog.\ that $\ueta\in C(0,T;\V)$ and $\pek(\chieta)\in C(0,T;L^2(\Omega))$, as $\partial_t \ueta \in L^2(0,T;\V)$ and $\partial_t \pek(\chieta) \in L^2(Q_T)$ by~\eqref{eta-regularization:stability-bound-dchi-dt} and assumption~(ND2). Hence,~\eqref{proof:regularization-eta-dt-mechanics:aux1} holds pointwise on $[0,T]$.

 Now we show~\eqref{result:regularization-eta-dt-mechanics}. Let $\v \in L^2(0,T;\V) \cap C^\infty(0,T;\V)$. By Lemma~\ref{lemma:limit-satisfies-w1-w3-eps}, it holds that
 \begin{align*}
  \int_0^T a(\ueta, \partial_t \v) \, dt - \alpha \int_0^T \llangle \pek(\chieta), \DIV \partial_t \v \rrangle \, dt = \int_0^T \llangle \fext, \partial_t \v \rrangle \, dt.  
 \end{align*}
 Since $\partial_t \ueta \in L^2(0,T;\V)$, $\partial_t \pek(\chieta) \in L^2(Q_T)$ and $\partial_t \fext \in L^2(0,T;\V^\star)$, integration by parts is well-defined. Together with~\eqref{proof:regularization-eta-dt-mechanics:aux1}, we obtain
 \begin{align*}
  \int_0^T a(\partial_t \ueta, \v) \, dt - \alpha \int_0^T \llangle \partial_t \pek(\chieta), \DIV \v \rrangle \, dt = \int_0^T \llangle \partial_t \fext, \v \rrangle \, dt.  
 \end{align*}
 The assertion follows after applying a density argument allowing for arbitrary test functions in $L^2(0,T;\V)$ in~\eqref{result:regularization-eta-dt-mechanics}.
\end{proof}

\begin{lemma}[Stability estimate for the temporal derivative of the Kirchhoff pressure]
 There exists a constant $\etaStabilityDchiDt>0$ (independent of $\eta$) such that
 \begin{align*}
  \left\| \partial_t \chieta \right\|_{L^2(Q_T)} \leq \etaStabilityDchiDt.
 \end{align*}
\end{lemma}

\begin{proof}
 We show that $\| \partial_t \pek(\chieta)\|_{L^2(Q_T)}$ is uniformly bounded. The assertion follows then from assumption~(ND2), as
 \begin{align*}
  \left\| \partial_t \chieta \right\|_{L^2(Q_T)} \leq C_\mathrm{ND,2} \left\| \partial_t \pek(\chieta) \right\|_{L^2(Q_T)}.
 \end{align*}
 By Lemma~\ref{regularization-eta-dt-mechanics}, the time derivative of the mechanics equations is well-defined, cf.~\eqref{result:regularization-eta-dt-mechanics}. Using a standard inf-sup argument (introducing the constant $\COmegaInfSup$), cf.\ Lemma~\ref{appendix:lemma:inf-sup}, it follows from~\eqref{result:regularization-eta-dt-mechanics} that
 \begin{align*}
  \left\| \partial_t \pek(\chieta) \right\|_{L^2(Q_T)}
  \leq 
  \COmegaInfSup \left( \left\| \partial_t \ueta \right\|_{L^2(0,T;\V)} + \left\| \partial_t \fext \right\|_{L^2(0,T;\V^\star)} \right).
 \end{align*}
 Since $\| \partial_t \ueta \|_{L^2(0,T;\V)}$ is uniformly bounded by~\eqref{eta-regularization:uniform-stability-bound}, $\| \partial_t \pek(\chieta)\|_{L^2(Q_T)}$ is uniformly bounded, which concludes the proof.
\end{proof}

\subsection{Relative (weak) compactness for $\eta\rightarrow 0$}

Using the same line of argumentation used in Section~\ref{section:regularization:relative-compactness}, we can discuss the limit process $\eta\rightarrow 0$.

\begin{lemma}[Convergence of the primary variables]\label{lemma:convergence-eta-0}
 We can extract subsequences of $\{\ueta \}_{\eta}$ and $\{\chieta\}_{\eta}$ (still denoted like the original sequences), and there exist $\u\in H^1(0,T;\V)$ and $\chi\in H^1(0,T;L^2(\Omega))\cap L^\infty(0,T;Q)$ such that for $\eta\rightarrow 0$
 \begin{alignat*}{2}
  \ueta     &\rightharpoonup \u                      && \text{ in }H^1(0,T;\V), \\
  \chieta &\rightarrow \chi               && \text{ in }L^2(Q_T), \\
  \chieta &\rightharpoonup \chi           && \text{ in }L^\infty(0,T;Q), \\
  \partial_t \chieta &\rightharpoonup \partial_t \chi && \text{ in }L^2(Q_T).
 \end{alignat*}
\end{lemma}

\begin{proof}
 The proof is analog to the proofs of Lemma~\ref{lemma:regularization:convergence-u-chi}.
\end{proof}

\begin{lemma}[Convergence of the coupling terms]\label{lemma:eta:convergence-coupling}
 Up to subsequences it holds for $\eta\rightarrow 0$ that
 \begin{alignat*}{2}
  \pek(\chieta) &\rightharpoonup \pek(\chi)       &\quad& \text{ in }L^2(Q_T), \\
  \swk(\chieta)\partial_t \DIV \ueta &\rightharpoonup \swk(\chi)\partial_t \DIV \u &&\text{ in }L^2(Q_T).
 \end{alignat*}
\end{lemma}

\begin{proof}
 The proof is analog to the proof of Lemma~\ref{lemma:regularization:convergence-coupling}.
\end{proof}

\begin{lemma}[Initial conditions for the fluid flow]\label{lemma:eta:convergence-dbdt}
\label{lemma:convergence-eta-0}
 Up to subsequences it holds that
 \begin{alignat*}{2}
  \partial_t \bke(\chieta) &\rightharpoonup \partial_t \bk(\chi) &&\text{ in }L^2(0,T;Q^\star),
 \end{alignat*}
 where $\partial_t \bk(\chi)\in L^2(0,T;Q^\star)$ is understood in the sense of~$\mathrm{(W2)}$.
\end{lemma}

\begin{proof}
 The proof is analog to the proof of Lemma~\ref{lemma:regularization:convergence-dbdt}. We only stress that due to construction of $\bke$, one can show that if $\chieta \rightarrow \chi$ in $L^2(Q_T)$, it also holds
 \begin{alignat*}{2}
  \bke(\chi_0)  &\rightharpoonup \bk(\chi_0) &\ \  &\text{in }L^\infty(0,T;L^1(\Omega)),\\
  \bke(\chieta) &\rightharpoonup \bk(\chi)   &\ \ &\text{in }L^\infty(0,T;L^1(\Omega)),
 \end{alignat*}
 for $\eta\rightarrow 0$. Hence, (W2) can be deduced from~(W2)$_\eta$ for $\eta\rightarrow 0$.
\end{proof}

\begin{lemma}[Initial conditions of the mechanical displacement]\label{lemma:eta:initial-conditions-u}
 $\partial_t \DIV \u\in L^2(Q_T)$ satisfies~$\mathrm{(W3)}$.
\end{lemma}

\begin{proof}
 The proof is almost identical to the proof of Lemma~\ref{lemma:regularization:initial-conditions-u}. Standard compactness arguments and (W3)$_\eta$ yield
 \begin{align*}
  \int_0^T a(\partial_t \u, \v) \, dt + \int_0^T a(\u - \u_0, \partial_t \v) \, dt = 0
 \end{align*}
 for all $\v\in H^1(0,T;\V)$ with $\v(T)=\bm{0}$. Hence, $\u(0) = \u_0$ in $\V$; note that $\u \in C(0,T;\V)$ by a Sobolev embedding. Therefore also $\DIV \u(0) = \DIV \u_0$ in $L^2(\Omega)$, which yields (W3).
\end{proof}

\subsection{Identifying a weak solution for $\eta\rightarrow 0$}

Finally, we prove the existence of a weak solution to the unsaturated poroelasticity model.

\begin{lemma}[Limit satisfies (W1)--(W4)]\label{lemma:limit-satisfies-w1-w3}
 The limit $(\u,\chi)$ is a weak solution of~\eqref{model:governing:kirchhoff:u}--\eqref{model:initial:kirchhoff:flow}, cf.\ Definition~\ref{definition:weak-solution}.
\end{lemma}

\begin{proof}
 The proof follows directly from the convergence results in Lemma~\ref{lemma:convergence-eta-0} and Lemma~\ref{lemma:eta:initial-conditions-u} together with the validity of the regularized problem~\eqref{model:regularized:weak:kirchhoff:u}--\eqref{model:regularized:weak:kirchhoff:p} for $\zeta=0$. 
\end{proof}

\appendix

\section{Feasibility of assumptions}\label{appendix:feasibility-assumptions}

The analysis of this paper allows for arbitrary constitutive laws for $b$, $\porepressure$, $\sw$ and $\relativepermeability$, as long as they satisfy the conditions~(A0)--(A4), (ND1)--(ND3) and (A1$^\star$). In the following, we demonstrate the feasibility of those conditions for a prominent choice of models. Let $b$ as derived by~\cite{Lewis1998}
\begin{align*}
 b(\pw) = \phi_0 \sw(\pw) + c_\mathrm{w} \phi_0 \int_0^{\pw} \sw(p) \, dp  + \frac{1}{N} \int_0^{\pw} s_\mathrm{w}(p) \porepressure'(p) \, dp,
\end{align*}
with $\porepressure$ chosen as equivalent pore pressure~\cite{Coussy2004}
\begin{align*}
 \porepressure(\pw) &= \int_0^{\pw} \sw(p) \, dp,
\end{align*}
and the hydraulic properties $\sw$ and $\relativepermeability$ given by the van Genuchten-Mualem relations~\cite{vanGenuchten1980,Mualem1976}
\begin{align*}
\sw(\pw) &= \left\{ \begin{array}{ll} 
                    \left[ 1 + \left(- \vangenuchtenalpha \pw\right)^{\vangenuchtenn} \right]^{-\vangenuchtenm}, & \pw \leq 0, \\
                    1, & \pw \geq 0,
                   \end{array} \right. \\
 \relativepermeability(\sw) &= \sqrt{\sw} \left[1 - \left(1 - \sw^\frac{1}{\vangenuchtenm}\right)^\vangenuchtenm \right]^2.
\end{align*}
where $\vangenuchtenm\in(0,1)$, $\vangenuchtenn=(1-\vangenuchtenm)^{-1}$, and $\vangenuchtenalpha>0$ are constant fitting parameters.

\subsection{Checking (A0)}
By definition, it holds that $\sw(\pw)>0$ for all $\pw\in\mathbb{R}$ and $\relativepermeability(\sw)>0$ for all $\sw>0$. Hence,~(A0) is satisfied for the van Genuchten-Mualem relations.

\subsection{Checking (A1)--(A4) and (A1$^\star$)}
By definition, it follows directly, that $\sw$ is differentiable with a non-negative and uniformly bounded derivative $\sw'$, i.e., $\sw$ satisfies~(A2). Furthermore, $\porepressure'(\pw)=\sw(\pw)$, and hence, $\porepressure$ satisfies~(A3). We therefore only focus on~(A1),~(A1$^\star$) and~(A4).

\paragraph{(A1) Monotonicity of $\bk$.}

The function $\bk=\bk(\chi)$ is non-decreasing since
\begin{align}\label{van-genuchten-equivalent-pore-pressure-01}
 \bk'(\chi)
 =
 c_\mathrm{w} \phi_0 \frac{\swk(\chi)}{\relativepermeabilityk(\chi)}
 +
 \phi_0 \frac{\sw'(\pwk(\chi))}{\relativepermeabilityk(\chi)} + \frac{1}{N} \frac{\swk(\chi)^2}{\relativepermeabilityk(\chi)} \geq 0.
\end{align}

\paragraph{(A1$^\star$) Regularizing property of $\bke$.}
As $\bke$ is essentially equal to $\bk$ but with enhanced Biot Modulus, $\bke$ essentially satisfies (A1) with
\begin{align*}
 \bke'(\chi)
 =
 c_\mathrm{w} \phi_0 \frac{\swk(\chi)}{\relativepermeabilityk(\chi)}
 +
 \phi_0 \frac{\sw'(\pwk(\chi))}{\relativepermeabilityk(\chi)} + \left(\frac{1}{N} + \eta \right) \frac{\swk(\chi)^2}{\relativepermeabilityk(\chi)} \geq 0.
\end{align*}
In particular, it holds that
\begin{align*}
 \langle \bk(\chi_1) - \bk(\chi_2), \chi_1 - \chi_2 \rangle &\geq \left(c_\mathrm{w}\phi_0 \left\| \tfrac{\relativepermeability}{\sw} \right\|_\infty^{-1} + \left(\frac{1}{N}+\eta\right) \left\| \tfrac{\relativepermeability}{\sw^2} \right\|_\infty^{-1} \right) \| \chi_1 - \chi_2 \|^2.
\end{align*}
By l'H$\hat{\text{o}}$spital's rule (note $0<\vangenuchtenm<1$) it holds that
\begin{align*}
 \underset{\sw\rightarrow0}{\mathrm{lim}} \, \left(\frac{\relativepermeability(\sw)}{\sw} \right)^2
 &=
 \underset{\sw\rightarrow0}{\mathrm{lim}} \, 4  \left[1 - \left(1 - \sw^\frac{1}{\vangenuchtenm}\right)^\vangenuchtenm \right]^3 \left(1 - \sw^\frac{1}{\vangenuchtenm}\right)^{\vangenuchtenm-1} \sw^{1/\vangenuchtenm-1}= 0.
\end{align*}
and
\begin{align*}
 \underset{\sw\rightarrow0}{\mathrm{lim}} \, \left(\frac{\relativepermeability(\sw)}{\sw^2} \right)^{2/3}
 &=
 \underset{\sw\rightarrow0}{\mathrm{lim}} \,  \frac{4}{3} \left[1 - (1 - \sw^\frac{1}{\vangenuchtenm})^\vangenuchtenm \right]^{1/3}  \left(1 - \sw^\frac{1}{\vangenuchtenm}\right)^{\vangenuchtenm-1} \sw^{1/\vangenuchtenm-1} = 0.
\end{align*}
Hence, there exists a generic constant $c>0$, such that
\begin{align} 
\label{equivalent-pore-pressure-van-genuchten-relation-01a}
 \frac{\relativepermeability(\pw)}{\sw(\pw)}   &\in (0,c],\\
\label{equivalent-pore-pressure-van-genuchten-relation-01b}
 \frac{\relativepermeability(\pw)}{\sw(\pw)^2} &\in (0,c].
\end{align}
After all, it follows, for $\eta>0$, $\bke$ satisfies~(A1$^\star$). Furthermore, in the compressible case $\mathrm{max}\{c_\mathrm{w},\tfrac{1}{N}\}>0$, also $\bk$ satisfies~(A1$^\star$), cf.\ Remark~\ref{remark:consequence-compressible-system}.

\paragraph{(A4) Uniform growth of $\tfrac{\pek}{\swk}$.}
For all $\pw\in\mathbb{R}$, it holds that
\begin{align*}
 \frac{d}{d\pw} \left( \frac{\porepressure}{\sw} \right) &= 1 - \frac{\porepressure(\pw) \sw'(\pw)}{\sw(\pw)^2}\geq 1, \\
 \chi'(\pw)&=\relativepermeability(\sw(\pw))\leq 1.
\end{align*}
Hence, by using the chain rule, $\tfrac{\pek}{\swk}$ satisfies the uniform growth condition~(A4) with
\begin{align*}
 \frac{d}{d\chi} \left( \frac{\pek}{\swk} \right)
 \geq 
 1.
\end{align*}

\subsection{Checking (ND1)--(ND2)}

We demonstrate, that (ND1)--(ND2) hold assuming $\sw\geq s_\mathrm{min}$ for some minimal saturation value $s_\mathrm{min}>0$. It holds that
\begin{align*}
 \frac{\pek}{\swk \chi} \sim \frac{1}{\relativepermeabilityk} \quad \text{for} \quad \chi \rightarrow -\infty.
\end{align*}
Under above assumption, one can assume that $\relativepermeabilityk \geq \kappa_\mathrm{min}>0$, such that (ND1) holds. Furthermore, 
\begin{align*}
 \pek'(\chi) = \frac{\swk}{\relativepermeabilityk}.
\end{align*}
By~\eqref{equivalent-pore-pressure-van-genuchten-relation-01a}, $\pek'(\chi)$ is bounded from below by a constant independent of $\chi$. Assuming $\sw\geq s_\mathrm{min}$ for some minimal saturation value, also an upper bound is given. After all,~(ND2) holds.

\subsection{Discussion of~(ND3)}\label{appendix:condition-nd-3}
The condition (ND3) is equivalent with
\begin{align}
\label{appendix:nd3}
 \left( \frac{\swk(\chi)}{\pek'(\chi)} - 1 \right)^{-2} \frac{\bk'(\chi)}{\left(\pek'(\chi)\right)^2} > \frac{\alpha^2}{4 K_\mathrm{dr}}\qquad \text{for all } \chi\in\mathbb{R}.
\end{align}
First, we note that in the fully saturated regime condition, (ND3) is fulfilled since
\begin{align*}
 \frac{\swk(\chi)}{\pek'(\chi)} = 1,\quad \text{for all } \chi\geq 0.
\end{align*}
For the combination of the specific choices of $b$, $\sw$ and $\relativepermeability$ condition~\eqref{appendix:nd3} becomes
 \begin{align*}
  \left(\relativepermeability(\sw) - 1   \right)^{-2} \left( \phi_0 c_\mathrm{w} \frac{\relativepermeability(\sw)}{\sw} + \phi_0 \sw' \frac{\relativepermeability(\sw)}{\sw^2} + \frac{1}{N} \relativepermeability(\sw) \right) > \frac{\alpha^2}{4 K_\mathrm{dr}} \qquad \text{for all } \sw<1.
 \end{align*} 
We consider the more demanding case, the incompressible case with $c_\mathrm{w}=\frac{1}{N}=0$. The expression $\frac{\sw' \relativepermeability(\sw)}{ \sw^2 \left(1 - \relativepermeability(\sw) \right)^2}$ is increasing in $\pw$, see Figure~\ref{appendix:figure:increase} for two examples. Hence, there exists a minimal saturation value $s_\mathrm{min}$ such that~\eqref{appendix:condition-nd-3} holds in the regime $\sw\in[s_\mathrm{min},1]$. This value will depend on $\phi_0$, $\vangenuchtenalpha$, $\vangenuchtenn$, $\alpha$ and $K_\mathrm{dr}$. Assuming $\phi_0=0.1$ and $\alpha=1$, we compute $s_\mathrm{min}$ for a set of realistic parameters, see Table~\ref{appendix:table:sminvalues}. We observe, that the range of admissible saturation values becomes larger, the stiffer the system. Furthermore, for all parameters, $s_\mathrm{min}$ is relatively small. Hence, we can expect~(ND3) to hold for geotechnical applications, for which $K_\mathrm{dr}$ is typically large.


\begin{figure}[h!]
\centering
 \subfloat[$\vangenuchtenn=1.5$, $\vangenuchtenalpha=0.1$]{\includegraphics[width=0.46\textwidth]{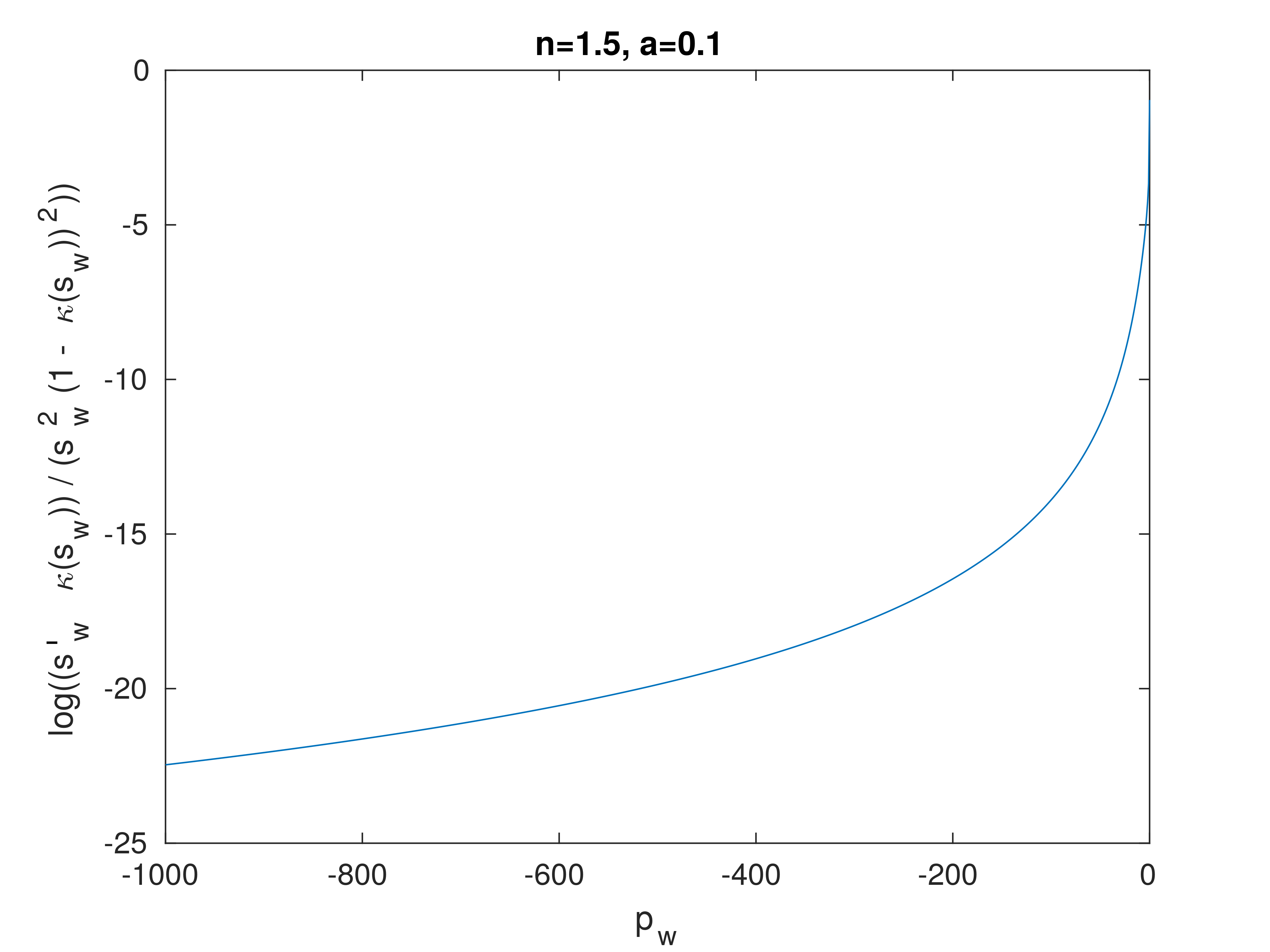}}
 \hspace{1cm}
 \subfloat[$\vangenuchtenn=2.5$, $\vangenuchtenalpha=2$]{\includegraphics[width=0.46\textwidth]{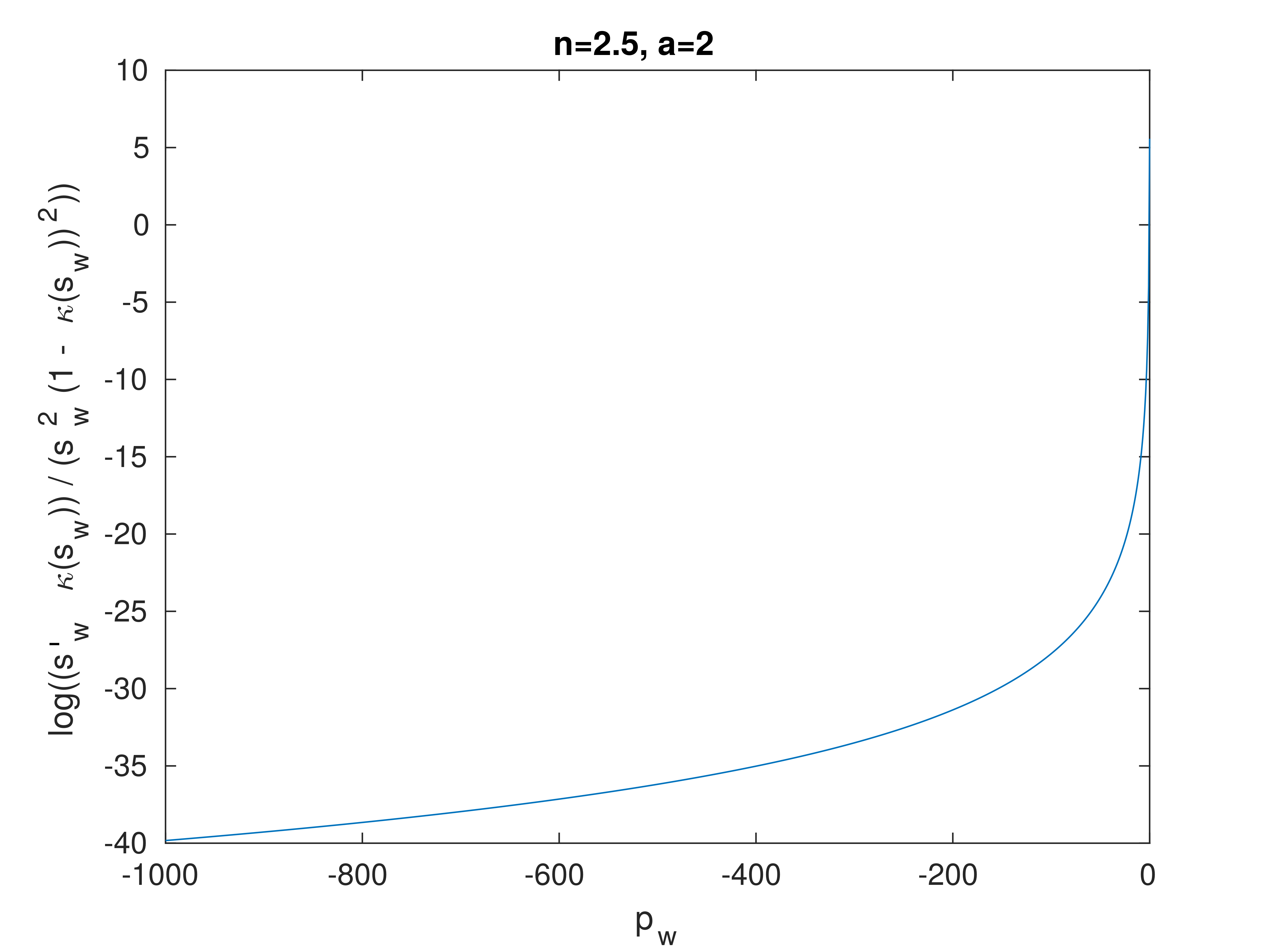}}
 \caption{\label{appendix:figure:increase} Increasing behavior of $\frac{\sw' \relativepermeability(\sw)}{ \sw^2 \left(1 - \relativepermeability(\sw) \right)^2}$ in the unsaturated regime.}
\end{figure}

\begin{table}[h!]
\begin{center}
\setlength{\tabcolsep}{0.6em}
\def\arraystretch{1.3}
\begin{tabular}{|c| c || c | c | c|}
 \hline
 $\vangenuchtenalpha$ & $\vangenuchtenn$ & $s_\mathrm{min}$ for $K_\mathrm{dr}=10^5$ & $s_\mathrm{min}$ for $K_\mathrm{dr}=10^8$ & $s_\mathrm{min}$ for $K_\mathrm{dr}=10^{11}$ \\
 \hline
 0.1 & 1.5 & 0.26 & 0.10 & 0.04 \\
 2   & 1.5 & 0.17 & 0.07 & 0.03 \\
 \hline
 0.1 & 2 & 0.08 & 0.02 & 0.004 \\
 2 & 2 & 0.04 & 0.009& 0.002 \\
 \hline
 0.1 & 2.5 & 0.03 & 0.004 & 0.0006 \\
 2 & 2.5 & 0.01 & 0.002 & 0.0003 \\
 \hline
\end{tabular}
\end{center}
\caption{\label{appendix:table:sminvalues} Minimal allowed saturation values for a set of realistic model parameters, assuming $\alpha=1$.}
\end{table}

\section{Useful results from literature}\label{appendix:literature-lemmas}

\begin{lemma}[Discrete Poincar\'e inequality~\cite{Eymard1999}]\label{appendix:lemma:discrete-poincare}
Let $\mathcal{T}$ be an admissible mesh, cf.\ Definition~\ref{definition:admissible-mesh}, and $u$ a piecewise constant function. Then there exists a constant $\COmegaDiscretePoincare\in (0,\mathrm{diam}(\Omega)]$ such that
\begin{align*}
 \| u \|_{L^2(\Omega)} \leq \COmegaDiscretePoincare \| u \|_{1,\mathcal{T}},
\end{align*}
where $\|\cdot\|_{1,\mathcal{T}}$ denotes the discrete $H^1_0(\Omega)$ norm, cf.\ Definition~\ref{definition:discrete-h1-norm}.
\end{lemma} 
 
\begin{lemma}[Discrete trace inequality~\cite{Eymard2000}]\label{appendix:lemma:discrete-trace}
Let $\mathcal{T}$ be an admissible mesh, cf.\ Definition~\ref{definition:admissible-mesh}, and $u$ a piecewise constant function. Let $\gamma(u)$ denote the trace of $u$, defined by $\gamma(u)=u_K$ on $\sigma\in\mathcal{E}_\mathrm{ext}\cap\mathcal{E}_K$, $K\in\mathcal{T}$. Then there exists a constant $\CDiscreteTrace>0$ such that
\begin{align*}
 \| \gamma(u) \|_{L^2(\partial\Omega)} \leq \CDiscreteTrace \left( \| u \|_{1,\mathcal{T}} + \| u \|_{L^2(\Omega)} \right),
\end{align*}
where $\|\cdot\|_{1,\mathcal{T}}$ denotes the discrete $H^1_0(\Omega)$ norm, cf.\ Definition~\ref{definition:discrete-h1-norm}.
\end{lemma}

\begin{lemma}[Stability of discrete gradients~\cite{Eymard1999}]\label{appendix:lemma:discrete-vs-continuous-gradients}
 Let $\mathcal{T}$ be an admissible mesh of some domain $\Omega$, cf.\ Definition~\ref{definition:admissible-mesh}, and $u\in H^1_0(\Omega)$. Define a piecewise constant function $\tilde{u}$ by
 \begin{align*}
  \tilde{u}(x) := \frac{1}{|K|} \int_K u(x) \, dx, \quad x \in K\in\mathcal{T}.
 \end{align*}
 Then there exists a constant $C>0$ (independent of $h$ for regular meshes) such that
 \begin{align*}
  \| \tilde{u} \|_{1,\mathcal{T}} \leq C \| u \|_{H^1(\Omega)}.
 \end{align*}
\end{lemma}

\begin{lemma}[Corollary of Brouwer's fixed point theorem~\cite{Ciarlet2013}] \label{appendix:lemma:brouwer}
 Let $\langle \cdot, \cdot \rangle$ denote the standard $\mathbb{R}^d$ scalar product and 
 let $\bm{F}:\mathbb{R}^d\rightarrow \mathbb{R}^d$ be a continuous function, satisfying
 \begin{align} \label{preliminaries-tools:brouwer-inequality}
  \langle \bm{F}(\bm{x}), \bm{x} \rangle \geq 0
 \end{align}
 for all $\bm{x}\in\mathbb{R}^d$ with $\langle \bm{x},\bm{x} \rangle \geq M$ for some fixed $M\in\mathbb{R}_+$.
 Then there exists a $\bm{x}^\star\in\mathbb{R}^d$ with $\llangle \bm{x}^\star, \bm{x}^\star \rrangle \leq M$ and $\bm{F}(\bm{x}^\star)=0$.
\end{lemma}

\begin{lemma}[Binomial identity]\label{appendix:lemma:binomial-identity}
 For $a,b\in\mathbb{R}$ it holds that
 \begin{align}
 \label{binomial-identity}
  a(a-b) = \frac{1}{2} \left( a^2 + (a - b)^2 - b^2 \right).
 \end{align}
\end{lemma}

\begin{lemma}[Summation by parts]\label{appendix:lemma:summation-by-parts}
 Given two sequences $(a_k)_{k\in\mathbb{N}_0},\,(b_k)_{k\in\mathbb{N}_0}\subset\mathbb{R}$, for all $N\in\mathbb{N}$ it holds that
 \begin{align*}
  \sum_{n=1}^N a_n (b_n - b_{n-1}) = a_Nb_N - a_1b_0 - \sum_{n=1}^{N-1} b_n (a_{n+1} - a_n).
 \end{align*}
\end{lemma}

\begin{lemma}[Discrete Gr\"onwall inequality~\cite{Clark1987}]\label{appendix:discrete-gronwall} 
 Let $(a_n)_n\subset\mathbb{R}_+$, $(\lambda_n)_n\subset\mathbb{R}_+$, $B\geq 0$. Assume for all $n\in\mathbb{N}$ it holds that 
 \begin{align*}
  a_n \leq B + \sum_{k=0}^{n-1} \lambda_k a_k.
 \end{align*}
 Then it follows
 \begin{align*}
  a_n \leq B \prod_{k=0}^{n-1}(1 + \lambda_k).
 \end{align*}
 In particular, if $\lambda_k = \frac{\lambda T}{N}$ for all $k\in\mathbb{N}$ for some $\lambda,T\in\mathbb{R}_+$ and $N\in\mathbb{N}$, it holds that
 \begin{align*}
  a_N \leq B\, \mathrm{exp}(\lambda T).
 \end{align*}
\end{lemma}

\begin{lemma}[Eberlein-$\check{\text{S}}$mulian theorem~\cite{Ciarlet2013}]\label{appendix:lemma:eberlein-smulian}
 Assume that $B$ is a reflexive Banach space and let $\{x_n\}_n\subset B$ be a bounded sequence in $B$. Then there exists a subsequence $\{x_{n_k}\}_k$ that converges weakly in $B$.
\end{lemma}

\begin{lemma}[Relaxed Aubin-Lions lemma~\cite{Simon1986}]\label{appendix:lemma:aubin-lions}
 Let $\{f_n\}_n \subset L^p(0,T;B)$, $1\leq p < \infty$, $B$ a Banach space. $\{f_n\}_n$ is relatively compact in $L^p(0,T;B)$ if the following two are fulfilled:
 
 \begin{itemize}
  \item $\{f_n\}_n$ is uniformly bounded in $L^p(0,T;X)$, for $X\subset B$ with compact embedding.
  
  \item $\int_\tau^T \| f_n(t) - f_n(t-\tau) \|_B^p \,dt \leq \mathcal{O}(\tau)$, as $\tau \rightarrow 0$.
 \end{itemize}
 For the second property it is sufficient that $\{ \partial_t f_n \}_n$ is uniformly bounded in $L^p(0,T;B)$.
\end{lemma}

\begin{lemma}[Riesz-Frechet-Kolmogorov compactness criterion~\cite{Brezis2010}]\label{appendix:lemma:kolmogorov}
Let $F$ be a bounded set in $L^p(\mathbb{R}^N)$ with $1\leq p < \infty$, $N\in\mathbb{N}$. Assume that
\begin{align*}
 \underset{|h|\rightarrow 0}{\mathrm{lim}}\, \left\| f(\cdot + h) - f(\cdot) \right\|_{L^p(\mathbb{R}^N)} = 0\quad\text{uniformly in }f\in F.
\end{align*}
Then the closure of $F|_\Omega:=\left\{ f : \Omega \rightarrow \mathbb{R} \, | \, f \in F \right\}$ is compact for any measurable set $\Omega\subset \mathbb{R}^N$ with finite measure. 
\end{lemma}

\begin{lemma}[Standard inf-sup argument~\cite{Boffi2013}]\label{appendix:lemma:inf-sup}
 Let $V$ and $Q$ be Hilbert spaces, and let $B$ be a linear continuous operator from $V$ to $Q'$. Denote by $B^t$ the transposed operator of $B$. Then, the following two statements are equivalent:
\begin{itemize}
\item $B^t$ is bounding, i.e., there exists a $\gamma > 0$ such that $\left\|B^tq\right\|_{V'}\geq\gamma\left\|q\right\|_Q$ for all $q \in Q$.
\item There exists a $L_B \in \mathcal{L}\left(Q', V\right) $ such that $B\left(L_B\left(\xi\right)\right)=\xi$  for all $\xi\in Q'$ with $\left\|L_b\right\|=\dfrac{1}{\gamma}=:\COmegaInfSup$.
\end{itemize}

\end{lemma}

\begin{lemma}[Properties of the Legendre transformation~\cite{Alt1983}]\label{appendix:lemma:legendre}
 Given $b:\mathbb{R}\rightarrow \mathbb{R}$ continuous and non-decreasing , we define its Legendre transformation
 \begin{align*}
  B(z) := \int_0^z (b(z) - b(s))\, ds \geq 0.
 \end{align*}
 It holds for all $x,y\in\mathbb{R}$ and for all $\delta >0$
 \begin{align*}
  0 &\leq B(x), \\
  B(x) - B(y) &\leq \left(b(x) - b(y)\right)x ,\\
  \left| b(x) \right| &\leq \delta \, B(x) + \underset{|y|\leq \delta^{-1}}{\mathrm{sup}}\, \left| b(y) \right|.
 \end{align*}
\end{lemma}

\section*{Acknowledgments}
The authors would like to thank Florin A.\ Radu and Willi J\"ager for their useful discussions.
 
\bibliographystyle{ieeetr} 

\begin{thebibliography}{100}

\bibitem{deBoer2000}
R.~De~Boer, {\em Theory of porous media: highlights in historical development
  and current state}.
\newblock Springer Science \& Business Media, 2000.

\bibitem{Terzaghi1936b}
K.~v. Terzaghi, ``The shearing resistance of saturated soils and the angle
  between the planes of shear,'' in {\em First international conference on soil
  mechanics, 1936}, vol.~1, pp.~54--59, 1936.

\bibitem{Biot1941}
M.~Biot, ``General theory of three-dimensional consolidation,'' {\em Journal of
  applied physics}, vol.~12, no.~2, pp.~155--164, 1941.

\bibitem{Lewis1998}
R.~Lewis and B.~Schrefler, {\em {The finite element method in the static and
  dynamic deformation and consolidation of porous media}}.
\newblock Numerical methods in engineering, John Wiley, 1998.

\bibitem{Coussy2004}
O.~Coussy, {\em Poromechanics}.
\newblock Wiley, 2004.

\bibitem{Szymkiewicz2012}
A.~Szymkiewicz, {\em Modelling water flow in unsaturated porous media:
  accounting for nonlinear permeability and material heterogeneity}.
\newblock Springer Science \& Business Media, 2012.

\bibitem{Nordbotten2011}
J.~M. Nordbotten and M.~A. Celia, {\em Geological storage of CO2: modeling
  approaches for large-scale simulation}.
\newblock John Wiley \& Sons, 2011.

\bibitem{Auriault1977}
J.-L. Auriault and E.~Sanchez-Palencia, ``Etude de comportment macroscopique
  d'un milieu poreux sature deformable,'' {\em Journal de M\'ecanique},
  vol.~16, pp.~575--603, 1977.

\bibitem{Zenisek1984}
A.~Zenisek, ``The existence and uniquencess theorem in {B}iot's consolidation
  theory,'' {\em Aplikace matematiky}, vol.~29, no.~3, pp.~194--211, 1984.

\bibitem{Showalter2000}
R.~Showalter, ``Diffusion in {P}oro-{E}lastic {M}edia,'' {\em Journal of
  Mathematical Analysis and Applications}, vol.~251, no.~1, pp.~310 -- 340,
  2000.

\bibitem{Ferronato2010}
M.~Ferronato, N.~Castelletto, and G.~Gaolati, ``A fully coupled 3-{D} mixed
  finite element model of {B}iot consolidation,'' {\em Journal of Computational
  Physics}, vol.~229, no.~12, pp.~4813 -- 4830, 2010.

\bibitem{Haga2012}
J.~B. Haga, H.~Osnes, and H.~P. Langtangen, ``On the causes of pressure
  oscillations in low-permeable and low-compressible porous media,'' {\em
  International Journal for Numerical and Analytical Methods in Geomechanics},
  vol.~36, no.~12, pp.~1507--1522, 2012.

\bibitem{Wheeler2014}
M.~Wheeler, G.~Xue, and I.~Yotov, ``Coupling multipoint flux mixed finite
  element methods with continuous {G}alerkin methods for poroelasticity,'' {\em
  Computational Geosciences}, vol.~18, no.~1, pp.~57--75, 2014.

\bibitem{Nordbotten2016}
J.~M. Nordbotten, ``Stable cell-centered finite volume discretization for
  {B}iot equations,'' {\em SIAM Journal on Numerical Analysis}, vol.~54, no.~2,
  pp.~942--968, 2016.

\bibitem{Rodrigo2016}
C.~Rodrigo, F.~Gaspar, X.~Hu, and L.~Zikatanov, ``Stability and monotonicity
  for some discretizations of the {B}iot's consolidation model,'' {\em Computer
  Methods in Applied Mechanics and Engineering}, vol.~298, pp.~183 -- 204,
  2016.

\bibitem{White2016}
J.~A. White, N.~Castelletto, and H.~A. Tchelepi, ``Block-partitioned solvers
  for coupled poromechanics: {A} unified framework,'' {\em Computer Methods in
  Applied Mechanics and Engineering}, vol.~303, pp.~55 -- 74, 2016.

\bibitem{Castelletto2017}
N.~Castelletto, H.~Hajibeygi, and H.~A. Tchelepi, ``Multiscale finite-element
  method for linear elastic geomechanics,'' {\em Journal of Computational
  Physics}, vol.~331, pp.~337 -- 356, 2017.

\bibitem{Lee2017}
J.~Lee, K.~Mardal, and R.~Winther, ``Parameter-{R}obust {D}iscretization and
  {P}reconditioning of {B}iot's {C}onsolidation {M}odel,'' {\em SIAM Journal on
  Scientific Computing}, vol.~39, no.~1, pp.~1--24, 2017.

\bibitem{Kim2011}
J.~Kim, H.~Tchelepi, and R.~Juanes, ``Stability and convergence of sequential
  methods for coupled flow and geomechanics: {F}ixed-stress and fixed-strain
  splits,'' {\em Computer Methods in Applied Mechanics and Engineering},
  vol.~200, no.~13, pp.~1591 -- 1606, 2011.

\bibitem{Mikelic2013}
A.~Mikeli{\'{c}} and M.~F. Wheeler, ``Convergence of iterative coupling for
  coupled flow and geomechanics,'' {\em Computational Geosciences}, vol.~17,
  no.~3, pp.~455--461, 2013.

\bibitem{Both2017}
J.~W. Both, M.~Borregales, J.~M. Nordbotten, K.~Kumar, and F.~A. Radu, ``Robust
  fixed stress splitting for {B}iot's equations in heterogeneous media,'' {\em
  Applied Mathematics Letters}, vol.~68, pp.~101 -- 108, 2017.

\bibitem{Gaspar2017}
F.~J. Gaspar and C.~Rodrigo, ``On the fixed-stress split scheme as smoother in
  multigrid methods for coupling flow and geomechanics,'' {\em Computer Methods
  in Applied Mechanics and Engineering}, vol.~326, pp.~526 -- 540, 2017.

\bibitem{Borregales2019}
M.~A. {Borregales}, K.~{Kumar}, J.~M. {Nordbotten}, and F.~A. {Radu},
  ``{Iterative solvers for {B}iot model under small and large deformation},''
  {\em arxiv e-prints}, 2019.
\newblock arXiv:1905.12996 [math.NA].

\bibitem{Storvik2019}
E.~Storvik, J.~W. Both, K.~Kumar, J.~M. Nordbotten, and F.~A. Radu, ``On the
  optimization of the fixed-stress splitting for biot's equations,'' {\em
  International Journal for Numerical Methods in Engineering}, vol.~120, no.~2,
  pp.~179--194, 2019.

\bibitem{Kumar2018}
K.~{Kumar}, S.~{Matculevich}, J.~{Nordbotten}, and S.~{Repin}, ``{Guaranteed
  and computable bounds of approximation errors for the semi-discrete {B}iot
  problem},'' {\em arxiv e-prints}, 2018.
\newblock arXiv:1808.08036 [math.NA].

\bibitem{Ahmed2019}
E.~Ahmed, F.~A. Radu, and J.~M. Nordbotten, ``Adaptive poromechanics
  computations based on a posteriori error estimates for fully mixed
  formulations of {B}iot's consolidation model,'' {\em Computer Methods in
  Applied Mechanics and Engineering}, vol.~347, pp.~264 -- 294, 2019.

\bibitem{Ahmed2020}
E.~Ahmed, J.~M. Nordbotten, and F.~A. Radu, ``Adaptive asynchronous
  time-stepping, stopping criteria, and a posteriori error estimates for
  fixed-stress iterative schemes for coupled poromechanics problems,'' {\em
  Journal of Computational and Applied Mathematics}, vol.~364, p.~112312, 2020.

\bibitem{Mikelic2012}
A.~Mikeli{\'c} and M.~F. Wheeler, ``Theory of the dynamic {B}iot-{A}llard
  equations and their link to the quasi-static {B}iot system,'' {\em Journal of
  Mathematical Physics}, vol.~53, no.~12, p.~123702, 2012.

\bibitem{Showalter2005}
R.~E. Showalter, ``Poroelastic filtration coupled to {S}tokes flow,'' {\em
  Lecture Notes in Pure and Applied Mathematics}, vol.~242, p.~229, 2005.

\bibitem{Ambartsumyan2018}
I.~{Ambartsumyan}, V.~J. {Ervin}, T.~{Nguyen}, and I.~{Yotov}, ``{A nonlinear
  {S}tokes-{B}iot model for the interaction of a non-{N}ewtonian fluid with
  poroelastic media},'' {\em arxiv e-prints}, 2018.
\newblock arXiv:1803.00947 [math.NA].

\bibitem{Ambartsumyan2018c}
I.~Ambartsumyan, E.~Khattatov, I.~Yotov, and P.~Zunino, ``A lagrange multiplier
  method for a {S}tokes--{B}iot fluid--poroelastic structure interaction
  model,'' {\em Numerische Mathematik}, vol.~140, no.~2, pp.~513--553, 2018.

\bibitem{Tavakoli2013}
A.~Tavakoli and M.~Ferronato, ``On existence-uniqueness of the solution in a
  nonlinear {B}iot's model,'' {\em Appl. Math}, vol.~7, no.~1, pp.~333--341,
  2013.

\bibitem{Bociu2016}
L.~Bociu, G.~Guidoboni, R.~Sacco, and J.~T. Webster, ``Analysis of nonlinear
  poro-elastic and poro-visco-elastic models,'' {\em Archive for Rational
  Mechanics and Analysis}, vol.~222, no.~3, pp.~1445--1519, 2016.

\bibitem{Mikelic2015}
A.~Mikeli{\'c}, M.~F. Wheeler, and T.~Wick, ``Phase-field modeling of a
  fluid-driven fracture in a poroelastic medium,'' {\em Computational
  Geosciences}, vol.~19, no.~6, pp.~1171--1195, 2015.

\bibitem{Girault2016}
V.~Girault, K.~Kumar, and M.~F. Wheeler, ``Convergence of iterative coupling of
  geomechanics with flow in a fractured poroelastic medium,'' {\em
  Computational Geosciences}, vol.~20, no.~5, pp.~997--1011, 2016.

\bibitem{Berge2019}
R.~L. {Berge}, I.~{Berre}, E.~{Keilegavlen}, J.~M. {Nordbotten}, and
  B.~{Wohlmuth}, ``{Finite volume discretization for poroelastic media with
  fractures modeled by contact mechanics},'' {\em arxiv e-prints}, 2019.
\newblock arXiv:1904.11916 [math.NA].

\bibitem{Ucar2018}
E.~Ucar, E.~Keilegavlen, I.~Berre, and J.~M. Nordbotten, ``A finite-volume
  discretization for deformation of fractured media,'' {\em Computational
  Geosciences}, vol.~22, no.~4, pp.~993--1007, 2018.

\bibitem{Both2019b}
J.~W. {Both}, K.~{Kumar}, J.~M. {Nordbotten}, and F.~A. {Radu}, ``{The gradient
  flow structures of thermo-poro-visco-elastic processes in porous media},''
  {\em arxiv e-prints}, 2019.
\newblock arXiv:1907.03134 [math.NA].

\bibitem{Borregales2018}
M.~Borregales, F.~A. Radu, K.~Kumar, and J.~M. Nordbotten, ``Robust iterative
  schemes for non-linear poromechanics,'' {\em Computational Geosciences},
  vol.~22, no.~4, pp.~1021--1038, 2018.

\bibitem{VanDuijn2019b}
C.~J. Van~Duijn and A.~Mikelic, ``{Mathematical Theory of Nonlinear
  Single-Phase Poroelasticity},'' 2019.

\bibitem{VanDujin2019}
C.~van Duijn, A.~Mikelić, M.~F. Wheeler, and T.~Wick, ``Thermoporoelasticity
  via homogenization: {M}odeling and formal two-scale expansions,'' {\em
  International Journal of Engineering Science}, vol.~138, pp.~1 -- 25, 2019.

\bibitem{Brun2019}
M.~K. Brun, E.~Ahmed, J.~M. Nordbotten, and F.~A. Radu, ``Well-posedness of the
  fully coupled quasi-static thermo-poroelastic equations with nonlinear
  convective transport,'' {\em Journal of Mathematical Analysis and
  Applications}, vol.~471, no.~1, pp.~239 -- 266, 2019.

\bibitem{Brun2019b}
M.~{Kirkes{\ae}ther Brun}, E.~{Ahmed}, I.~{Berre}, J.~M. {Nordbotten}, and
  F.~A. {Radu}, ``{Monolithic and splitting based solution schemes for fully
  coupled quasi-static thermo-poroelasticity with nonlinear convective
  transport},'' {\em arxiv e-prints}, 2019.
\newblock arXiv:1902.05783 [math.NA].

\bibitem{Kim2018}
J.~Kim, ``Unconditionally stable sequential schemes for all-way coupled
  thermoporomechanics: {U}ndrained-adiabatic and extended fixed-stress
  splits,'' {\em Computer Methods in Applied Mechanics and Engineering},
  vol.~341, pp.~93 -- 112, 2018.

\bibitem{Hong2019}
Q.~Hong, J.~Kraus, M.~Lymbery, and F.~Philo, ``Conservative discretizations and
  parameter-robust preconditioners for {B}iot and multiple-network flux-based
  poroelasticity models,'' {\em Numerical Linear Algebra with Applications},
  vol.~26, no.~4, p.~2242, 2019.

\bibitem{Hong2018}
Q.~{Hong}, J.~{Kraus}, M.~{Lymbery}, and M.~{Fanett Wheeler},
  ``{Parameter-robust convergence analysis of fixed-stress split iterative
  method for multiple-permeability poroelasticity systems},'' {\em arxiv
  e-prints}, 2018.
\newblock arXiv:1812.11809 [math.NA].

\bibitem{Lee2019}
J.~Lee, E.~Piersanti, K.~Mardal, and M.~Rognes, ``A {M}ixed {F}inite {E}lement
  {M}ethod for {N}early {I}ncompressible {M}ultiple-{N}etwork
  {P}oroelasticity,'' {\em SIAM Journal on Scientific Computing}, vol.~41,
  no.~2, pp.~722--747, 2019.

\bibitem{Showalter2001}
R.~Showalter and N.~Su, ``Partially saturated flow in a poroelastic medium,''
  {\em Discrete and Continuous Dynamical Systems - Series B}, vol.~1, no.~4,
  pp.~403--420, 2001.

\bibitem{Both2019}
J.~W. Both, K.~Kumar, J.~M. Nordbotten, and F.~A. Radu, ``Anderson accelerated
  fixed-stress splitting schemes for consolidation of unsaturated porous
  media,'' {\em Computers \& Mathematics with Applications}, vol.~77, no.~6,
  pp.~1479--1502, 2019.

\bibitem{Kim2013}
J.~Kim, H.~A. Tchelepi, and R.~Juanes, ``Rigorous {C}oupling of {G}eomechanics
  and {M}ultiphase {F}low with {S}trong {C}apillarity,'' {\em Society of
  Petroleum Engineers}, 2013.

\bibitem{Jha2014}
B.~Jha and R.~Juanes, ``Coupled multiphase flow and poromechanics: {A}
  computational model of pore pressure effects on fault slip and earthquake
  triggering,'' {\em Water Resources Research}, vol.~50, no.~5, pp.~3776--3808,
  2014.

\bibitem{Bui2019}
Q.~M. Bui, D.~Osei-Kuffuor, N.~Castelletto, and J.~A. White, ``A {S}calable
  {M}ultigrid {R}eduction {F}ramework for {M}ultiphase {P}oromechanics of
  {H}eterogeneous {M}edia,'' {\em arxiv e-prints}, 2019.
\newblock arXiv:1904.05960 [math.NA].

\bibitem{Alt1983}
H.~Wilhelm~Alt and S.~Luckhaus, ``Quasilinear elliptic-parabolic differential
  equations,'' {\em Mathematische Zeitschrift}, vol.~183, no.~3, pp.~311--341,
  1983.

\bibitem{Castelletto2019}
N.~Castelletto, S.~Klevtsov, H.~Hajibeygi, and H.~A. Tchelepi, ``Multiscale
  two-stage solver for {B}iot's poroelasticity equations in subsurface media,''
  {\em Computational Geosciences}, vol.~23, no.~2, pp.~207--224, 2019.

\bibitem{Eymard1999b}
R.~Eymard, M.~Gutnic, and D.~Hilhorst, ``The finite volume method for
  {R}ichards equation,'' {\em Computational Geosciences}, vol.~3, no.~3-4,
  pp.~259--294, 1999.

\bibitem{Klausen2008}
R.~A. Klausen, F.~A. Radu, and G.~T. Eigestad, ``Convergence of {MPFA} on
  triangulations and for {R}ichards' equation,'' {\em International Journal for
  Numerical Methods in Fluids}, vol.~58, no.~12, pp.~1327--1351, 2008.

\bibitem{Cances2016}
C.~Canc{\`e}s and C.~Guichard, ``Convergence of a nonlinear entropy diminishing
  control volume finite element scheme for solving anisotropic degenerate
  parabolic equations,'' {\em Mathematics of Computation}, vol.~85, no.~298,
  pp.~549--580, 2016.

\bibitem{Ait2018}
A.~Ait Hammou~Oulhaj, C.~Canc{\`e}s, and C.~Chainais-Hillairet, ``Numerical
  analysis of a nonlinearly stable and positive control volume finite element
  scheme for {R}ichards equation with anisotropy,'' {\em ESAIM: Mathematical
  Modelling \& Numerical Analysis}, vol.~52, no.~4, 2018.

\bibitem{Arbogast1996}
T.~Arbogast and M.~Wheeler, ``A {N}onlinear {M}ixed {F}inite {E}lement {M}ethod
  for a {D}egenerate {P}arabolic {E}quation {A}rising in {F}low in {P}orous
  {M}edia,'' {\em SIAM Journal on Numerical Analysis}, vol.~33, no.~4,
  pp.~1669--1687, 1996.

\bibitem{Radu2008}
F.~A. Radu, I.~S. Pop, and P.~Knabner, ``Error estimates for a mixed finite
  element discretization of some degenerate parabolic equations,'' {\em
  Numerische Mathematik}, vol.~109, no.~2, pp.~285--311, 2008.

\bibitem{Saad2013}
B.~Saad and M.~Saad, ``Study of full implicit petroleum engineering
  finite-volume scheme for compressible two-phase flow in porous media,'' {\em
  SIAM Journal on Numerical Analysis}, vol.~51, no.~1, pp.~716--741, 2013.

\bibitem{Eymard1999}
R.~Eymard, T.~Gallou{\"e}t, and R.~Herbin, ``Convergence of finite volume
  schemes for semilinear convection diffusion equations,'' {\em Numerische
  Mathematik}, vol.~82, no.~1, pp.~91--116, 1999.

\bibitem{vanGenuchten1980}
{M.Th. van Genuchten}, ``{A closed-form equation for predicting the hydraulic
  conductivity of unsaturated soils},'' {\em Soil Science Society of America
  Journal}, vol.~44(5), pp.~892--898, 1980.

\bibitem{Eymard2000}
R.~Eymard, T.~Gallouët, and R.~Herbin, ``Finite volume methods,'' vol.~7,
  pp.~713 -- 1018, 2000.

\bibitem{Baranger1996}
{Baranger, Jacques}, {Maitre, Jean-Fran\c{c}ois}, and {Oudin, Fabienne},
  ``Connection between finite volume and mixed finite element methods,'' {\em
  ESAIM: M2AN}, vol.~30, no.~4, pp.~445--465, 1996.


\bibitem{Boffi2013}
D.~Boffi, F.~Brezzi, and M.~Fortin, {\em {Mixed {F}inite {E}lement {M}ethods
  and {A}pplications}}.
\newblock Springer Series in Computational Mathematics, Springer, 2013.


\bibitem{Evans2010}
L.~Evans, {\em Partial {D}ifferential {E}quations}.
\newblock Graduate studies in mathematics, American Mathematical Society, 2010.

\bibitem{Mualem1976}
Y.~Mualem, ``A new model for predicting the hydraulic conductivity of
  unsaturated porous media,'' {\em Water Resources Research}, vol.~12, no.~3,
  pp.~513--522.

\bibitem{Ciarlet2013}
P.~G. Ciarlet, {\em Linear and {N}onlinear {F}unctional {A}nalysis with
  {A}pplications}.
\newblock Philadelphia, PA, USA: Society for Industrial and Applied
  Mathematics, 2013.

\bibitem{Clark1987}
D.~S. Clark, ``Short proof of a discrete gronwall inequality,'' {\em Discrete
  Applied Mathematics}, vol.~16, no.~3, pp.~279 -- 281, 1987.

\bibitem{Simon1986}
J.~Simon, ``Compact sets in the space {$L^p(0,T; B)$},'' {\em Annali di
  Matematica Pura ed Applicata}, vol.~146, no.~1, pp.~65--96, 1986.

\bibitem{Brezis2010}
H.~Brezis, {\em Functional analysis, {S}obolev spaces and partial differential
  equations}.
\newblock Springer Science \& Business Media, 2010.

\end{thebibliography}

\end{document}